\documentclass[twoside,a4paper,reqno,11pt]{amsart} 
\usepackage[top=28mm,right=28mm,bottom=28mm,left=28mm]{geometry}
\usepackage{microtype}
\usepackage{amsfonts, amsmath, amssymb, amsbsy, mathrsfs, bm, latexsym, stmaryrd, hyperref, array, enumitem, textcomp, url}
\usepackage[table]{xcolor}

\renewcommand{\a}{\alpha}
\renewcommand{\b}{\beta}

\newcommand{\e}{\epsilon}
 
\renewcommand{\l}{\lambda} \renewcommand{\O}{\Omega} 

 \renewcommand{\to}{\rightarrow}
 \newcommand{\s}{\sigma}

\newcommand{\G}{\bar{G}}

\newcommand{\la}{\langle}
\newcommand{\ra}{\rangle}

\newcommand{\leqs}{\leqslant}
\newcommand{\geqs}{\geqslant}

 \newcommand{\vs}{\vspace{3mm}}

\makeatletter
\newcommand{\imod}[1]{\allowbreak\mkern4mu({\operator@font mod}\,\,#1)}
\makeatother

\newtheorem{theorem}{Theorem} 
\newtheorem*{conj*}{Conjecture}

\newtheorem{corol}[theorem]{Corollary}

\newtheorem{thm}{Theorem}[section] 
\newtheorem{prop}[thm]{Proposition} 
\newtheorem{lem}[thm]{Lemma}
\newtheorem{cor}[thm]{Corollary}

\theoremstyle{definition}
\newtheorem{rem}[thm]{Remark}
\newtheorem{remk}{Remark}

\newtheorem{exa}[thm]{Example}

\begin{document}

\author{Timothy C. Burness}
 \address{T.C. Burness, School of Mathematics, University of Bristol, Bristol BS8 1UG, UK}
 \email{t.burness@bristol.ac.uk}
 
 \author{Adam R. Thomas}
 \address{A.R. Thomas, Mathematics Institute,
Zeeman Building, University of Warwick, Coventry CV4 7AL, UK}
 \email{adam.r.thomas@warwick.ac.uk}
 
\title{The classification of extremely primitive groups}
\dedicatory{\rm Dedicated to the memory of Jan Saxl}

\begin{abstract}
Let $G$ be a finite primitive permutation group on a set $\O$ with nontrivial point stabilizer $G_{\a}$. We say that $G$ is extremely primitive if $G_{\a}$ acts primitively on each of its orbits in $\O\setminus \{\a\}$. These groups arise naturally in several different contexts and their study can be traced back to work of Manning in the 1920s. In this paper, we determine the almost simple extremely primitive groups with socle an exceptional group of Lie type. By combining this result with earlier work of Burness, Praeger and Seress, this completes the classification of the almost simple extremely primitive  groups. Moreover, in view of results by Mann, Praeger and Seress, our main theorem gives a complete classification of all finite extremely primitive groups, up to finitely many affine exceptions (and it is conjectured that there are no exceptions). Along the way, we also establish several new results on base sizes for primitive actions of exceptional groups, which may be of independent interest. 
\end{abstract}

\date{\today}
\maketitle

\setcounter{tocdepth}{1}
\tableofcontents

\section{Introduction}\label{s:intro}

Let $G \leqs {\rm Sym}(\O)$ be a finite primitive permutation group with point stabilizer $H=G_{\a} \neq 1$. We say that $G$ is \emph{extremely primitive} if $H$ acts primitively on each of its orbits in $\O \setminus \{\a\}$ (this term was    coined by Mann, Praeger and Seress in \cite{MPS}). Equivalently, $G$ is extremely primitive if and only if $H \cap H^x$ is a maximal subgroup of $H$ for all $x \in G \setminus H$. For example, the natural action of $G = {\rm PGL}_{2}(q)$ on the projective line over $\mathbb{F}_q$ is extremely primitive (here $G$ is $2$-transitive and $H$ is a Borel subgroup). These groups arise naturally in several different contexts, including the construction of some of the sporadic simple groups (in particular ${\rm J}_{2}$ and ${\rm HS}$) and the study of permutation groups with restricted movement (see 
\cite{PP}). 

By a theorem of Manning \cite[Corollary I, p.821]{Manning} from 1927, if $G$ is extremely primitive then $H$ acts faithfully on each of its orbits in $\O \setminus \{\a\}$. In particular, $H \cap H^x$ is a core-free maximal subgroup of $H$ for all $x \in G \setminus H$. In turn, this implies that $H$ is itself a primitive permutation group on each of its nontrivial orbits and thus the O'Nan-Scott Theorem imposes strong restrictions on the structure of $H$. In particular, $H$ has at most two minimal normal subgroups, and the socle of $H$ (denoted by ${\rm soc}(H)$) is a direct product of isomorphic simple groups.

Recall that $G$ is \emph{almost simple} if $G_0 \leqs G \leqs {\rm Aut}(G_0)$ for some nonabelian simple group $G_0$, which is the unique minimal normal subgroup of $G$. Also recall that $G$ is \emph{affine} if $\O$ has the structure of a vector space over a prime field $\mathbb{F}_p$ and $G$ acts by affine transformations. A major step towards the classification of the extremely primitive groups is a theorem of Mann, Praeger and Seress \cite{MPS}, which states that each extremely primitive group is either almost simple or of affine type. Furthermore, they classify the affine examples up to the possibility of finitely many extremely primitive groups of the form $G = V{:}H$, where $V = \mathbb{F}_2^d$ and $H \leqs {\rm GL}(V)$ is irreducible and almost simple (as discussed below, it is conjectured that there are no additional examples). In later work \cite{BPS, BPS2}, Burness, Praeger and Seress determined all the extremely primitive almost simple groups with socle an alternating, classical or sporadic group. 

In this paper, we complete the picture by determining the extremely primitive almost simple groups with socle an exceptional group of Lie type. Note that in the statement of Theorem \ref{t:main}, we exclude the Ree groups with socle ${}^2G_2(3)' \cong {\rm L}_{2}(8)$.

\begin{theorem}\label{t:main}
Let $G \leqs {\rm Sym}(\O)$ be an almost simple group with stabilizer $H$ and socle an exceptional group of Lie type. Then $G$ is extremely primitive if and only if $(G,H) = (G_2(4),{\rm J}_{2})$ or $(G_2(4).2,{\rm J}_{2}.2)$.
\end{theorem}

By combining Theorem \ref{t:main} with the results in \cite{BPS, BPS2}, we can now complete the classification of the almost simple extremely primitive groups. 

\begin{theorem}\label{t:main2}
Let $G \leqs {\rm Sym}(\O)$ be an almost simple group with stabilizer $H$ and socle $G_0$. Then $G$ is extremely primitive if and only if $(G,H)$ is one of the cases in Table \ref{tab:main}.
\end{theorem}

\begin{table}
$$\begin{array}{llcl} \hline
G_0 & H \cap G_0 & \mbox{Rank} & \mbox{Conditions} \\ \hline
{\rm Alt}_n & ({\rm Sym}_{n/2} \wr {\rm Sym}_2) \cap G_0 & \frac{1}{4}(n+2) & n \equiv 2 \imod{4} \\
{\rm Alt}_n & {\rm Alt}_{n-1} & 2 & \mbox{$G={\rm Sym}_n$ or ${\rm Alt}_n$}  \\
{\rm Alt}_6 & {\rm L}_{2}(5) & 2 & \mbox{$G={\rm Sym}_6$ or ${\rm Alt}_6$} \\
{\rm Alt}_5 & D_{10} & 2 &   \\ 
{\rm L}_{2}(q) & P_1 & 2 & \\
{\rm L}_{2}(q) & D_{2(q+1)} & \frac{1}{2}q & \mbox{$G=G_0$, $q+1$ is a Fermat prime} \\
{\rm Sp}_{n}(2) &  {\rm O}_{n}^{\pm}(2) & 2 & n \geqs 6 \\
{\rm U}_{4}(3) & {\rm L}_{3}(4) & 3 & \mbox{See Remark \ref{r:main2}(v)} \\ 
{\rm L}_{3}(4) & {\rm Alt}_6 & 3 & \mbox{See Remark \ref{r:main2}(vi)} \\ 
{\rm L}_{2}(11) & {\rm Alt}_5 & 2 & G=G_0 \\
G_2(4) & {\rm J}_{2} & 3 & \\
{\rm M}_{11} & {\rm Sym}_6  & 2 & \\ 
{\rm M}_{11} & {\rm L}_{2}(11)  &  2 & \\
{\rm M}_{12} & {\rm M}_{11}  & 2 & G=G_0  \\
{\rm M}_{22} & {\rm L}_{3}(4)   & 2 & \\
{\rm M}_{23} & {\rm M}_{22}  & 2 & \\
{\rm M}_{24} & {\rm M}_{23}   & 2 & \\
{\rm J}_{2} & {\rm U}_{3}(3)  & 3 & \\
{\rm HS} & {\rm M}_{22}   &  3 & \\
{\rm HS} & {\rm U}_{3}(5)  & 2 & G=G_0  \\
{\rm Suz} & G_2(4)  & 3 & \\
{\rm McL} & {\rm U}_{4}(3)  & 3 & \\
{\rm Ru} & {}^2F_4(2)  & 3 & \\
{\rm Co}_{2} & {\rm U}_{6}(2)  & 3 & \\
{\rm Co}_{2} & {\rm McL}  & 6 & \\
{\rm Co}_{3} & {\rm McL}  & 2 & \\ \hline  
\end{array}$$
\caption{The extremely primitive almost simple groups}
\label{tab:main}
\end{table}

\begin{remk}\label{r:main2}
Let us make some comments on the statement of Theorem \ref{t:main2}.
\begin{itemize}\addtolength{\itemsep}{0.2\baselineskip}
\item[{\rm (i)}] In view of the isomorphisms
\[
{\rm L}_{2}(4) \cong {\rm L}_{2}(5) \cong {\rm Alt}_5,\;\; {\rm L}_{2}(9) \cong {\rm Alt}_6
\]
we assume $q \geqs 7$ and $q \ne 9$ in the fifth and sixth rows of Table \ref{tab:main} with $G_0 = {\rm L}_{2}(q)$. Similarly, we assume $G_0 \ne {\rm L}_{4}(2), {\rm L}_{3}(2), {}^2G_2(3)'$ since ${\rm L}_{4}(2) \cong {\rm Alt}_8$, ${\rm L}_{3}(2) \cong {\rm L}_{2}(7)$ and ${}^2G_2(3)' \cong {\rm L}_{2}(8)$.
\item[{\rm (ii)}] In the fifth row of the table,  $P_1$ denotes a Borel subgroup of $G$. 
\item[{\rm (iii)}] In the final column we describe the extremely primitive groups with socle $G_0$ and $H \cap G_0$ as given in the second column (if no conditions are recorded, then every almost simple group with socle $G_0$ is extremely primitive).
\item[{\rm (iv)}] In the third column we record the rank of $G$, which is the number of orbits of $H$ on $\O$ (so the almost simple $2$-primitive groups have rank $2$). In the special case $G = G_2(4).\a$ arising in Theorem \ref{t:main}, we have 
\[
|G:H| = 1 + 100 + 315
\]
and the stabilizers for the nontrivial orbits of $H$ are the maximal subgroups ${\rm U}_{3}(3).\a$ and $2^{1+4}{:}{\rm Alt}_5.\a$ of $H = {\rm J}_{2}.\a$.
\item[{\rm (v)}]  In the eighth row, $G_0 = {\rm U}_{4}(3)$ and $(G,H) = (G_0.2^2,{\rm L}_{3}(4).2^2)$ or $(G_0.2,{\rm L}_{3}(4).2)$. More precisely, $G = G_0.\la x,y\ra$ or $G_0.\la y \ra$, where $x$ is an involutory diagonal automorphism (class \texttt{2B} in the notation of \cite{ATLAS}) and $y$ is an involutory graph automorphism with centralizer of type ${\rm O}_{4}^{-}(3)$ (class  \texttt{2F}). 
\item[{\rm (vi)}] In the ninth row, $G_0 = {\rm L}_{3}(4)$ and $(G,H) = (G_0.2^2,{\rm Alt}_6.2^2)$ or 
$(G_0.2, {\rm Alt}_6.2)$, where in the latter case, $G = G_0.\la x \ra$ and $x$ is an involutory graph or graph-field automorphism (classes \texttt{2B} or \texttt{2D}).
\item[{\rm (vii)}] It is worth noting the following unrefinable chain of subgroups of Conway's sporadic group ${\rm Co}_{3}$:
\[
{\rm Co}_{3} > {\rm McL}.2 > {\rm U}_{4}(3).2 > {\rm L}_{3}(4).2 > {\rm PGL}_{2}(9) > 3^2{:}8
\]
For each inclusion $K>H$ in this chain, $K$ is extremely primitive on $K/H$.
\item[{\rm (viii)}] Let us also highlight the remarkable rank $6$ example that arises when  $G = {\rm Co}_{2}$ and $H = {\rm McL}$. Here 
\[
|G:H| = 47104 = 1+ 275+2025+7128+15400+22275
\]
and the respective stabilizers for the nontrivial orbits of $H$ are ${\rm U}_{4}(3)$,  ${\rm M}_{22}$, ${\rm U}_{3}(5)$, $3^4{:}{\rm M}_{10}$ and $2^4{:}{\rm Alt}_7$, each of which is a maximal subgroup of $H$. 
\end{itemize}
\end{remk}

Let $G\leqs {\rm Sym}(\O)$ be a finite primitive permutation group and let $e(G) \geqs 0$ be the largest integer $k$ with the property that for every $\ell$-set $\Delta \subseteq \O$ with $1 \leqs \ell \leqs k$, the pointwise stabilizer $G_{\Delta}$ is nontrivial and acts primitively on its orbits in $\Omega\setminus\Delta$. Note that $G$ is extremely primitive if and only if $e(G) \geqs 1$. Suppose $G = V{:}H$ is affine and $e(G) \geqs 3$, where $V = \mathbb{F}_p^d$. Then $H_{v_1,v_2,v_3}$ is a maximal subgroup of $H_{v_1,v_2}$ for all triples of distinct nonzero vectors in $V$. By setting $v_3 = \l v_1$ if $p>2$ and $v_3 = v_1+v_2$ if $p = 2$, where $1 \ne \l \in \mathbb{F}_p^{\times}$, we get $H_{v_1,v_2,v_3} = H_{v_1,v_2}$ and so   there are no affine groups with $e(G) \geqs 3$. By inspecting the almost simple groups in Table \ref{tab:main}, we obtain the following corollary.

\begin{corol}
Let $G$ be a finite primitive group of degree $n$ with $e(G) \geqs 3$. Then $G$ is $4$-transitive and one of the following holds:
\begin{itemize}\addtolength{\itemsep}{0.2\baselineskip} 
\item[{\rm (i)}] $G \cong {\rm Sym}_{n}$ and $e(G) = n-2$.
\item[{\rm (ii)}] $G \cong {\rm Alt}_{n}$, $n \geqs 6$ and $e(G) = n-3$.
\item[{\rm (iii)}] $(G,n) = ({\rm M}_{12},12)$ or $({\rm M}_{24}, 24)$ and $e(G)=3$.
\end{itemize}
\end{corol}

By combining Theorem \ref{t:main2} with the main results from \cite{MPS} on affine groups, we obtain the following theorem. Note that in part (ii)(a), a prime divisor $r$ of $p^d-1$ is a \emph{primitive prime divisor} if $r$ does not divide $p^i-1$ for all $i=1, \ldots, d-1$ (in other words, the order of $\mbox{$p$ mod $r$}$ is $d$). Also recall that a primitive group is \emph{simply primitive} if it is not $2$-transitive.

\begin{theorem}\label{t:main3}
Let $G \leqs {\rm Sym}(\O)$ be a finite extremely primitive group with point stabilizer $H$. Then either  
\begin{itemize}\addtolength{\itemsep}{0.2\baselineskip}
\item[{\rm (i)}] $G$ is almost simple and $(G,H)$ is one of the cases in Table \ref{tab:main}; or
\item[{\rm (ii)}] $G = V{:}H \leqs {\rm AGL}_{d}(p)$ is affine, $p$ is a prime and one of the following holds:

\vspace{1mm}

\begin{itemize}\addtolength{\itemsep}{0.2\baselineskip}
\item[{\rm (a)}] $H=Z_r.Z_e$, where $e$ divides $d$ and $r$ is a primitive prime divisor of $p^d-1$.
\item[{\rm (b)}] $p=2$ and $H = {\rm SL}_{d}(2)$ with $d \geqs 3$, or $H={\rm Sp}_{d}(2)$ with $d \geqs 4$. 
\item[{\rm (c)}] $p=2$ and $(d,H) = (4,{\rm Alt}_6)$, $(4, {\rm Alt}_7)$, $(6, {\rm U}_{3}(3))$ or $(6,{\rm U}_{3}(3).2)$.
\item[{\rm (d)}] $p=2$ and $(d,H)$ is one of the following:
\[
\begin{array}{llll}
(10,{\rm M}_{12}) & (10,{\rm M}_{22}) & (10,{\rm M}_{22}.2) & (11,{\rm M}_{23}) \\
(11,{\rm M}_{24}) & (22, {\rm Co}_{3}) &  (24,{\rm Co}_{1}) & (2k, {\rm Alt}_{2k+1}) \\
(2k, {\rm Sym}_{2k+1}) &  (2\ell, {\rm Alt}_{2\ell+1}) & (2\ell, {\rm Sym}_{2\ell+1}) & (2\ell, \O_{2\ell}^{\pm}(2)) \\
(2\ell, {\rm O}_{2\ell}^{\pm}(2)) &  (8, {\rm L}_{2}(17)) & (8, {\rm Sp}_{6}(2)) & 
\end{array}
\]
where $k \geqs 2$ and $\ell \geqs 3$.
\item[{\rm (e)}] $p=2$, $H$ is almost simple and $G$ is simply primitive.
\end{itemize}
\end{itemize}
Moreover, every group in parts (i) and (ii)(a,b,c,d) is extremely primitive.
\end{theorem}

As an immediate corollary we get the following result, which shows that if $G \leqs {\rm Sym}(\O)$ is extremely primitive, then in almost every case $G_{\a}$ acts as a primitive group of almost simple or affine type on each of its orbits in $\O \setminus \{\a\}$ (the exceptions arise in part (iii), where $G_{\a}$ acts as a product-type primitive group).

\begin{corol}\label{c:struc}
Let $G \leqs {\rm Sym}(\O)$ be a finite extremely primitive group with point stabilizer $H$. Then one of the following holds:
\begin{itemize}\addtolength{\itemsep}{0.2\baselineskip}
\item[{\rm (i)}] $H$ is almost simple.
\item[{\rm (ii)}] $H$ is solvable and either 

\vspace{1mm}

\begin{itemize}\addtolength{\itemsep}{0.2\baselineskip}
\item[{\rm (a)}] $G$ is almost simple with socle $G_0$ and $(G,H)$ is recorded in Table \ref{tab:sol}; or
\item[{\rm (b)}] $G \leqs {\rm AGL}_{d}(p)$ is affine and $H=Z_r.Z_e$, where $e$ divides $d$ and $r$ is a primitive prime divisor of $p^d-1$.
\end{itemize}
\item[{\rm (iii)}] $G = {\rm Sym}_n$ or ${\rm Alt}_{n}$, $n \equiv 2 \imod{4}$, $n \geqs 10$ and $H = ({\rm Sym}_{n/2} \wr {\rm Sym}_{2}) \cap G$.
\end{itemize}
\end{corol}

\begin{table}
$$\begin{array}{lll} \hline
G_0 & H \cap G_0 & \mbox{Conditions} \\ \hline
{\rm Alt}_6 & ({\rm Sym}_3 \wr {\rm Sym}_2) \cap G_0 & \\
{\rm Alt}_{5} & {\rm Alt}_4 & \\
{\rm Alt}_5 & D_{10} & \\ 
{\rm L}_{2}(q) & P_1 & \\
{\rm L}_{2}(q) & D_{2(q+1)} & \mbox{$G=G_0$, $q+1$ is a Fermat prime} \\
\hline
\end{array}$$
\caption{The extremely primitive almost simple groups with solvable point stabilizer}
\label{tab:sol}
\end{table}

By Theorem \ref{t:main3}, in order to complete the classification of extremely primitive groups, it remains to handle the affine groups $G = V{:}H$ arising in part (ii)(e), where $H \leqs {\rm GL}(V)$ is almost simple and $V = \mathbb{F}_2^d$. Let $\mathcal{M}$ be the set of maximal subgroups of $H$, and for $M \in \mathcal{M}$, let ${\rm fix}(M)$ be the space of vectors in $V$ fixed by $M$. By \cite[Lemma 4.1]{MPS}, we have
\[
\sum_{M \in \mathcal{M}} (|{\rm fix}(M)|-1) \leqs 2^d-1,
\]
with equality if and only if $G$ is extremely primitive. Since $\dim {\rm fix}(M) \leqs d/2$ for each $M \in \mathcal{M}$, it follows that $G$ is not extremely primitive if $|\mathcal{M}|<2^{d/2}$. In this way, upper bounds on $|\mathcal{M}|$, combined with lower bounds on the dimensions of irreducible modules for $H$, play an important role in the analysis. In particular, a theorem of Liebeck and Shalev \cite{LSh96} implies that $|\mathcal{M}|<|H|^{8/5}$ for all sufficiently large almost simple groups $H$ and this is a key ingredient in the proof that there are at most finitely many extremely primitive affine groups arising in part (ii)(e) of Theorem \ref{t:main3}. 

A well known conjecture of G.E. Wall from 1961 asserts that $|\mathcal{M}| < |H|$. Wall's conjecture was originally formulated for all finite groups, but counterexamples have recently been constructed, see \cite{Lub_wall}. However, the conjecture has been established for all sufficiently large alternating and symmetric groups (see \cite{LSh96}) and a theorem of Liebeck, Martin and Shalev \cite{LMS} implies that $|\mathcal{M}|< |H|^{1+o(1)}$ for all almost simple groups $H$ of Lie type. If one assumes Wall's conjecture for almost simple groups, then \cite[Theorem 4.8]{MPS} identifies a very short and explicit list of affine groups that can arise in part (ii)(e) of Theorem \ref{t:main3} (see \cite[Table 2]{MPS}). In each case, $H$ is an almost simple group of Lie type (and defined over the field $\mathbb{F}_2$, with just one exception) and it is conjectured in \cite{MPS} that none of these groups are extremely primitive. In other words, the list of extremely primitive groups in parts (i) and (ii)(a,b,c,d) of Theorem \ref{t:main3} is conjectured to be complete. See Remark \ref{r:main4} below for some additional comments on the cases in \cite[Table 2]{MPS}.

\vs

Let $G \leqs {\rm Sym}(\O)$ be a finite primitive permutation group with stabilizer $H = G_{\a}$. There are several different methods for showing that $G$ is not extremely primitive. As mentioned previously, Manning's result \cite{Manning} implies that if $G$ is extremely primitive, then $H$ acts faithfully and primitively on each of its nontrivial orbits and this imposes strong restrictions on the socle of $H$, via the O'Nan-Scott Theorem (see Lemma \ref{l:structure}, for example). If the socle is compatible with extreme primitivity, then it may be the case that the rank of $G$ and the indices of the maximal subgroups of $H$ are incompatible (for instance, see Lemma \ref{l:char}). In other situations, it may be possible to identify an explicit element $x \in G$ such that $H \cap H^x < H$ is non-maximal (see Lemma \ref{l:simple}, for example). 

Recall that a subset $B$ of $\O$ is a \emph{base} for $G$ if the pointwise stabilizer of $B$ in $G$ is trivial. The \emph{base size} of $G$, denoted $b(G,H)$, is then the minimal size of a base for $G$. If $b(G,H) = 2$, then this implies that there exists $x \in G$ such that $H \cap H^x = 1$, which is maximal in $H$ if and only if $H$ has prime order. Since no maximal subgroup of an almost simple group has prime order, the base-two property rules out extreme primitivity in this situation. This criterion, combined with a probabilistic approach for bounding the base size (see Lemma \ref{l:base}), provides a powerful technique for showing that a given group is not extremely primitive.

There is a substantial literature on bases for almost simple primitive groups, see \cite{Bur07,Bur18,BGS0,BLS,BOW}, for example. In particular, there has been significant interest in determining the primitive permutation groups with a base of size $2$, which is a far-reaching project initiated by Jan Saxl in the 1990s. This remains an open problem, although there has been a lot of  progress in recent years. In order to prove Theorem \ref{t:main}, we will establish several new base results for primitive groups with socle an  exceptional group of Lie type. These results make an important contribution to ongoing efforts to determine all the base-two almost simple groups and they  significantly strengthen some of the results presented in \cite{BLS}, where the general bound $b(G,H) \leqs 6$ is established. We anticipate that Propositions \ref{p:mr_base} and \ref{p:typeIbase}, as well as Theorem \ref{t:typeV}, will be of independent interest and applicable to other problems (see \cite{Bur20, BGL2} for some immediate applications). A systematic study of bases for almost simple exceptional groups will be the subject of a future paper.

\begin{remk}\label{r:main4}
The base-two problem has also been studied for primitive groups of affine type. Here $G = V{:}H$, where $H \leqs {\rm GL}(V)$ is irreducible, and in this setting we have $b(G,H) = 2$ if and only if $H$ has a regular orbit on the module $V$. In particular, determining if $G$ admits a base of size $2$ is a very natural problem in the representation theory of finite groups. For example, it plays an important role in the solution to the famous $k(GV)$ problem \cite{Schmid}, which in turn proves part of a conjecture of Brauer on defect groups of blocks \cite{Brauer}.

In recent work, Lee \cite{Lee1,Lee2,Lee3} has conducted an in-depth study of  base sizes for affine groups of the form $G = V{:}H$, where $H$ is an almost simple group of Lie type. In particular, as a corollary of her much more detailed results, she is able to eliminate some of the extremely primitive candidates in \cite[Table 2]{MPS}. More precisely, if $G = V{:}H$ with $V=\mathbb{F}_2^d$ then Lee proves that $b(G,H)=2$ in each of the following cases $(d,{\rm soc}(H))$ listed in \cite[Table 2]{MPS}:
\[
(40, {\rm PSp}_4(9)), \; (40, {\rm L}_{5}(2)), \; (48, \O_{8}^{\pm}(2)),\; (100, {\rm Sp}_{10}(2)), \; (126, {\rm L}_9(2)),
\]
together with the cases $d = \binom{k}{3}$ and $H={\rm L}_{k}(2)$ with $10 \leqs k \leqs  14$. For the remaining groups, either the precise base size is undetermined, or it is known to be at least $3$. 

By adopting a different approach, we prove in \cite{BTh_aff} that none of the groups recorded in \cite[Table 2]{MPS} are extremely primitive. In particular, this reduces the classification of the affine extremely primitive groups to Wall's conjecture for almost simple groups.  
\end{remk}

\noindent \textbf{Notation.} Let $G$ be a finite group and let $n$ be a positive integer. Our group theoretic notation is standard. In particular, we will write $Z_n$, or just $n$, for a cyclic group of order $n$ and $G^n$ will denote the direct product of $n$ copies of $G$. An unspecified extension of $G$ by a group $H$ will be denoted by $G.H$. If $X$ is a subset of $G$, then $i_n(X)$ is the number of elements of order $n$ in $X$. We adopt the standard notation for simple groups from \cite{KL}. The Fitting subgroup of $G$ will be denoted $F(G)$ and the socle of $G$ is ${\rm soc}(G)$. For positive integers $a$ and $b$, we write $(a,b)$ for the greatest common divisor of $a$ and $b$. Further notation will be introduced as and when needed in the main text.

\vs

\noindent \textbf{Organisation.} In Section \ref{s:prel} we record some preliminary results, which will be needed in the proof of Theorem \ref{t:main}. This includes a discussion of some general techniques for proving that a given primitive group is not extremely primitive. We also present several results on conjugacy classes in almost simple exceptional groups of Lie type, which will be applied repeatedly later in the paper. In Sections \ref{s:parab} and \ref{s:mr} we prove Theorem \ref{t:main} in the cases where a point stabilizer is a maximal parabolic or maximal rank subgroup of $G$, respectively. It is worth noting that the latter subgroups require considerably more effort and the proof of Theorem \ref{t:maxrank} spans almost $30$ pages. In Sections \ref{s:part1}, \ref{s:part2} and \ref{s:part3}, we complete the proof of Theorem \ref{t:main} for the groups with socle $E_8(q)$, $E_7(q)$, $E_6^{\e}(q)$, $F_4(q)$ or $G_2(q)$. Here we organise our analysis in accordance with a key theorem of Liebeck and Seitz on the subgroup structure of exceptional groups (Theorem \ref{t:types}), which partitions the remaining possibilities for $G_{\a}$ into several families. Finally, we complete the proof in Section \ref{s:part4}, where the remaining twisted groups are handled.

\vs

\noindent \textbf{Acknowledgements.} The authors thank David Craven, Martin Liebeck, Alastair Litterick and Jay Taylor for helpful conversations regarding the content of this paper. They also thank the Isaac Newton Institute for Mathematical Sciences for support
and hospitality during the programme \emph{Groups, Representations and Applications: New perspectives}, when some of the work on this paper was undertaken. This work was supported by: EPSRC grant number EP/R014604/1.

\vs\vs\vs

\section{Preliminaries}\label{s:prel}

\subsection{Extremely primitive groups}\label{ss:ep}

Let $G \leqs {\rm Sym}(\O)$ be a finite primitive permutation group with point stabilizer $H = G_{\a}$. In this section, we record three results which can be used to show that $G$ is not extremely primitive. 

By a theorem of Guralnick \cite[Theorem 3]{Gur2000}, $H$ acts faithfully on at least one of its orbits in $\O\setminus\{\a\}$. Moreover, if we assume $G$ is extremely primitive then Manning's theorem \cite[Corollary I, p.821]{Manning} implies that $H$ acts faithfully and primitively on \emph{all} of its orbits in $\O\setminus\{\a\}$. In particular, if $G$ is extremely primitive then we can view $H$ itself as a primitive permutation group and this allows us to apply the O'Nan-Scott Theorem to impose strong restrictions on the structure of $H$.

\begin{lem}\label{l:structure}
Suppose one of the following holds:
 \begin{itemize}\addtolength{\itemsep}{0.2\baselineskip}
 \item[{\rm (i)}] $Z(H)\ne 1$;
 \item[{\rm (ii)}] $F(H)$ is not elementary abelian;
 \item[{\rm (iii)}] $F(H) = Z_p^d$ is elementary abelian and $p^d$ does not divide $|\Omega|-1$;
 \item[{\rm (iv)}] $F(H) = Z_p^d$ is elementary abelian and $H/F(H)$ is not isomorphic to an irreducible subgroup of ${\rm GL}_d(p)$;
 \item[{\rm (v)}] ${\rm soc}(H)$ is not a direct product of isomorphic simple groups.
 \end{itemize}
Then $G$ is not extremely primitive.
\end{lem}

\begin{proof}
As noted above, if $G$ is extremely primitive then $H$ is a primitive permutation group on all of its orbits in $\Omega \setminus \{a\}$ and by applying the O'Nan-Scott Theorem we deduce that either
\begin{itemize}\addtolength{\itemsep}{0.2\baselineskip}
 \item[{\rm (a)}] $F(H)=1$ and ${\rm soc}(H)$ is a direct product of isomorphic nonabelian simple groups; or
 \item[{\rm (b)}] $H = F(H)K$ is an affine group, where $K \leqs {\rm GL}_{d}(p)$ is irreducible and ${\rm soc}(H) = F(H) = Z_p^d$ acts regularly on each $H$-orbit in $\O \setminus \{\a\}$. 
 \end{itemize}
The result follows (also see \cite[Lemma 2.2]{BPS}).
\end{proof}

\begin{rem}\label{r:neww}
Suppose $G$ is extremely primitive with $F(H) \ne 1$. Then as in case (b) in the proof of Lemma \ref{l:structure}, we have $H = F(H)K$ with ${\rm soc}(H) = F(H) = Z_p^d$ and $K \leqs {\rm GL}_{d}(p)$. Since $F(H)$ acts regularly on the $H$-orbits in $\O \setminus \{\a\}$, every nontrivial element in $F(H)$ has a unique fixed point on $\O$ and we deduce that $F(H) \cap K^g = 1$ for all $g \in G$. We thank an anonymous referee for making this observation.
\end{rem}

The next result records an elementary observation which will be useful when we can compute the rank $r$ of $G$ (that is, the number of orbits of $H$ on $\O$) and we know the indices of all the core-free maximal subgroups of $H$. For example, if the character tables of $G$ and $H$ are available and we can compute the fusion of $H$-classes in $G$ (using the \textsf{GAP} Character Table Library \cite{GAPCTL}, for example), then we can use the Orbit Counting Lemma to compute $r$. See the proofs of Lemmas \ref{l:maxrank_e6_3}, \ref{l:maxrank_f4_4}, \ref{t:e6typeI_1}, \ref{l:ndef2} and \ref{l:ndef20}, for example.

\begin{lem}\label{l:char}
Suppose $G$ has rank $r$ and let $\{M_1, \ldots, M_k\}$ be representatives of the $H$-classes of core-free maximal subgroups of $H$. Set $n_i = |H:M_i|$ for $i = 1, \ldots, k$. Suppose there is no $k$-tuple of non-negative integers $[a_1, \ldots, a_k]$ with $\sum_{i}a_i = r-1$ and $\sum_{i}a_in_i = |\O|-1$. Then $G$ is not extremely primitive.
\end{lem}

\begin{proof}
Since $G$ has rank $r$, there exist $\b_i \in \O$ such that  
\[
\O = \{\a\} \cup \b_1^H \cup \cdots \cup \b_{r-1}^H
\]
is a disjoint union of $H$-orbits. If $G$ is extremely primitive, then each stabilizer $H_{\b_j}$ is a maximal core-free subgroup of $H$, whence 
\[
|\O| = 1 + \sum_{j=1}^{r-1}|\b_j^H| = 1 + \sum_{j=1}^{r-1}|H:H_{\b_j}| = 1 + \sum_{i=1}^{k}a_in_i 
\]
for some $k$-tuple $[a_1, \ldots, a_k]$ of non-negative integers with $\sum_{i}a_i = r-1$. The result follows.
\end{proof}

The next lemma is  a key tool in the proof of Theorem \ref{t:main}. 

\begin{lem}\label{l:simple}
Let $G = G_0.A$ be an almost simple group with socle $G_0$ and let $H=H_0.A$ be a maximal subgroup of $G$ with $H_0 = H \cap G_0$. Let $K$ be a proper $A$-stable subgroup of $H_0$ and let $\mathcal{M}$ be the set of maximal overgroups of $K$ in $H_0$. Assume that each of the following conditions are satisfied:
\begin{itemize}\addtolength{\itemsep}{0.2\baselineskip}
\item[{\rm (i)}] $H_0$ is a maximal subgroup of $G_0$;
\item[{\rm (ii)}] Each $M \in \mathcal{M}$ is $A$-stable;
\item[{\rm (iii)}] There exists $g \in N_{G_0}(K)$ such that $M^g \not\leqs H_0$ for all $M \in \mathcal{M}$. 
\end{itemize}
Then $H \cap H^{g^{-1}}$ is a non-maximal subgroup of $H$ and thus the action of $G$ on $G/H$ is not extremely primitive.
\end{lem}

\begin{proof}
Let $L$ be a maximal subgroup of $H$ containing $K$. If $L$ contains $H_0$ then $L/H_0$ is maximal in $H/H_0 = A$ and thus $L = H_0.B$ with $B<A$ maximal. On the other hand, if $L$ does not contain $H_0$ then $H = H_0L$ and we deduce that $L = (H_0 \cap L).A$. Then since $K \leqs H_0 \cap L$, the maximality of $L$ in $H$ implies that $L = M.A$ for some $M \in \mathcal{M}$ (here we are using (ii)).

Let $g \in N_{G_0}(K)$ be an element satisfying the condition in (iii). Seeking a contradiction, suppose $H \cap H^{g^{-1}}$ is a maximal subgroup of $H$. From the above description of the maximal overgroups of $H$ containing $K$, it follows that either $(M.A)^g$ or $(H_0.B)^g$ must be contained in $H$ for some $M \in \mathcal{M}$ or maximal subgroup $B<A$. If $(M.A)^g < H$ then $M^g < H \cap G_0 = H_0$ and this  contradicts (iii). Similarly, if $(H_0.B)^g < H$ then $\langle H_0, H_0^g \rangle \leqs H$, but $\langle H_0, H_0^g \rangle = G_0$ since $H_0$ is a maximal subgroup of $G_0$ and $g \not \in H_0$ (by (i) and (iii)). In both cases we reach a contradiction and this completes the proof of the lemma.  
\end{proof}

The following example demonstrates how we will apply Lemma \ref{l:simple}.

\begin{exa}\label{e:ex}
Suppose $G = E_8(q)$ and $H = \O_{16}^{+}(q)$ with $q$ even, so $G = G_0$ and $H = H_0$. Set $K = \Omega^+_8(q)^2<H$ and observe that $M = N_{H}(K) = K.2^2$ is the unique maximal overgroup of $K$ in $H$. Since $L=K.({\rm Sym}_{3} \times 2)$ is a maximal subgroup of $G$ (see \cite[Table 5.1]{LSS}), it follows that $N_G(K) = L$. Now $M$ is not normal in $L$, so there exists $g \in N_G(K)$ such that $M^g \ne M$. Since $M$ is the unique maximal overgroup of $K$ in $H$, it follows that $M^g \not\leqs H$ and thus Lemma \ref{l:simple} implies that $H \cap H^{g^{-1}}$ is non-maximal in $H$. We refer the reader to Lemma \ref{l:maxrank_e8_4}, where the general case with $G_0 = E_8(q)$ and $H$ of type $D_8(q)$ is handled.
\end{exa}

\subsection{Base-two groups}\label{ss:base2}

Recall that the \emph{base size} of $G$, denoted by $b(G,H)$, is the smallest size of a subset $B \subseteq \Omega$ such that $\bigcap_{\a \in B}G_{\a} = 1$. In particular, $b(G,H)=2$ if and only if there exists an element $x \in G$ such that $H \cap H^x=1$.

Since no maximal subgroup of an almost simple group has prime order, we obtain the following result.

\begin{lem}\label{l:base2}
If $G$ is almost simple and $b(G,H)=2$, then $G$ is not extremely primitive.
\end{lem}

As discussed in Section \ref{s:intro}, there is an extensive literature on bases for almost simple primitive groups and there has been a special interest in determining the groups with $b(G,H)=2$. For the exceptional groups of Lie type, the main references are \cite{Bur18,BLS}. In particular, the main theorem of \cite{BLS} states that $b(G,H) \leqs 6$ if $G$ is any almost simple primitive group with socle of exceptional type. This is a key step in the proof of an influential conjecture of Cameron on bases for so-called \emph{non-standard} almost simple primitive groups (see \cite{BLS} and the references therein). 

Probabilistic methods play a key role in the proof of Cameron's base size conjecture. This approach arises from an elementary observation due to Liebeck and Shalev (see the proof of Theorem 1.3 in \cite{LSh99}). Fix a positive integer $c$ and notice that a $c$-tuple of points in $\O$ is \emph{not} a base for $G$ if and only if there exists an element $x \in G$ of prime order fixing each element in the tuple. Now the probability that a given element $x \in G$ fixes a uniformly random element in $\O$ is given by the expression
\[
{\rm fpr}(x,G/H) = \frac{|C_{\O}(x)|}{|\O|} = \frac{|x^G \cap H|}{|x^G|},
\]
which is the \emph{fixed point ratio} of $x$ (here $C_{\O}(x)$ is the set of fixed points of $x$ on $\O$). It follows easily that if $\mathcal{Q}(G,c)$ is the probability that a randomly chosen $c$-tuple in $\O$ is not a base then
\[
\mathcal{Q}(G,c) \leqs \sum_{x \in \mathcal{P}} {\rm fpr}(x,G/H)^c,
\]
where $\mathcal{P}$ is the set of elements of prime order in $G$. In particular, if this upper bound is less than $1$, then $b(G,H) \leqs c$ and thus upper bounds on fixed point ratios can be used to bound the base size. For exceptional groups of Lie type, which are the main focus of this paper, we refer the reader to \cite{LLS2} for a systematic study of fixed point ratios in this setting.

As a special case, we record the following lemma.

\begin{lem}\label{l:base}
Let $x_1, \ldots,x_k$ be representatives of the $G$-classes of elements of prime order in $H$ and set 
\begin{equation}\label{e:QGH}
\mathcal{Q}(G,H) = \sum_{i=1}^{k} |x_i^{G}| \cdot \left(\frac{|x_i^{G} \cap H|}{|x_i^{G}|}\right)^2.
\end{equation}
If $\mathcal{Q}(G,H)<1$ then $b(G,H)=2$.	
\end{lem}

The next result is \cite[Lemma 2.1]{Bur07}, which is a useful tool for bounding $\mathcal{Q}(G,H)$.

\begin{lem}\label{l:calc}
Suppose $x_{1}, \ldots, x_{m}$ represent distinct $G$-classes such that $\sum_{i}{|x_{i}^{G}\cap H|}\leqs A$ and $|x_{i}^{G}|\geqs B$ for all $i$. Then 
$$\sum_{i=1}^{m} |x_i^{G}| \cdot \left(\frac{|x_i^{G} \cap H|}{|x_i^{G}|}\right)^2 \leqs A^2/B.$$
\end{lem}

\begin{rem}\label{r:strategy}
We have now introduced several methods for showing that a given permutation group $G$ with point stabilizer $H = G_{\a}$ is not extremely primitive. Let us briefly summarise how we will apply these techniques in the proof of Theorem \ref{t:main}.

\begin{itemize}\addtolength{\itemsep}{0.2\baselineskip}
\item[{\rm (i)}] First we will seek to apply Lemma \ref{l:structure}, noting that if the structure of $H$ is incompatible, then it can be quickly eliminated. The remaining groups will then be partitioned into two collections, according to the order of $H$. If $|H|$ is small, which will typically mean $|H| \ll |G|^{1/2}$, then it is often possible to force $b(G,H) = 2$ via Lemma \ref{l:base} and a careful analysis of the conjugacy classes of elements of prime order in $H$ (in particular, we are interested in the corresponding $G$-classes of these elements). 

\item[{\rm (ii)}] For the remaining groups, it may be possible to apply Lemma \ref{l:char} in some special cases; typically, this will depend on whether or not we can access the relevant character tables in \cite{GAPCTL}. But in general, our main aim will be to identify a subgroup $K$ of $H_0$ with the desired properties in Lemma \ref{l:simple}. To do this, it will often be convenient to work first with the ambient algebraic groups and then descend to the relevant finite groups by taking the fixed points of a suitable Steinberg endomorphism.

\item[{\rm (iii)}] In order to handle some special cases where the underlying field $\mathbb{F}_q$ is small, we will sometimes use computational methods. See Section \ref{ss:comp} for further details.
\end{itemize}
\end{rem}

\subsection{Conjugacy classes}\label{ss:conj}

Let $G$ be an almost simple group with socle $G_0$, an exceptional group of Lie type over $\mathbb{F}_q$, where $q=p^f$ with $p$ prime. Write $G_0 = (\bar{G}_{\s})'$, where $\bar{G}$ is a simple algebraic group of adjoint type over the algebraic closure $\bar{\mathbb{F}}_p$ and $\s$ is an appropriate Steinberg endomorphism of $\bar{G}$. Recall that $\bar{G}_{\s} = {\rm Inndiag}(G_0)$ is the group of inner-diagonal automorphisms of $G_0$. 

In order to effectively apply the base-two criterion discussed in Section \ref{ss:base2}, we will need detailed information on the centralizers and conjugacy classes of elements of prime order in $G$. Here there is an extensive literature to draw upon and our primary sources will be \cite{LieS} for an in-depth treatment of unipotent classes and \cite{Lubeck} for information on semisimple classes. The centralizers of prime order graph, field and graph-field automorphisms of exceptional groups are described in \cite[Proposition 1.1]{LLS}. We refer the reader to \cite[Chapter 3]{BG} for a convenient source of information on conjugacy classes in the finite classical groups. 

\begin{rem}\label{r:aut}
The terminology we adopt for automorphisms in this paper is fairly  standard, although there are differences in the literature. In particular, we will refer to graph automorphisms of $F_4(q)$ and $G_2(q)$ when $p = 2$ and $3$, respectively. This is consistent with \cite{Carter}, for example, but not \cite{GLS}, where the term graph-field automorphism is preferred.
\end{rem}

The next result gives lower bounds on the sizes of conjugacy classes in $G$, according to the type of elements in the class. Note that in Table \ref{tab:cbds}, we set $\a=(q-1)/q$, $\b=(2,q-1)$, $\gamma=(3,q-1)$ and $\delta=(3,q+1)$.

\begin{table}
\begin{center}
\[
\begin{array}{ll}\hline
x_1 & \; \mbox{ Long root element} \\
x_2 & \; \mbox{ Unipotent element of order $p$, not a long root element} \\
& \; \mbox{ (nor a short root element if $(\bar{G},p) = (F_4,2)$ or $(G_2,3)$)} \\
x_3 & \; \mbox{ Semisimple involution} \\
x_4 & \; \mbox{ Semisimple element of odd prime order} \\
x_5 & \; \mbox{ Prime order element in $G \setminus {\rm Inndiag}(G_0)$} \\ \hline
\end{array}
\]
\caption{The elements $x_i \in G$ in Proposition \ref{p:bounds}}
\label{tab:elts}
\end{center}
\end{table}

\begin{prop}\label{p:bounds}
Let $G$ be an almost simple group with socle $G_0$, an exceptional group of Lie type over $\mathbb{F}_q$, where $q=p^f$ with $p$ prime. Let $x_i \in G$ be an element of prime order, as described in Table \ref{tab:elts}. Then $|x_i^G| > \ell_i$, where the $\ell_i$ are given in Table \ref{tab:cbds}.
\end{prop}

\renewcommand{\arraystretch}{1.2}
\begin{table}
\begin{center}
$$\begin{array}{llllll}\hline
G_0 & \ell_1 & \ell_2 & \ell_3 & \ell_4 & \ell_5 \\ \hline
E_8(q) & q^{58} & q^{92} & q^{112} & \a q^{114} & q^{124} \\
E_7(q) & q^{34} & q^{52} & \frac{\a}{2}q^{54} & \a q^{54} & \frac{1}{\b}q^{133/2} \\ 
E_6(q) & q^{22} & q^{32} & q^{32} & q^{32} & \frac{\a}{\gamma}q^{26}\\
{}^2E_6(q) & \a q^{22} & \a q^{32} & \a q^{32} & \a q^{32} & \frac{1}{\delta}q^{26} \\
F_4(q) & q^{16} & \frac{\a}{2}q^{22} & q^{16} & \a q^{30} & q^{26} \\ 
G_2(q) & \a q^6 & \a q^8 & q^8 & \a q^6 & q^7 \\ 
{}^3D_4(q) & \a q^{10} & q^{16} & q^{16} & \a q^{18} & q^{14} \\
{}^2F_4(q),\, q \geqs 2^3 & \a q^{11} & \a q^{14} & - & \a q^{18} & q^{52/3} \\
{}^2G_2(q),\, q \geqs 3^3 & \a q^4 & \frac{\a}{2}q^5 & \a q^4 & \frac{1}{2}q^6 & q^{14/3} \\ 
{}^2B_2(q) & \a q^3 & - & - & \frac{1}{2}q^4 & q^{10/3} \\ \hline
\end{array}$$
\caption{The lower bounds $|x_i^G| > \ell_i$ in Proposition \ref{p:bounds}}
\label{tab:cbds}
\end{center}
\end{table}
\renewcommand{\arraystretch}{1}

\begin{proof}
This is an entirely straightforward computation, using the available information on conjugacy classes and centralizers in \cite{LieS, Lubeck} and \cite[Proposition 1.1]{LLS}. 

For example, suppose $\bar{G}_{\s} = G_0 = F_4(q)$ and $q$ is odd, so 
\[
|G_0| = q^{24}(q^2-1)(q^6-1)(q^8-1)(q^{12}-1).
\]
If $x \in G_0$ is a long root element, then $x$ is contained in the unipotent class labelled $A_1$ in \cite[Table 22.2.4]{LieS} and we read off
\[
|x^G| = \frac{|G_0|}{q^{15}|{\rm Sp}_{6}(q)|} = (q^4+1)(q^{12}-1)>q^{16}=\ell_1.
\]
The next smallest unipotent class is labelled $\tilde{A}_1$ in $\bar{G}$ (these are the short root elements); since $p$ is odd, this $\bar{G}$-class splits into two $\bar{G}_{\s}$-classes and we have
\[
|x^G| = \frac{|G_0|}{2q^{15}|{\rm SL}_{4}^{\e}(q)|} = \frac{1}{2}q^{3}(q^3+\e)(q^4+1)(q^{12}-1) > \frac{1}{2}(q-1)q^{21} = \ell_2.
\]
Now assume $x \in G_0$ is a semisimple involution. There are two classes of involutions in $G_0$, with $C_{\bar{G}}(x) = A_1C_3$ or $B_4$, and we see that $|x^G|$ is minimal when $C_{\bar{G}}(x) = B_4$ (\cite[Table 4.5.1]{GLS} is an excellent source of information on semisimple involutions). Therefore, 
\[
|x^G| \geqs \frac{|G_0|}{|{\rm SO}_{9}(q)|} = q^8(q^8+q^4+1) > q^{16} = \ell_3.
\]
If $x \in G_0$ is a semisimple element of odd order then by inspecting \cite{Lubeck} we deduce that $|x^G|$ is minimal when $C_{\bar{G}}(x) = B_3T_1$ or $C_3T_1$ (here $T_1$ denotes a $1$-dimensional torus). This yields
\[
|x^G| \geqs \frac{|G_0|}{|{\rm SO}_{7}(q)|(q+1)} > (q-1)q^{29} = \ell_4.
\]
Finally, if $x \in G \setminus G_0$ has prime order $r$, then $x$ is a field automorphism, $q=q_0^r$ and $C_{G_0}(x) = F_4(q_0)$, so $|x^G|$ is minimal when $r=2$ and we get
\[
|x^G| \geqs \frac{|G_0|}{|F_4(q^{1/2})|} = q^{12}(q+1)(q^3+1)(q^4+1)(q^6+1)>q^{26}=\ell_5.
\]
The other groups are handled in a similar fashion and we omit the details.
\end{proof}

In almost every case, we see that $|x^G|$ is minimal when $x$ is a long root element (or a short root element when $(\bar{G},p) = (F_4,2)$ or $(G_2,3)$). Therefore, it will be important to know when certain maximal subgroups of $G$ contain such elements. With this in mind, we present the following result for algebraic groups, which is a simplified version of \cite[Proposition 1.13]{LLS}.

\begin{prop}\label{p:root}
Let $\bar{G}$ be a simple algebraic group, let $\bar{M}$ be a connected reductive subgroup of $\bar{G}$ and assume $u \in N_{\bar{G}}(\bar{M})$ is a long root element.
\begin{itemize}\addtolength{\itemsep}{0.2\baselineskip}
\item[{\rm (i)}] If $u \in \bar{M}$ and $\bar{M}$ is semisimple, then $u$ is a root element in one of the simple factors of $\bar{M}$.
\item[{\rm (ii)}] If $u \not\in \bar{M}$, then $p=2$ and $\bar{M} = \bar{X}\bar{Y}$ is a commuting product such that $u$ centralizes $\bar{X}$, and $\bar{Y}$ is either a simple factor of type $D_n$ or a $1$-dimensional torus.
\end{itemize}
\end{prop}

In the special case where $\bar{M}$ is a maximal torus of $\bar{G}$, we get the following corollary (this is based on an observation in the proof of \cite[Lemma 4.3]{LLS2}). 

\begin{cor}\label{c:root}
Let $\bar{G}$ be a simple algebraic group and let $\s$ be a Steinberg endomorphism of $\bar{G}$. Let $G = \bar{G}_{\s}$ and let $H = N_G(\bar{T}_{\s})$, where $\bar{T}$ is a $\s$-stable maximal torus of $\bar{G}$. If $x \in H$ is a root element in $G$, then $p=2$ and
\[
|x^G \cap H| \leqs |\Sigma^{+}(\bar{G})|\,|\bar{T}_{\s}|,
\]
where $\Sigma^{+}(\bar{G})$ is the set of positive roots in the root system of $\bar{G}$.
\end{cor}

\begin{proof}
By Proposition \ref{p:root}(ii), if $w \in N_{\bar{G}}(\bar{T})$ is a root element, then $p=2$ and $w$ centralizes a subtorus in $\bar{T}$ of codimension $1$. In particular, $w$ corresponds to a reflection in the Weyl group $W(\bar{G}) = N_{\bar{G}}(\bar{T})/\bar{T}$ and the result follows since there are precisely $|\Sigma^{+}(\bar{G})|$ reflections in $W(\bar{G})$.
\end{proof}

The following result is \cite[Proposition 1.3]{LLS2}.

\begin{prop}\label{p:i23}
Let $G_0 = (\bar{G}_{\s})'$ be a finite simple group of Lie type and let $N = |\Sigma^{+}(\bar{G})|$ be the number of positive roots in the root system of $\bar{G}$. Set 
\[
N_2 = \frac{1}{\a}(\dim \bar{G} - N), \;\; N_3 = \frac{1}{\a}(\dim \bar{G} - \frac{2}{3}N),
\]
where $\a=2$ if $G_0 = {}^2F_4(q)$, ${}^2G_2(q)$ or ${}^2B_2(q)$, otherwise $\a=1$. 
\begin{itemize}\addtolength{\itemsep}{0.2\baselineskip}
\item[{\rm (i)}] For $r \in \{2,3\}$, we have $i_r({\rm Aut}(G_0)) < 2(q+1)q^{N_r-1}$.
\item[{\rm (ii)}] The number of unipotent elements in $\bar{G}_{\s}$ is equal to $q^{2N/\a}$.
\end{itemize}
\end{prop}
 
\subsection{Computational methods}\label{ss:comp}

As previously remarked, it is feasible to use computational methods to 
handle certain groups defined over small fields and these computations can  be implemented in \textsc{Magma} \cite{Magma} or \textsf{GAP} \cite{GAP}. Here we briefly outline the main techniques we will use, referring the reader to \cite{BTh_comp} for a more detailed discussion, which includes the relevant code we used to obtain the results.

\subsubsection{Permutation representations}

In some cases, we can work with a suitable permutation representation of $G$ in {\sc Magma} and we can construct $H$ as a subgroup of $G$. Then by random search, we can seek an element $x \in G$ such that $H \cap H^x$ is non-maximal in $H$  (or $H \cap H^x = 1$ if we wish to show that $b(G,H) = 2$). Typically, we will use the {\sc Magma} command \texttt{ AutomorphismGroupSimpleGroup} to construct ${\rm Aut}(G_0)$ as a permutation group and we then identify $G$ as a subgroup of ${\rm Aut}(G_0)$. We can then construct $H$ via the \texttt{MaximalSubgroups} command, or by a direct construction if needed.
For example, we may have $H = N_G(C_G(x))$ for some element $x \in G$, which provides a way to construct $H$ directly.

\subsubsection{Character tables}

In order to effectively apply Lemma \ref{l:char}, we need to know the rank of $G$ (or a suitable bound on the rank). If the character tables of $G$ and $H$ are available in the \textsf{GAP} Character Table Library \cite{GAPCTL} then we may be able to use \textsf{GAP} \cite{GAP} to compute the fusion map from $H$-classes to $G$-classes via the command \texttt{FusionConjugacyClasses}. If this is possible, then we can calculate $|x^G \cap H|$ for each $x \in H$, which yields
\[
|C_{\O}(x)| = \frac{|x^G \cap H|}{|x^G|}\cdot |G:H|.
\]
(If the fusion map is not stored, then it may still be possible to proceed in the same way by using the command \texttt{PossibleClassFusions}.)
We can then calculate the rank $r$ of $G$ since
\[
r = \frac{1}{|H|}\sum_{x \in H}|C_{\O}(x)|
\]
by the Orbit Counting Lemma. Finally, it may be feasible to determine the indices of the maximal subgroups of $H$, for example by working with a suitable permutation representation and the {\sc Magma} command \texttt{MaximalSubgroups}. Typically, $H$ is almost simple and the relevant information on maximal subgroups may also be available in the Web Atlas \cite{WebAt}, for example.   

\subsubsection{Lie type computations}

In {\sc Magma}, there are sophisticated in-built functions for working with groups of Lie type in terms of their associated Lie structures, such as their root subgroups and Weyl group, etc. Therefore, if $H$ can be defined in terms of this data (for example, if $H = \bar{H}_{\s}$ and $\bar{H}$ is a $\s$-stable subsystem subgroup of $\bar{G}$), then it may be possible to construct $H$ as a subgroup of an appropriate group of Lie type via the \texttt{GroupOfLieType} command (in practice, it may be easier to construct $H$ as a subgroup of a larger group of Lie type which contains $G$). We can then 
obtain detailed information on the conjugacy classes of $H$ and the action of class representatives on certain modules for $G_0$, such as the adjoint module. In turn, this will allow us to estimate $|x^G \cap H|$ and $|x^G|$ for all elements $x \in H$ of prime order and we can use these estimates to derive an upper bound on the function $\mathcal{Q}(G,H)$ in \eqref{e:QGH}. For example, if $x \in H$ is unipotent, then we can often use the Jordan form of $x$ on the adjoint module for $G_0$ to determine the $\bar{G}$-class of $x$ via \cite{Lawunip}. If the bound we obtain gives $\mathcal{Q}(G,H)<1$, then $b(G,H) = 2$ by Lemma \ref{l:base} and $G$ is not extremely primitive. 

\subsubsection{Small groups}

To close this preliminary section, it is convenient to use computational methods to establish Theorem \ref{t:main} for some small groups. 

\begin{thm}\label{t:small}
Let $G$ be an almost simple primitive group with socle $G_0$ and point stabilizer $H$, where $G_0$ is one of
 \begin{equation}\label{e:small}
{}^2B_2(8),\, {}^2B_2(32),\, {}^2F_4(2)',\, {}^3D_4(2),\, G_2(2)',\, G_2(3),\, G_2(4),\, G_2(5).
\end{equation}
Then $G$ is extremely primitive if and only if $(G,H) = (G_2(4),{\rm J}_2)$ or $(G_2(4).2,{\rm J}_{2}.2)$.
\end{thm}

\begin{proof}
This is an entirely straightforward {\sc Magma} \cite{Magma} computation, working with a suitable permutation representation of $G$ (see \cite[Theorem 2.1]{BTh_comp} for further details).
\end{proof}

\begin{rem}\label{r:small}
Suppose $G$ is one of the groups in Theorem \ref{t:small} and let $b=b(G,H)$ denote the base size of $G$. If $G_0 \ne G_2(2)'$ then $b$ is recorded in \cite[Tables 11 and 12]{BLS}. For completeness, let us record that if $G_0 = G_2(2)' \cong{\rm U}_{3}(3)$ then $b \leqs 3$, with equality if and only if $H$ is one of
\[
3^{1+2}{:}8.\a, \, {\rm GU}_{2}(3).\a, 4^2{:}{\rm Sym}_{3}.\a,\, {\rm L}_{2}(7).\a
\]
where $\a = |G:G_0|$. In each case, this is an easy {\sc Magma} computation.
\end{rem}

\section{Parabolic subgroups}\label{s:parab}

Let $G$ be an almost simple primitive permutation group with socle $G_0$ and point stabilizer $H$, where $G_0$ is a simple exceptional group of Lie type over $\mathbb{F}_q$ and $q=p^f$ with $p$ a prime. In this section, we begin the proof of Theorem \ref{t:main} by handling the groups where $H$ is a maximal parabolic subgroup.

We begin with a preliminary lemma. The following result is presumably well known, but we include a proof because we were unable to find one in the literature. Note that we are not assuming $P$ is a maximal parabolic subgroup of $G_0$ (although the subgroups that arise in parts (i) and (ii) are maximal and they are labelled in the usual manner). 

\begin{lem}\label{l:abunirads}
Let $P$ be a parabolic subgroup of $G_0$. Then the unipotent radical of $P$ is abelian if and only if $(G_0,P)$ is one of the following: 
\begin{itemize}\addtolength{\itemsep}{0.2\baselineskip}
\item[{\rm (i)}] $G_0 = E_6(q)$ and $P = P_1$ or $P_6$.
\item[{\rm (ii)}] $G_0 = E_7(q)$ and $P = P_7$.
\end{itemize} 
\end{lem}

\begin{proof}
Let $Q$ be the unipotent radical of $P$. To begin with, let us assume $G_0$ is not one of the following groups:
\begin{equation}\label{e:excl}
F_4(q) \, (p=2), \; G_2(q)\, (p=2,3), \;  {}^2F_4(q), \; {}^2G_2(q),\; {}^2B_2(q).
\end{equation}

Let $r$ be the untwisted Lie rank of $G_0$ and fix a set of simple roots $\a_1, \ldots, \a_r$ for the corresponding root system, labelled as in \cite{Bou}. By \cite[Theorems 2 and 3]{ABS}, the nilpotency class of $Q$, which we denote by $c(Q)$, is independent of the field and it can be calculated as follows. First recall that the conjugacy classes of parabolic subgroups of $G_0$ are in bijective correspondence with the subsets of $[r]=\{1, \ldots, r\}$; under this correspondence, the maximal parabolic subgroups line up with subsets of size  $r-1$, while the Borel subgroups correspond to the empty set. Now, if $P$ corresponds to the subset $I = \{i_1 ,\ldots, i_k\}$, then \cite{ABS} gives $c(Q) = \sum_{i \in [r]\setminus I} c_i$, where $\alpha_0 = \sum_{i=1}^{r} c_i \alpha_i$ is the highest root for $G_0$. It is now a routine calculation with the root system of $G_0$ to show that the only cases with $c(Q) = 1$ are the maximal parabolic subgroups in the statement of the lemma. 

To complete the proof, let us consider the groups in \eqref{e:excl}. First assume $G_0 = F_4(q)$ (with $p=2$) or $G_2(q)$ (with $p=2,3$). These groups are called \emph{special} in \cite{ABS} due to the existence of degeneracies in Chevalley's commutator relations. However, a straightforward calculation with these relations shows that if $P$ is maximal then $Q$ is nonabelian. For example, if $G_0 = F_4(q)$ and $P = P_1$ (with $p=2$) then 
$u = x_{\alpha_1}(1)$ and $v = x_{\alpha_1+3\alpha_2+4\alpha_3 + 2\alpha_4}(1)$ are contained in $Q$ and we have $[u,v] = x_{\a_0}(1)$, 
where $X_{\a}  =\{ x_{\a}(c)\,:\, c \in \mathbb{F}_q\}$ is the root subgroup of $G_0$ corresponding to the root $\a$. The result now follows because the unipotent radical of any non-maximal parabolic subgroup contains the unipotent radical of a maximal parabolic subgroup.

Similarly, for $G_0 = {}^2F_4(q)$ it suffices to check that $Q$ is nonabelian when $P$ is maximal. Generators and relations for these subgroups are given in \cite[(2.2) and (2.3)]{Shin3} and the desired conclusion follows immediately. Finally, we note that the parabolic subgroups of ${}^2G_2(q)$ and ${}^2B_2(q)$ are Borel subgroups and so in these cases $Q$ is a Sylow $p$-subgroup of $G_0$. These Sylow subgroups are nonabelian by \cite[Main Theorem (2)]{Ward} and \cite[Section 13]{Suz}, respectively.
\end{proof}

\begin{thm}\label{t:parab}
If $H$ is a parabolic subgroup of $G$, then $G$ is not extremely primitive.
\end{thm}

\begin{proof}
Let $H_0 = QL$ be a Levi decomposition of $H_0 = H \cap G_0$, so $F(H) = Q$ is the Fitting subgroup of $H$. In view of Lemma \ref{l:structure}(ii), we may as well assume $Q$ is elementary abelian, in which case we can apply Lemma \ref{l:abunirads}. 

If $G_0 = E_6(q)$ and $H$ is a $P_1$ or $P_6$ parabolic then 
$$|F(H)| = q^{16},\;\; |\O| = |G_0:H_0| = (q^8+q^4+1)(q^6+q^3+1)(q^2+q+1),$$
and similarly if $G_0 = E_7(q)$ and $H$ is a $P_7$-parabolic then
$$|F(H)| = q^{27},\;\; |\O| = |G_0:H_0| = \frac{(q^{14}-1)(q^9+1)(q^5+1)}{q-1}.$$
In both cases, we see that $|F(H)|$ does not divide $|\O|-1$, so Lemma \ref{l:structure}(iii) implies that $G$ is not extremely primitive.

Alternatively, we can appeal to Remark \ref{r:neww}, noting that both $L$ and $F(H)=Q$ contain long root elements, which are conjugate in $G$.
\end{proof}

\section{Maximal rank subgroups}\label{s:mr}

Let $G_0 = (\bar{G}_{\s})'$ be a simple exceptional group of Lie type over $\mathbb{F}_{q}$, where $\bar{G}$ is a simple algebraic group of adjoint type and $\s$ is an appropriate Steinberg endomorphism of $\bar{G}$. In this section, we prove Theorem \ref{t:main} in the cases where $H = N_G(\bar{H}_{\s})$ and $\bar{H}$ is a $\s$-stable non-parabolic maximal rank subgroup of $\bar{G}$ (in particular, the connected component $\bar{H}^0$ contains a $\s$-stable maximal torus of $\bar{G}$). The possibilities for $H$ are determined in \cite{LSS} (see \cite[Tables 5.1 and 5.2]{LSS}) and there are two subcases to consider, according to whether or not $H$ is the normalizer of a maximal torus. Throughout this section, we will continue to exclude the groups in \eqref{e:small}.

Our main result is the following. 

\begin{thm}\label{t:maxrank}
If $H$ is a maximal rank subgroup of $G$, then $G$ is not extremely primitive.
\end{thm}

Recall that $b(G,H)$ denotes the \emph{base size} of $G$ and the condition $b(G,H)=2$ is equivalent to the existence of an element $x \in G$ with $H \cap H^x = 1$. In particular, if $b(G,H)=2$ then $G$ is not extremely primitive (see Lemma \ref{l:base2}). In proving  Theorem \ref{t:maxrank}, we will establish the following result on base-two groups, which may be of independent interest. In particular, this result significantly strengthens various bounds on $b(G,H)$ presented in \cite{BLS}. (Note that in the second column of Table \ref{tab:mr} we record the \emph{type} of $H$, which gives an approximate description of the structure of $H$.)

\begin{prop}\label{p:mr_base}
Let $G \leqs {\rm Sym}(\O)$ be an almost simple primitive group with point stabilizer $H$ and socle $G_0$, a simple exceptional group of Lie type over $\mathbb{F}_q$. Suppose 
\begin{itemize}\addtolength{\itemsep}{0.2\baselineskip}
\item[{\rm (i)}] $H$ is the normalizer of a maximal torus of $G$; or 
\item[{\rm (ii)}] $(G,H)$ is one of the maximal rank cases recorded in Table \ref{tab:mr}. 
\end{itemize}
Then $b(G,H) = 2$ and the probability that two randomly chosen points in $\O$ form a base for $G$ tends to $1$ as $q$ tends to infinity.
\end{prop}

\renewcommand{\arraystretch}{1.2}
\begin{table}
\begin{center}
$$\begin{array}{ll}\hline
G_0 & \mbox{Type of $H$} \\ \hline
E_8(q) & A_4^{\e}(q)^2, A_4^{-}(q^2), D_4(q)^2, D_4(q^2), {}^3D_4(q)^2, {}^3D_4(q^2) \\
& A_2^{\e}(q)^4, A_2^{-}(q^2)^2, A_2^{-}(q^4), A_1(q)^8 \\  
E_7(q) & A_1(q^7), A_1(q)^7 \\
E_6^{\e}(q) & A_2^{\e}(q)^3, A_2(q^2)A_2^{-\e}(q), A_2^{\e}(q^3) \\
F_4(q) & A_2^{\e}(q)^2, C_2(q)^2 \, (p=2), C_2(q^2) \, (p=2) \\
{}^2F_4(q) & A_2^{-}(q), {}^2B_2(q)^2, C_2(q) \\ \hline
\end{array}$$
\caption{Some maximal rank subgroups $H$ with $b(G,H)=2$}
\label{tab:mr}
\end{center}
\end{table}
\renewcommand{\arraystretch}{1}

In order to prove Proposition \ref{p:mr_base}, we will apply the probabilistic approach explained in Section \ref{ss:base2}. More precisely, in view of Lemma \ref{l:base}, we will aim to show that $\mathcal{Q}(G,H)<1$, where
\[
\mathcal{Q}(G,H) = \sum_{i=1}^{k}|x_i^G|\cdot \left(\frac{|x_i^G \cap H|}{|x_i^G|}\right)^2
\]
and $x_1, \ldots, x_k$ are representatives for the $G$-classes of elements of prime order. Of course, if $x_i^G$ does not meet $H$, then the contribution to $\mathcal{Q}(G,H)$ from $x_i^G$ is zero, so we are interested in the elements of prime order in $H$. To  estimate $\mathcal{Q}(G,H)$ effectively, we will apply Lemma \ref{l:calc}, using information on the conjugacy classes of elements of prime order in both $H$ and $G$. In particular, the lower bounds in Proposition \ref{p:bounds} will be applied repeatedly. In some cases, we will need more detailed information from \cite{LieS} (for unipotent elements) and \cite{Lubeck} (for semisimple elements). Finally, in order to establish the asymptotic statement in Proposition \ref{p:mr_base}, it suffices to show that $\mathcal{Q}(G,H)$ tends to $0$ as $q$ tends to infinity. In every case, we will derive an explicit upper bound on $\mathcal{Q}(G,H)$ as a function of $q$ and the desired asymptotic property will follow immediately.

\subsection{$G_0=E_8(q)$}\label{ss:mr_e8}

\begin{lem}\label{l:maxrank_e8_1}
If $G_0=E_8(q)$ and $H$ is the normalizer of a maximal torus, then $b(G,H)=2$.
\end{lem}

\begin{proof}
Let $W(\bar{G}) = 2.{\rm O}_{8}^{+}(2)$ be the Weyl group of $\bar{G} = E_8$ and note that the possibilities for $H$ are recorded in \cite[Table 5.2]{LSS}. If $x \in G$ is a long root element, then Proposition \ref{p:bounds} gives $|x^G| > q^{58} = b_1$ and Corollary \ref{c:root} implies that $|x^G \cap H| \leqs 120(q+1)^8 = a_1$. For all other nontrivial elements we have $|x^G|>q^{92} = b_2$ and we note that 
\[
|H| \leqs (q+1)^8|W(\bar{G})|.\log_2q = a_2.
\]
By applying Lemma \ref{l:calc}, we deduce that  
\begin{equation}\label{e:eq1}
\mathcal{Q}(G,H) < a_1^2/b_1 + a_2^2/b_2 < q^{-1}
\end{equation}
and the result follows via Lemma \ref{l:base}.
\end{proof}

\begin{lem}\label{l:maxrank_e8_2}
Suppose $G_0=E_8(q)$ and $H$ is of type 
\[
A_4^{\e}(q)^2, \; A_4^{-}(q^2),  \; A_2^{\e}(q)^4, \; A_2^{-}(q^2)^2, \; A_2^{-}(q^4), \; A_1(q)^8.
\]
Then $b(G,H)=2$.
\end{lem}

\begin{proof}
In each case, the precise structure of $H_0 = H \cap G_0$ is presented in \cite[Table 5.1]{LSS}. 

First assume $H$ is of type $A_1(q)^8$. If $x \in G$ is a long root element, then $|x^G|>q^{58} = b_1$ and Proposition \ref{p:root} implies that $|x^G \cap H| = 8(q^2-1)=a_1$. Otherwise, $|x^G|>q^{92} = b_2$ and
\[
|H| \leqs (q(q^2-1))^8|{\rm AGL}_{3}(2)|.\log_2q = a_2.
\]
These bounds imply that \eqref{e:eq1} holds and the result follows. An entirely similar argument applies when $H$ is of type $A_2^{\e}(q)^4$ and we omit the details.

Next assume $H$ is of type $A_2^{-}(q^2)^2$ or $A_2^{-}(q^4)$. By considering the structure of $H_0$ and its embedding in $\bar{H} = A_2^4.{\rm GL}_{2}(3)$, it follows via  Proposition \ref{p:root} that $H$ contains no long root elements of $G_0$. Therefore, $|x^G|>q^{92} = b_1$ for all nontrivial $x \in H$ and we observe that $|H|< 8\log_2q.q^{32} = a_1$. This gives $\mathcal{Q}(G,H) < a_1^2/b_1 < q^{-1}$.

Finally, let us assume $H$ is of type $A_4^{-}(q^2)$ or $A_4^{\e}(q)^2$, so $\bar{H}^0 = A_4^2$ and $\bar{H} = \bar{H}^0.4$. The total contribution to $\mathcal{Q}(G,H)$ from elements $x \in G$ with $|x^G|>q^{112}=b_1$ is less than $a_1^2/b_1$, where $a_1 = 4\log_2q.q^{48}$ is an upper bound on $|H|$. So it remains to consider the contribution from the elements with $|x^G| \leqs q^{112}$, which implies that $x$ is a unipotent element in one of the classes labelled $A_1$ and $A_1^2$ (see \cite{LieS,Lubeck}).

Let $V$ be the adjoint module for $\bar{G}$. By considering the composition factors of the restriction of $V$ to $\bar{H}^0$ (see \cite[Table 5]{Thomas}, for example), we calculate that if $p=2$ then each involution in $\bar{H} \setminus \bar{H}^0$ has Jordan form $(J_{2}^{120}, J_1^8)$ on $V$ and by  inspecting \cite[Table 9]{Lawunip} we deduce that they are contained in the class labelled $A_1^4$ in \cite[Table 22.2.1]{LieS}. In particular, the condition $|x^G| \leqs q^{112}$ implies that $x^G \cap H \subseteq \bar{H}^0$. 

Suppose $H$ is of type $A_4^{-}(q^2)$. By considering the embedding of $H_0$ in $G_0$, we deduce that $H$ does not contain any long root elements of $G_0$, so we may assume $x$ is in the $A_1^2$ class and thus $|x^G|>q^{92}=b_2$. Here $|x^G \cap H|$ coincides with the number of long root elements in ${\rm U}_{5}(q^2)$ (we noted above that $\bar{H}\setminus \bar{H}^0$ contains no elements in the $A_1^2$ class), so 
\[
|x^G \cap H|= (q^2-1)(q^4+1)(q^{10}+1) < 4q^{16}= a_2
\]
and we conclude that \eqref{e:eq1} holds.

Finally, suppose $H$ is of type $A_4^{\e}(q)^2$. By applying Proposition \ref{p:root}, we deduce that if $x \in G$ is a long root element, then 
\[
|x^G \cap H| \leqs 2(q+1)(q^2+1)(q^5-1) = a_3
\] 
and $|x^G|>q^{58} = b_3$. Now assume $x \in G$ is a unipotent element in the $A_1^2$ class. Suppose $y \in {\rm L}_{5}^{\e}(q)$ has Jordan form $(J_2^2,J_1)$ on the natural module and let $z \in {\rm L}_{5}^{\e}(q)$ be a long root element (so $z$ has Jordan form $(J_2,J_1^3)$). Then by appealing to \cite[Section 4.17]{Law09}, we calculate that
\[
|x^G \cap H| = 2|y^{{\rm L}_{5}^{\e}(q)}| + |z^{{\rm L}_{5}^{\e}(q)}|^2  = 2\left(\frac{|{\rm SL}_{5}^{\e}(q)|}{q^8|{\rm GL}_{2}^{\e}(q)|}\right) + \left(\frac{|{\rm SL}_{5}^{\e}(q)|}{q^{7}|{\rm GL}_{3}^{\e}(q)|} \right)^2< a_2
\]
and thus 
\[
\mathcal{Q}(G,H) < a_1^2/b_1 + a_2^2/b_2 + a_3^2/b_3 < q^{-1}
\]
as required. 
\end{proof}

\begin{lem}\label{l:maxrank_e8_3}
Suppose $G_0=E_8(q)$ and $H$ is of type 
\[
D_4(q)^2, \; D_4(q^2), \; {}^3D_4(q)^2, \; {}^3D_4(q^2).
\]
Then $b(G,H)=2$.
\end{lem}

\begin{proof}
Here $\bar{H} = \bar{H}^0.(2 \times {\rm Sym}_3)$ and $\bar{H}^0 = D_4^2$. The precise structure of $H_0$ is presented in \cite[Table 5.1]{LSS} and in each case, one checks that  
\[
|H|< 12q^{56}\log_2q = a_1.
\] 
Therefore, the contribution to $\mathcal{Q}(G,H)$ from elements $x \in G$ with $|x^G|> \frac{1}{2}q^{124} = b_1$ is less than $a_1^2/b_1$. For the remainder, let $x \in H$ be an element of prime order $r$ with $|x^G| \leqs b_1$. Note that the bound on $|x^G|$ implies that $x \in G_0$ (see Proposition \ref{p:bounds}).

Suppose $x \in H_0$ is a long root element of $G_0$, so $|x^G|>q^{58} = b_2$. By considering Proposition \ref{p:root} and the embedding of $H_0$, we see that $H$ is of type $D_4(q)^2$ or ${}^3D_4(q)^2$. Since there are fewer than $q^{10}$ long root elements in both ${\rm P\O}_{8}^{+}(q)$ and ${}^3D_4(q)$, it follows that $|x^G \cap H|< 2q^{10} = a_2$. For the remainder, we may assume $q^{92} < |x^G| \leqs b_1$. 

Next suppose $r \in \{2,3\}$. By inspecting each possibility for $H_0$ in turn and applying \cite[Proposition 1.3]{LLS2}, we deduce that
\[
i_r(H_0) \leqs 4(q+1)^2q^{38} = a_3.
\]
For example, suppose $H$ is of type $D_4(q)^2$, so  \cite[Table 5.1]{LSS} gives
\[
H_0 = d^2.{\rm P\O}_{8}^{+}(q)^2.d^2.({\rm Sym}_3 \times 2)
\]
with $d = (2,q-1)$. Let $Z$ be the normal subgroup of order $d^2$. Then 
\[
i_3(H_0)  = i_3(H_0/Z) \leqs i_3({\rm Aut}({\rm P\O}_{8}^{+}(q)) \wr {\rm Sym}_2)
\]
and by applying Proposition \ref{p:i23}(i) we deduce that
\[
i_3(H_0) \leqs (1+i_3({\rm Aut}({\rm P\O}_{8}^{+}(q))))^2 \leqs (2(q+1)q^{19})^2 = 4(q+1)^2q^{38}.
\]
Similarly, 
\begin{align*}
i_2(H_0) & \leqs |Z|\cdot (1+i_2({\rm Aut}({\rm P\O}_{8}^{+}(q)) \wr {\rm Sym}_2)) \\
& \leqs |Z| \cdot (1+(1+i_2({\rm Aut}({\rm P\O}_{8}^{+}(q))))^2 + |{\rm Aut}({\rm P\O}_{8}^{+}(q))|) \\
& < 4\left(4(q+1)^2q^{30}+6\log_2q.q^{28}\right)
\end{align*}
and the desired bound follows. The other cases are very similar. We conclude that the combined contribution to $\mathcal{Q}(G,H)$ from elements $x \in G$ of order $2$ and $3$ with $|x^G|>q^{92} = b_3$ is less than $2a_3^2/b_3$. 

Now assume $r \geqs 5$, so $i_r(H_0) = i_r(L)$, where $L = {\rm P\O}_{8}^{+}(q)^2$, ${\rm P\O}_{8}^{+}(q^2)$, ${}^3D_4(q)^2$ and ${}^3D_4(q^2)$ in each of the respective cases. If $r \ne p$ then $|x^G|>\frac{1}{2}q^{114}=b_4$ and $|L|<q^{56} = a_4$. Now assume $r=p$. If $x$ is not in the class $A_1^2$, then $|x^G| > q^{112} = b_5$ and we note that $L$ contains fewer than $q^{48}=a_5$ elements of order $p$ (see Proposition \ref{p:i23}(ii)). 

Finally, suppose $x$ is a unipotent element in the $A_1^2$ class, so $|x^G|>q^{92} = b_6$. Let $V$ be the adjoint module for $\bar{G}$ and note that 
$x$ has Jordan form $(J_3^{14},J_2^{64},J_1^{78})$ on $V$ (see \cite[Table 9]{Lawunip}). By considering the restriction of $V$ to $\bar{H}^0 = D_4^2$ (see \cite[Table 5]{Thomas}), we deduce that $x \in \bar{H}^0$ is of the form $(u,u')$, $(v,1)$ or $(1,v)$, where $u,u' \in D_4$ are long root elements and $v \in D_4$ is in the class labelled $A_1^2$ (that is, $v$ has Jordan form $(J_3,J_1^5)$ on the natural module for $D_4$). See \cite[p.2327]{BGS} for further details.

If $L = {\rm P\O}_{8}^{+}(q^2)$ or ${}^3D_4(q^2)$, it follows that $|x^G \cap H| = |z^L|$, where $z \in L$ is a long root element, and we deduce that $|x^G \cap H| < 2q^{20}=a_6$. Similarly, if $L = {}^3D_4(q)^2$ then $|x^G \cap H| = |z^{{}^3D_4(q)}|^2$, where $z \in {}^3D_4(q)$ is a long root element, and the same bound holds. Finally, suppose $L = {\rm P\O}_{8}^{+}(q)^2$. Here
\[
|x^G \cap H| = 2|y^{{\rm P\O}_{8}^{+}(q)}| + |z^{{\rm P\O}_{8}^{+}(q)}|^2 < a_6
\]
where $y \in {\rm P\O}_{8}^{+}(q)$ is a unipotent element in the $A_1^2$ class of $D_4$ and $z \in {\rm P\O}_{8}^{+}(q)$ is a long root element.

By bringing the above estimates together, we conclude that
\[
\mathcal{Q}(G,H) < \sum_{i=1}^{6} a_i^2/b_i + a_3^2/b_3 < q^{-1}
\]
for all $q \geqs 3$. In addition, the given bound is less than $1$ when $q=2$.
\end{proof}

\begin{lem}\label{l:maxrank_e8_4}
Suppose $G_0=E_8(q)$ and $H$ is of type 
\[
D_8(q), \; A_1(q)E_7(q), \; A_8^{\e}(q), \; A_2^{\e}(q)E_6^{\e}(q). 
\]
Then $G$ is not extremely primitive. 
\end{lem}

\begin{proof}
Write $G = G_0.A$, where $A$ is a group of field automorphisms of $G_0$. Set $d=(2,q-1)$.

First assume $H$ is of type $A_1(q)E_7(q)$, so $H_0 = d.({\rm L}_{2}(q) \times E_7(q)).d$ (see \cite[Table 5.1]{LSS}). If $q$ is odd then $H$ is the centralizer of an involution, so $Z(H) \ne 1$ and $G$ is not extremely primitive by Lemma \ref{l:structure}(i). On the other hand, if $q$ is even then ${\rm soc}(H)$ is not a direct product of isomorphic simple groups (indeed, either $q=2$ and ${\rm soc}(H) = 3 \times E_7(2)$, or $q \geqs 4$ and ${\rm soc}(H) = {\rm L}_{2}(q) \times E_7(q)$), so extreme primitivity is ruled out by Lemma \ref{l:structure}(v). A very similar argument shows that $G$ is not extremely primitive if $H$ is of type $A_2^{\e}(q)E_6^{\e}(q)$.

Next assume $H$ is of type $D_8(q)$, so $H_0 = d.{\rm P\O}_{16}^{+}(q).d$. In view of Lemma \ref{l:structure}(i), we may assume $q$ is even, so $H_0 = \O_{16}^{+}(q)$ and $H = H_0.A$. Let $K = \Omega^+_8(q)^2<H_0$ and observe that $M = N_{H_0}(K) = K.2^2$ is the unique maximal overgroup of $K$ in $H_0$. In addition, we note that  
$N_{G_0}(K) = K.({\rm Sym}_{3} \times 2)$, which is a maximal subgroup of $G_0$ (see \cite[Table 5.1]{LSS}). Clearly, $M$ is not a normal subgroup of $N_{G_0}(K)$ and so there exists $g \in N_{G_0}(K)$ which does not normalize $M$. Finally, since $M$ is the unique maximal overgroup of $K$ in $H_0$, it follows that $M^g \not\leqs H_0$ and we conclude by applying Lemma \ref{l:simple}, noting that $K$ and $M$ are both $A$-stable.

Finally, let us assume $H$ is of type $A_8^{\e}(q)$, so $H_0 = h.\text{L}^\epsilon_9(q).e.2$, where $e=(3,q-\e)$ and $h = (9,q-\e)/e$. In view of Lemma \ref{l:structure}, we may assume that $h=1$.

Fix a set of simple roots $\a_1, \ldots, \a_8$ for $\bar{G}$, labelled in the usual way (see \cite{Bou}), and let $X_{\a}$ be the root subgroup of $\bar{G}$ corresponding to the root $\a$. We may assume that $\bar{H} = \bar{H}^0.2$, where  
\[
\bar{H}^0 = \la X_{\pm \a_1}, X_{\pm \a_3}, X_{\pm \a_4}, X_{\pm \a_5}, X_{\pm \a_6}, X_{\pm \a_7}, X_{\pm \a_8}, X_{\pm \a_0} \ra
\]
is of type $A_8$ and $\a_0$ is the highest root. Let $\bar{P}=\bar{Q}\bar{L}$ be the parabolic subgroup of $\bar{H}^0$ corresponding to the simple roots $\a_4$ and $\a_7$ with Levi factor $\bar{L} = A_2^3T_2$. Note that $\bar{L}$ is contained in a maximal rank subgroup $\bar{J}$ of $\bar{H}^0$ of type $A_2^4$ (indeed, as noted in \cite[Table~2]{LS94}, we have $C_{\bar{G}}(A_2)^0 = E_6$ and $C_{E_6}(A_2)^0 = A_2^2$). The subgroup $Z(\bar{L})^0 = T_2$ centralizes $\bar{K}$ and is therefore a maximal torus in the fourth $A_2$ factor of $\bar{J}$, which we denote by $\bar{M}$. So $C_{\bar{G}}(\bar{K})^0 = \bar{M}$ and $N_{\bar{M}}(Z(\bar{L})^0) = T_2.{\rm Sym}_3$. A straightforward calculation in the Weyl group of $\bar{H}^0$ shows that $N_{\bar{H}^0}(\bar{L}) = \bar{L}.{\rm Sym}_3$, where ${\rm Sym}_3$ acts naturally on the factors of $\bar{K} = \bar{L}' = A_2^3$ and it acts on $Z(\bar{L})^0$ in the same way as the Weyl group of $\bar{M}$. By working in the Weyl group of $\bar{G}$ we see that $N_{\bar{G}}(\bar{L}) = \bar{L}.({\rm Sym}_3 \times {\rm Sym}_3 \times 2)$, where the first ${\rm Sym}_3$ factor naturally permutes the $A_2$ factors, the second acts on $Z(\bar{L})^0$ as before and the central involution acts as a graph automorphism on each $A_2$ factor and inverts $Z(\bar{L})^0$. In particular, $N_{\bar{H}^0}(\bar{L}) / \bar{L}$ is isomorphic to a diagonal subgroup of ${\rm Sym}_3 \times {\rm Sym}_3$. 

We now have a maximal rank subgroup $\bar{L}$ with $N_{\bar{H}^0}(\bar{L})/\bar{L} \cong {\rm Sym}_3$ and we note that $\bar{H}$ and $\bar{L}$ are $\s$-stable. As explained in \cite[Section~1]{LSS}, we may compose the Steinberg endormorphism $\s$ of $\bar{G}$ with the inner automorphism of $\bar{G}$ corresponding to the lift of an element in $N_{\bar{H}^0}(\bar{L})/\bar{L}$. By a slight abuse of notation, we will write $\s$ to denote this composition. Then by choosing the element in $N_{\bar{H}^0}(\bar{L})/\bar{L}$ appropriately, we will obtain the two different possibilities for $H_0$ by taking the fixed points in $\bar{H}$ of the modified map $\s$. We consider the cases $\e=+$ and $\e=-$ separately. 

First assume $\e=+$. Here we take $\s$ to be the product of the standard Frobenius morphism of $\bar{G}$ with an inner automorphism corresponding to the lift of an element of order $3$ in $N_{\bar{H}^0}(\bar{L})/\bar{L}$. Set
$K = (\bar{K}_\sigma)' = \text{L}_3(q^3) < H_0$.
By inspecting \cite[Tables 8.54 and 8.55]{BHR}, we see that $K$ is contained in a unique maximal subgroup of $H_0$, namely 
\[
S = N_{H_0}(K) = (Z_{q^2+q+1} \times \text{L}_3(q^3)).3.e.2
\]
(recall that we may assume $h=1$). It follows that $C_{H_0}(K) = C_{S}(K) = Z_{q^2+q+1}$, which we know from above is a maximal torus of $\bar{M}_\sigma = \text{SL}_3(q) \leqs C_{G_0}(K)$. Therefore, we may choose $g \in \bar{M}_\sigma$ so that it does not normalize $C_{H_0}(K)$. It follows that $g$ does not normalize $S$ because otherwise $C_{S}(K)^g = C_{S^g}(K^g) = C_{S}(K)$, which contradicts our choice of $g$. Therefore, $S^g \not\leqs H_0$. Since $K$ and $S$ are $A$-stable, we can now conclude by applying Lemma \ref{l:simple}. 

Finally, let us assume $\e=-$. Let $\s$ be the product of the standard Frobenius morphism with the lift of an involution in $N_{\bar{H}^0}(\bar{L})/\bar{L}$. Set $K = \bar{L}_\sigma$.
Then by considering \cite[Tables 8.56 and 8.57]{BHR}, 
we deduce that $S = N_{H_0}(K) = K.({\rm Sym}_3 \times 2)$ is the unique maximal overgroup of $K$ in $H_0$. As noted above, we have $N_{\bar{G}}(\bar{L}) = \bar{L}.({\rm Sym}_3 \times {\rm Sym}_3 \times 2)$ and thus $S<R \leqs N_{G_0}(K)$ with 
\[
R = K.({\rm Sym}_3 \times {\rm Sym}_3 \times 2).
\]
Since $S/K$ is not normal in $R/K$, there exists $g \in R$ which does not normalize $S$ and we now conclude as above, via Lemma \ref{l:simple}.
\end{proof}

\subsection{$G_0=E_7(q)$}\label{ss:mr_e7}

\begin{lem}\label{l:maxrank_e7_1}
If $G_0=E_7(q)$ and $H$ is the normalizer of a maximal torus, then $b(G,H)=2$.
\end{lem}

\begin{proof}
Let $W(\bar{G}) = {\rm Sp}_{6}(2) \times 2$ be the Weyl group of $\bar{G}$. If $x \in G$ is a long root element, then $|x^G| > q^{34} = b_1$ and Corollary \ref{c:root} implies that $|x^G \cap H| \leqs 63(q+1)^7 = a_1$. For all other nontrivial elements we have $|x^G|>q^{52} = b_2$ (see Proposition \ref{p:bounds}) and we note that 
\[
|H| \leqs (q+1)^7.|W(\bar{G})|.\log_2q = a_2.
\]
For $q \geqs 3$ we deduce that $\mathcal{Q}(G,H) < a_1^2/b_1 + a_2^2/b_2 <q^{-1}$.

Finally, let us assume $q=2$, so $G = E_7(2)$ and $H = 3^7.W(\bar{G})$. Here Lemma \ref{l:structure}(iv) implies that $G$ is not extremely primitive, but we need to work harder to show that $b(G,H)=2$. Using {\sc Magma} \cite{Magma}, we can construct $H$ as a subgroup of $E_7(4)$ and we can then determine the Jordan form of each element $x \in H$ of prime order $r$ on the adjoint module $V$ for $\bar{G}$ (see \cite[Proposition 2.2]{BTh_comp} for further details on this computation). For $r=2$, we inspect \cite[Table 8]{Lawunip} to determine the $G_0$-class of $x$. Similarly, if $r$ is odd then $\dim C_{\bar{G}}(x) = \dim C_V(x)$ and by inspecting \cite[Table 2]{BBR} we can read off the $G_0$-class of $x$ (note that if $r=3$ and $x$ is in one of the classes labelled \texttt{3C} or \texttt{3D} in \cite[Table 2]{BBR}, then $\dim C_{\bar{G}}(x) = 49$ in both cases, so the eigenvalues on $V$ do not allow us to distinguish between these classes). The results are as follows (here we label the involution classes as in \cite[Table 22.2.2]{LieS}):
\[
\begin{array}{llllllll} \hline
A_1 & 189 & & \texttt{3A} & 56 & & \texttt{5A} & 3919104 \\
A_1^2 & 8505 & & \texttt{3B} & 6174 & & \texttt{7C} & 151165440 \\
(A_1^3)^{(1)} & 8505 & & \mbox{\texttt{3C} or \texttt{3D}} & 1392372 & & & \\
(A_1^3)^{(2)} & 127575 & & \texttt{3E} & 3992352 & & & \\
A_1^4 & 583929 & & & & & \\ \hline
\end{array}
\]
So for example, if $x \in G$ is an involution in the class labelled $A_1^4$, then $|x^G \cap H| = 583929$. It is now entirely straightforward to check that $\mathcal{Q}(G,H)<1$ and thus $b(G,H) = 2$.
\end{proof}

\begin{lem}\label{l:maxrank_e7_2}
If $G_0=E_7(q)$ and $H$ is of type $A_1(q)^7$ or $A_1(q^7)$, then $b(G,H)=2$.
\end{lem}

\begin{proof}
First assume $H$ is of type $A_1(q)^7$. If $x \in G$ is a long root element, then  by applying Proposition \ref{p:root} we deduce that $|x^G \cap H| = 7(q^2-1)=a_1$ and we have $|x^G|>q^{34} = b_1$. Otherwise, $|x^G|>q^{52} = b_2$ and we note that 
\[
|H| \leqs (q(q^2-1))^7.|{\rm L}_{3}(2)|.\log_2q < q^{30} = a_2.
\]
Therefore, $\mathcal{Q}(G,H) < a_1^2/b_1 + a_2^2/b_2 < q^{-1}$ and thus $b(G,H)=2$. 

Now assume $H$ is of type $A_1(q^7)$, so $|H| \leqs 7q^7(q^{14}-1).\log_2q =a_1$.
Here $H$ does not contain any long root elements of $G_0$, so $|x^G|>q^{52} = b_1$ for all nontrivial $x \in H$ and we get $\mathcal{Q}(G,H) < a_1^2/b_1< q^{-1}$.
\end{proof}

\begin{lem}\label{l:maxrank_e7_3}
Suppose $G_0=E_7(q)$ and $H$ is of type 
\[
A_1(q)D_6(q), \; A_2^{\e}(q)A_5^{\e}(q), \; D_4(q)A_1(q)^3, \; {}^3D_4(q)A_1(q^3), \; E_6^{\e}(q).(q-\e). 
\]
Then $G$ is not extremely primitive. 
\end{lem}

\begin{proof}
Set $H_0 = H \cap G_0$ and note that the structure of $N_{L}(H_0)$ is recorded in \cite[Table 5.1]{LSS}, where $L = {\rm Inndiag}(G_0)$. Set $d=(2,q-1)$.

First assume $H$ is of type $A_1(q)D_6(q)$, so $H_0 = d.({\rm L}_{2}(q) \times {\rm P\O}_{12}^{+}(q))$. If $q$ is odd then $H$ is the centralizer of an involution, so $Z(H) \ne 1$ and thus $G$ is not extremely primitive by Lemma \ref{l:structure}(i). On the other hand, if $q$ is even then ${\rm soc}(H)$ is not a direct product of isomorphic simple groups, so Lemma \ref{l:structure}(v) implies that $G$ is not extremely primitive. Very similar reasoning rules out extreme primitivity when $H$ is of type $A_2^{\e}(q)A_5^{\e}(q)$, $D_4(q)A_1(q)^3$ and ${}^3D_4(q)A_1(q^3)$. 

Finally, let us assume $H$ is of type $E_6^{\e}(q).(q-\e)$, so 
\[
H_0 = e.(E_6^{\e}(q) \times (q-\e)/de).e.2
\]
with $e = (3,q-\e)$. If $e=3$ then $F(H) = Z_3$ and we apply Lemma \ref{l:structure}(iv). For the remainder, we may assume $e=1$. If $q \geqs 4$ is even then $H_0 = (E_6^{\e}(q) \times (q-\e)).2$ and ${\rm soc}(H)$ is not a direct product of isomorphic simple groups, hence $G$ is not extremely primitive by Lemma \ref{l:structure}(v). The same argument applies if $(\e,q) = (-,2)$. Similarly, if $q$ is odd then we may assume $(\e,q) = (+,3)$. Note that if $G = E_7(3).2$ then $H = (E_6(3) \times 2).2$ and the structure of ${\rm soc}(H)$ is incompatible with extreme primitivity.

To complete the proof, we may assume $\e = +$, $q \in \{2,3\}$ and $G = G_0$, in which case $H = \la E_6(q), \tau \ra = {\rm Aut}(E_6(q))$, where $\tau$ is an involutory graph automorphism of $S = {\rm soc}(H)$ with $K = C_{S}(\tau) = F_4(q)$. Now $K$ is a maximal subgroup of $S$ and $\mathcal{M}= \{M,S\}$ is the set of maximal overgroups of $K$ in $H$, where $M = K \times \langle \tau \rangle$. By considering the embedding of $S$ in $G$, we observe that $K$ centralizes a subgroup $L \cong \text{L}_2(q)$ of $G$. In particular, $K \times L< G$ and we see that $\tau \in L$.

Since $H = N_{G_0}(S)$ it follows that $\tau$ is the only nontrivial element of $L$ which normalizes $S$ and hence the only nontrivial element normalizing $H$. Since $\la \tau \ra$ is non-normal in $L$, it follows that there exists $g \in L$ which centralizes $K$ but does not normalize $M$ nor $S$. We can now apply Lemma \ref{l:simple} to conclude that $G$ is not extremely primitive.
\end{proof}

In order to complete the proof of Theorem \ref{t:maxrank} for $G_0 = E_7(q)$, we may assume $H$ is of type $A_7^{\e}(q)$. 

\begin{lem}\label{l:maxrank_e7_4}
If $G_0=E_7(q)$ and $H$ is of type $A_7^{\e}(q)$, then $G$ is not extremely primitive. 
\end{lem}

\begin{proof}
First let us record that the structure of $N_{\bar{G}_{\s}}(H_0)$ given in 
\cite[Table 5.1]{LSS} is incorrect. The correct structure is
\[
N_{\bar{G}_{\s}}(H_0) = h.\text{L}^{\epsilon}_8(q).g.2,
\]
where $h=(4,q-\epsilon)/d$ and $g=(8,q-\epsilon)/h$. Therefore, by appealing to Lemma \ref{l:structure}(iv), we may assume that $q \not\equiv \e \imod{4}$ for the remainder of the proof. Write $G = G_0.A$ and $H = H_0.A$, where $A$ is a group of automorphisms of $G_0$.

Fix a set of simple roots $\a_1, \ldots, \a_7$ for $\bar{G}$, labelled in the usual way (see \cite{Bou}). For each root $\a$, let $X_{\a}$ be the corresponding root subgroup of $\bar{G}$ and note that
\[
\bar{H}^0 = \la X_{\pm \a_0}, X_{\pm \a_1}, X_{\pm \a_3}, X_{\pm \a_4}, X_{\pm \a_5}, X_{\pm \a_6}, X_{\pm \a_7} \ra 
\]
is of type $A_7$, where $\a_0$ is the highest root. Let $\bar{P} = \bar{Q}\bar{L}$ be the standard $P_4$ maximal parabolic subgroup of $\bar{H}^0$ with Levi factor $\bar{L} = A_3^2T_1$. Then $N_{\bar{H}^0}(\bar{L}) = \bar{L}.2$, where the outer involution swaps the two $A_3$ factors and inverts the $T_1$ torus (as can be calculated in the Weyl group of $\bar{H}^0$). Let  $\bar{K} = \bar{L}' = A_3^2$ and note that $\bar{K}.2 < N_{\bar{H}^0}(\bar{L})$. It is straightforward to check that $C_{\bar{H}^0}(\bar{K})^0 = Z(\bar{L})^0= T_1$. 

By inspecting \cite[Table 2]{LS94}, we see that $A_3A_1$ is the connected centralizer in $\bar{G}$ of the first $A_3$ factor of $\bar{K}$ and it follows that $\bar{M} = C_{\bar{G}}(\bar{K})^0$ is of type $A_1$ and contains $Z(\bar{L})^0$ as a maximal torus. In fact, since $\la X_{\pm \a_2} \ra$ clearly centralizes $\bar{K}$, we have $\bar{M} = \la X_{\pm \a_2} \ra$. Now
\[
\bar{L} = A_3^2T_1 < A_3^2A_1 < D_6A_1 < \bar{G},
\]
where $D_6A_1$ is a maximal subsystem subgroup of $\bar{G}$, and the Weyl group of $D_6$ contains an involution swapping the two $A_3$ factors of $\bar{K}$. So $N_{\bar{G}}(\bar{L})/\bar{L}$ has a subgroup $Z_2 \times Z_2 = \la a \ra \times \la b \ra$, where $a$ swaps the two $A_3$ factors and $b$ inverts the torus $Z(\bar{L})^0$. The diagonal subgroup $Z_2 = \la ab \ra$ is contained in the Weyl group of $\bar{H}^0$.

The case $q=2$ requires special attention and it will be handled at the end of the proof. So for now, let us assume $q \geqs 3$. 

First assume $\epsilon=+$, so $q \not\equiv 1 \imod{4}$. Here we replace the standard Frobenius morphism 
$\s$ of $\bar{G}$ defining $G_0$ by the product of $\s$ and the inner automorphism of $\bar{G}$ induced by the lift of an outer involution in $\bar{K}.2$ which swaps the two $A_3$ factors. We will abuse notation by writing $\s$ for this modified map. Then $(\bar{G}_\sigma)' = G_0$, $N_{\bar{G}_{\s}}(H_0) = (\bar{H}.2)_\sigma = \text{L}_8(q).d.2$ and $\bar{K}_\sigma = d.\text{L}_4(q^2).d$. Set $K = (\bar{K}_\sigma)'  = d.\text{L}_4(q^2) < H_0 = \text{L}_8(q).2$.

By inspecting \cite[Tables 8.44 and 8.45]{BHR}, we see that $K$ is contained in a unique maximal subgroup of $H_0$, namely its normalizer
\[
L = (\bar{L}.2)_\sigma = ((q+1) \circ K).d.2^2 = d.((q+1)/d \times \text{L}_4(q^2)).d.2^2.
\]
Since $q \geqs 3$, we can choose an element $g \in \bar{M}_\sigma = \text{SL}_2(q) \leqs C_{G_0}(K)$ that does not normalize the non-normal subgroup $(Z(\bar{L})^0)_\sigma = C_{H_0}(K) = Z_{q+1}$ of $\bar{M}_\sigma$. Suppose $L^g$ is contained in $H_0$. Then $K = K^g$ is contained in $L^g$, which is a maximal subgroup of $H_0$, so $L = L^g$ and thus $g$ normalizes $L$. But $C_{H_0}(K) = C_L(K)$ and so $g$ also normalizes $C_{H_0}(K)$, which is a contradiction. Therefore, $L^g \not\leqs H_0$ and the desired result now follows by applying Lemma \ref{l:simple}, noting that $K$ and $L$ are $A$-stable.

Next suppose $\epsilon=-$ and let us continue to assume $q \geqs 3$. Recall that $q \not\equiv 3 \imod{4}$. In this case, we replace the standard Frobenius morphism $\sigma$ by the product of $\s$ with a lift of the longest element of the Weyl group of $\bar{G}$ (we will continue to write $\s$ for the modified map). Define $K = (\bar{K}_\sigma)' = d.\text{U}_4(q)^2 < H_0$. By inspecting \cite[Tables 8.46 and 8.47]{BHR}, we see that 
\[
L = N_{H_0}(K) = ((q+1) \circ K).d.2^2 = d.((q+1)/d \times \text{U}_4(q)^2.d).2^2
\]
is the unique maximal overgroup of $K$ in $H_0$. Since $q \geqs 3$, it follows that $(Z(\bar{L})^0)_\sigma = C_{H_0}(K) = Z_{q+1}$ is a non-normal subgroup of $\bar{M}_\sigma = \text{SL}_2(q) \leqs C_{G_0}(K)$, so there exists $g \in \bar{M}_\sigma$  which does not normalize $C_{H_0}(X)$. We now complete the argument as we did in the $\e=+$ case above.   

To complete the proof of the lemma, we may assume $q=2$ and thus $G = G_0$. For $\e=+$ we take $\sigma$ to be the standard Frobenius morphism, and we take the product of this with a lift of the following Weyl group element 
\begin{equation}\label{e:welt}
w = (1,6) (2) (3,5) (7,126) \ldots \in W(\bar{G})
\end{equation}
when $\epsilon = -$. Here we are expressing $w$ as a permutation of the set of roots of $\bar{G}$, where our labelling is consistent with {\sc Magma} (we only give part of the permutation, but this is enough to uniquely determine it). Set
\[
K = \bar{K}_\sigma = \left\{\begin{array}{ll} 
\text{SL}_4(2) \times \text{SL}_4(2) & \mbox{if $\epsilon=+$} \\
\text{SL}_4(4) & \mbox{if $\epsilon=-$.} 
\end{array}\right.
\]
Clearly, $M = N_{H_0}(K) = K.2$ is a maximal overgroup of $K$ in $H_0$. By  inspecting \cite{BHR}, we see that every other maximal overgroup of $K$ is a parabolic subgroup of the form $2^{16}{:}K$. 

We claim that $K$ is contained in precisely two maximal parabolic subgroups of $H_0$.   Clearly, we can take the fixed points under $\s$ of the standard maximal parabolic subgroup of $\bar{H}^0$ containing $\bar{K}$, as well as the opposite parabolic subgroup, so there are at least two such subgroups. Let $P$ be a maximal parabolic subgroup of $H_0$ and suppose $K \leqs P \cap P^h$ for some $h \in H_0$. Then $K, K^{h^{-1}} < P$ and \cite[Proposition 26.1(b)]{MT} implies that $K$ and $K^{h^{-1}}$ are $P$-conjugate. So $K^x = K^{h^{-1}}$ for some $x \in P$ and thus $xh \in N_{H_0}(K) = K.2$. If $xh \in K$ then $P^h = P$. On the other hand, if $xh \not \in K$ then $P^{xh} \neq P$ since $N_{H_0}(P) = P$. But for each $y$ in the coset $Kxh$ we have $P^{y} = P^{xh}$ and thus $K$ is contained in precisely two maximal parabolic subgroups of $H_0$ as claimed.

In view of the claim, let us write $\mathcal{M} = \{M, M_1,M_2\}$ where $M_1$ and $M_2$ are the maximal parabolic subgroups of $H_0$ containing $K$. To complete the proof, we will exhibit an element $g \in N_{G_0}(K)$ such that none of the subgroups $M^g, M_1^g, M_2^g$ are contained in $H_0$. The result will then follow from Lemma \ref{l:simple}.

By considering the above set up at the algebraic group level, we see that $C_{G_0}(K)$ contains $\bar{M}_\sigma = \text{SL}_2(2) \cong {\rm Sym}_3$. In addition, we note that $\s$ induces a standard Frobenius morphism on $\bar{M}$ since the Weyl group element $w$ in \eqref{e:welt} fixes the roots $\alpha_2$ and $-\alpha_2$. Therefore, we may choose $g = x_{\alpha_2}(1)x_{-\alpha_2}(1) \in \bar{M}_\sigma$. Since $g$ has order $3$, it does not commute with the involution $n_{\alpha_2} \in \bar{M}_\sigma$ (here $n_{\alpha_2}$ is the standard lift of the fundamental reflection in $W(\bar{G})$ corresponding to $\a_2$). The lift of the involution in $N_{\bar{H}^0}(\bar{K}) / \bar{K} \cong {\rm Sym}_2$ is the product of $n_{\alpha_2}$, which inverts a maximal torus of $\bar{M}$, and an involution that swaps the two $A_3$ factors of $\bar{K}$. Therefore, $g$ does not normalize $\bar{K}.2$ and so it does not normalize $M = K.2$. In particular, $M^g \not\leqs H_0$.

Finally, we need to consider $M_1$ and $M_2$. They are the fixed points under $\s$ of the standard maximal parabolic subgroups of $\bar{H}^0$ containing $\bar{K}$. In particular, the unipotent radicals of $M_1$ and $M_2$ contain $x_1 = x_{\alpha_4}(1)$ and $x_2 = x_{-\alpha_4}(1)$, respectively (since both elements are $\s$-stable). A straightforward calculation shows that
$x_1^g = x_{\alpha_2 + \alpha_4}(1)$ and $x_2^g$ is the product of $x_2$ with $x_{-(\alpha_2 + \alpha_4)}(1)$. Therefore, if $M_1^g \leqs H_0$ then $x_{\alpha_2 + \alpha_4}(1) \in H_0$. Similarly, if $M_2^g \leqs H_0$ then $x_{-(\alpha_2 + \alpha_4)}(1) \in H_0$. However, $x_{\pm(\alpha_2 + \alpha_4)}(1)^{n_{\alpha_4}} = x_{\pm\alpha_2}(1)$ and thus $x_{\pm(\alpha_2 + \alpha_4)}(1)$ is not even contained in $\bar{H}^0$. This implies that $M_i^g \not\leqs H_0$ for $i=1,2$ and the proof of the lemma is complete. 
\end{proof}

\subsection{$G_0=E_6^{\e}(q)$}\label{ss:mr_e6}

\begin{lem}\label{l:maxrank_e6_1}
If $G_0=E_6^{\e}(q)$ and $H$ is the normalizer of a maximal torus, then $b(G,H)=2$.
\end{lem}

\begin{proof}
Set $H_0 = H \cap G_0$ and note that the structure of $N_{L}(H_0)$ is recorded in \cite[Table 5.2]{LSS}, where $L = {\rm Inndiag}(G_0)$. Let $W(\bar{G}) = {\rm O}_{6}^{-}(2)$ be the Weyl group of $\bar{G}$.

Let $x \in G$ be an element of prime order. If $x$ is a long root element, then $|x^G| >(q-1)q^{21} = b_1$ and Corollary \ref{c:root} implies that $|x^G \cap H| \leqs 36(q+1)^6 = a_1$. If $x$ is not a long root element, nor an involutory graph automorphism with $C_{\bar{G}}(x) = F_4$, then $|x^G|>(q-1)q^{31}=b_2$ and we have 
\[
|H| \leqs (q+1)^6|W(\bar{G})|.2\log_2q = a_2.
\]
Finally, suppose $x$ is an $F_4$-type graph automorphism, so $|x^G|>\frac{1}{3}(q-1)q^{25}=b_3$. Since $W(\bar{G})$ is centralized by a graph automorphism, it follows that 
\[
|x^G \cap H| \leqs (q+1)^6i_{2}(W(\bar{G}) \times 2) = 3567(q+1)^6 = a_3
\]
and we conclude that $\mathcal{Q}(G,H) < \sum_{i=1}^{3}a_i^2/b_i<q^{-1}$ for all $q \geqs 5$. 

To complete the proof, we may assume $q \leqs 4$. First let us consider the case where  
\[
N_{L}(H_0)= (q^2 + \e q+1)^3.3^{1+2}.{\rm SL}_{2}(3),
\] 
so $(q,\e) \ne (2,-)$ (see \cite[Table 5.2]{LSS}). By arguing as in the previous paragraph, we see that the contribution to $\mathcal{Q}(G,H)$ from all elements other than $F_4$-type graph automorphisms is less than $a_1^2/b_1+a_2^2/b_2$, where $b_1$ and $b_2$ are defined as above, $a_1 = 36(q^2-\e q+1)^3$ and 
\[
a_2 = (q^2+\e q+1)^33^3|{\rm SL}_{2}(3)|.2\log_2q
\]
is an upper bound on $|H|$. Now assume $x$ is a graph automorphism with $C_{\bar{G}}(x) = F_4$. Here $|x^G|>b_3$ as above and we have
\[
|x^G \cap H| \leqs (q^2+\e q+1)^33^3.i_2({\rm SL}_{2}(3) \times 2) = 81(q^2+\e q+1)^3 = a_3.
\]
This gives $\mathcal{Q}(G,H)<\sum_{i=1}^{3}a_i^2/b_i<1$ for $q \geqs 3$. 

Suppose $(q,\e) = (2,+)$, so $G_0 = E_6(2)$ and $H_0 = 7^3{:}3^{1+2}.{\rm SL}_{2}(3)$. To handle this case, we view $\bar{G}.2<E_7$ and we use {\sc Magma} to construct $H_0.2$ as a subgroup of 
$E_7(8)$ (see \cite[Example 1.11]{BTh_comp}). In this way, we calculate that $i_2(H) \leqs 847=a_1$. Since $|x^G|>2^{41}=b_2$ for all $x \in G$ of odd prime order (see \cite[Table 9]{BLS}), it follows that 
\[
\mathcal{Q}(G,H) < a_1^2/b_1 + a_2^2/b_2 < 1,
\]
where $b_1 = 2^{21}$ and $a_2 = 2|H_0|$. 

For the remainder of the proof, we may assume $\e=-$ and $N_{L}(H_0) = (q+1)^6.W(\bar{G})$ with $q \leqs 4$ (according to \cite[Table 5.2]{LSS}, the condition 
$q \leqs 4$ implies that $\e=-$). First assume $q=2$, so either $G = G_0.3$ and $H = 3^6.W(\bar{G})$, or $G = G_0.{\rm Sym}_3$ and $H = 3^6.(W(\bar{G}) \times 2)$ (see \cite[Table 5.2]{LSS}). To get started, let us assume $G = G_0.3$. As explained in \cite[Proposition 2.2]{BTh_comp}, we can use {\sc Magma} to construct $H$ as a subgroup of $E_6(4)$ and we then compute the action of each element $x \in H$ of prime order $r$ on the adjoint module $V$ for $\bar{G}$. More precisely, if $r=2$ we compute the Jordan form of $x$ on $V$ and we inspect \cite[Table 6]{Lawunip} to determine the $G_0$-class of $x$ (in the table below, we use the labelling of unipotent classes in \cite[Table 22.2.3]{LieS}). For $r \in \{3,5\}$ we compute $\dim C_V(x) = \dim C_{\bar{G}}(x)$ and this allows us to identify the structure of $C_{\bar{G}}(x)^0$. The results we obtain are summarised in the following table:
\[
\begin{array}{llllllll} \hline
\multicolumn{2}{l}{r=2} & & \multicolumn{2}{l}{r=3} & & \multicolumn{2}{l}{r=5} \\ \hline
A_1 & 108 & & D_5T_1 & 54 & & A_3T_3 & 419904 \\
A_1^2 & 2430 & & A_5T_1 & 2232 & & & \\
A_1^3 & 18225 & & D_4T_2 & 47610 & & & \\
& & & A_4A_1T_1 & 39312 & & & \\
& & & A_2^3 & 144800 & & & \\ \hline
\end{array}
\]
In each case, it is easy to determine a lower bound on $|x^G|$ and the desired bound $\mathcal{Q}(G,H) < 1$ quickly follows. For example, one checks that the contribution from elements of order $3$ is less than $\sum_{i=1}^{5}a_i^2/b_i$, where
\[
a_1 = 54,\; a_2 = 2232,\; a_3 = 47610,\; a_4 = 39312,\; a_5 = 144800,
\]
\[
b_1 =  2^{31},\; b_2 = 2^{41},\; b_3 = 2^{45}, \; b_4 = 2^{44},\; b_5 = 2^{52}.
\]

To complete the analysis of this case, let us now assume $G = G_0.{\rm Sym}_3$ and $x \in G$ is an involutory graph automorphism. Since the algebraic group $E_7$ contains a subgroup $E_6.2$, we can use {\sc Magma} to construct $H$ as a subgroup of $E_7(4)$ (once again, see \cite{BTh_comp}) and we find that there are $5$ conjugacy classes of involutions in $H \setminus H_0$; the size of each class and the Jordan form of a representative on the adjoint module $\mathcal{L}(E_7)$ for $E_7$ are as follows:
\begin{align*}
(J_2^{53},J_1^{27}){:} & \; 405 \\
(J_{2}^{63},J_1^7){:} & \; 729+8748+14580+21870 = 45927
\end{align*}
If $C_{\bar{G}}(x) = F_4$, then \cite[Table 7]{LLS2014} indicates that $x$ is contained in the $E_7$-class labelled $(A_1^3)''$ in \cite{Lawunip}, whence $x$ has Jordan form $(J_2^{53},J_1^{27})$ on $\mathcal{L}(E_7)$ (see \cite[Table 8]{Lawunip}). Similarly, if $C_{\bar{G}}(x) \ne F_4$ then $x$ has Jordan form $(J_{2}^{63},J_1^7)$. It follows that the contribution to $\mathcal{Q}(G,H)$ from involutory graph automorphisms is less than $a_1^2/b_1 + a_2^2/b_2< 0.02$, where
\[
a_1 = 405,\; b_1 = \frac{1}{3}2^{25},\; a_2 = 45927,\; b_2 = \frac{1}{3}2^{41}.
\]
By combining this with the above estimates, we conclude that $\mathcal{Q}(G,H)<1$ and the result follows.

Next assume $q=3$. Here $G = G_0$ and $H = 4^6.W(\bar{G})$, or $G = G_0.2$ and $H = 4^6.(W(\bar{G}) \times 2)$. First assume $G = G_0$. Here we construct $H$ as a subgroup of $E_6(9)$ (see \cite{BTh_comp}) and by considering the action of $H$ on the adjoint module for $\bar{G}$ we obtain the following results:
\[
\begin{array}{llllllll} \hline
\multicolumn{2}{l}{r=2} & & \multicolumn{2}{l}{r=3} & & \multicolumn{2}{l}{r=5} \\ \hline
D_5T_1 & 5211 & & A_2 & 3840 & & A_3T_3 & 1327104 \\
A_1A_5 & 60516 & & A_2^2 & 122880 & & & \\
& & & A_2^2A_1 & 327680 & & & \\ \hline
\end{array}
\]
The desired bound $\mathcal{Q}(G,H)<1$ quickly follows. 

Now assume $G = G_0.2$. Here we construct $H$ as a subgroup of $E_7(9)$ (see \cite{BTh_comp} again for the details) and we deduce that $i_2(H \setminus H_0) = 147520$. Moreover, by calculating the eigenvalues of the involutions in $H \setminus H_0$ on the adjoint module for $E_7$, we see that $720$ have a $79$-dimensional $1$-eigenspace and the remainder have a $63$-dimensional $1$-eigenspace. If $x \in H$ is such an involution, then we may view $x$ as a semisimple involution in $E_7$ via the embedding $E_6.2<E_7$ and we recall that the connected component of the centralizer of an involution in $E_7$ is one of $A_1D_6$, $A_7$ or $E_6T_1$ (see \cite[Proposition 1.2]{LLS}). Since $F_4$ is not contained in $A_1D_6$ or $A_7$, it follows that if $C_{\bar{G}}(x) = F_4$ then $C_{E_7}(x)^0 = E_6T_1$ and thus $|x^G \cap H| = 720=a_1$. On the other hand, if $C_{\bar{G}}(x) \ne F_4$ then $|x^G \cap H| = 146800=a_2$. Therefore, the combined contribution to 
$\mathcal{Q}(G,H)$ from graph automorphisms is less than $a_1^2/b_1 + a_2^2/b_2<10^{-6}$, where $b_1 = \frac{1}{2}3^{26}$ and $b_2 = \frac{1}{2}3^{42}$. It is now easy to check that $\mathcal{Q}(G,H)<1$ and the result follows.

Finally, let us assume $q=4$, so either $G = G_0$ and $H = 5^6.W(\bar{G})$, or $G = G_0.2$ and $H = 5^6.(W(\bar{G}) \times 2)$, or $G = G_0.4$ and $H = 5^6.W(\bar{G}).4$. First assume $G = G_0$. Here we construct $H$ as a subgroup of $E_6(16)$ and as before we study the action of $H$ on the adjoint module for $\bar{G}$ (as usual, see \cite{BTh_comp} for the details). If $x \in G$ is a long root element, then $|x^G|>\frac{1}{2}4^{22}=b_1$ and we find that $|x^G \cap H| = 180=a_1$. On the other hand, if $x$ is not a long root element then Proposition \ref{p:bounds} gives $|x^G|>3.4^{31}=b_2$ and we set $a_2 = |H|$. Now assume $G \ne G_0$. Since every element of prime order in $G_0.4$ is contained in $G_0.2$, we may assume that $G = G_0.2$. We now construct $H$ as a subgroup of $E_7(16)$ and we find that $i_2(H \setminus H_0) = 365500=a_3$. Since $|x^G|>\frac{1}{2}4^{26}=b_3$ for all involutory graph automorphisms $x \in G$, we conclude that $\mathcal{Q}(G,H) < \sum_{i=1}^3 a_i^2/b_i<1$ and thus $b(G,H)=2$.
\end{proof}

\begin{lem}\label{l:maxrank_e6_2}
Suppose $G_0=E_6^{\e}(q)$ and $H$ is of type $A_2^{\e}(q)^3$, $A_2(q^2)A_2^{-\e}(q)$ or $A_2^{\e}(q^3)$. Then $b(G,H)=2$.
\end{lem}

\begin{proof}
Here $\bar{H} = A_2^3.{\rm Sym}_{3}$, where ${\rm Sym}_3 = \la a, b \ra$ and the action of $a$ and $b$ on $\bar{H}^0 = A_2^3$ is given by
\begin{equation}\label{e:action}
a : \, (x_1,x_2,x_3) \mapsto (x_3,x_1,x_2),\;\; b:\, (x_1,x_2,x_3) \mapsto (x_2^{\tau},x_1^{\tau},x_3^{\tau})
\end{equation}
where $x_i$ denotes an arbitrary element in the $i$th $A_2$ factor of $\bar{H}^0$ and $\tau$ is an involutory graph automorphism of $A_2$ (in terms of matrices, we may view $\tau$ as the inverse-transpose map $y \mapsto y^{-T}$). Let $V = \mathcal{L}(\bar{G})$ be the adjoint module for $\bar{G}$ and let $V_{27}$ be one of the $27$-dimensional minimal modules. As noted in \cite[Table 3]{Thomas}, we have  
\begin{equation}\label{e:adj}
V{\downarrow}\bar{H}^0 = \mathcal{L}(\bar{H}^0) \oplus (W \otimes W \otimes W) \oplus (W^* \otimes W^* \otimes W^*)
\end{equation}
and
\[
V_{27}{\downarrow}\bar{H}^0 = (W \otimes W^* \otimes 0) \oplus (W^* \otimes 0 \otimes W) \oplus (0 \otimes W \otimes W^*)
\]
where $W$ is the natural module for $A_2$, $W^*$ its dual and $0$ the trivial module. Before we begin the main analysis, it will be useful to record some preliminary observations concerning the embedding of $\bar{H}$ in $\bar{G}$. 

Suppose $x \in \bar{H}$ has order $p$. If $x \in \bar{H}^0$ then the $\bar{G}$-class of $x$ is determined in \cite[Section 4.9]{Law09}. The nontrivial unipotent classes in $A_2$ are labelled $A_1$ and $A_2$, which gives a corresponding labelling of the classes in $\bar{H}^0$. It turns out that the $\bar{G}$-class containing a given $\bar{H}^0$-class inherits the same label, with the exception of the regular $\bar{H}^0$-class labelled $A_2^3$ (the latter is contained in the $\bar{G}$-class $D_4(a_1)$ if $p \ne 3$ and $A_2^2A_1$ if $p=3$).

Now assume $x \in \bar{H} \setminus \bar{H}^0$ has order $p$, so $p = 2$ or $3$. Suppose $p=3$ and note that $x$ is $\bar{H}$-conjugate to $a$. Now \eqref{e:action} indicates that $x$ cyclically permutes the three $A_2$ factors of $\bar{H}$, whence $x$ has Jordan form $(J_3^9)$ on $V_{27}$. By inspecting \cite[Table 5]{Lawunip}, it follows that $x$ is in one of the $\bar{G}$-classes $A_2^2$ or $A_2^2A_1$. Visibly, $x$ centralizes a diagonally embedded $A_2$ subgroup of $\bar{H}^0$ and so by considering the possibilities for $C_{\bar{G}}(x)^0$ (see \cite[Table 22.1.3]{LieS}) we conclude that $x$ is in the $A_2^2$ class. 

Now assume $p=2$ and $x = (x_1,x_2,x_3)b \in \bar{H}$ has order $2$, so $x_2 = x_1^{T}$ and $x_3 = x_3^{T}$. Since every invertible symmetric matrix is congruent to the identity matrix, it is easy to see that $x$ is $\bar{H}^0$-conjugate to $b$, so there is a unique $\bar{H}$-class of involutions in $\bar{H} \setminus \bar{H}^0$. Note that 
\[
C_{\bar{H}^0}(b) = \{(x_1,x_2,x_3) \in \bar{H}^0 \,:\, x_2 = x_1^{\tau},\, x_3 = x_3^{\tau}\}
\]
and thus $C_{\bar{H}^0}(b)^0$ is of type $A_2A_1$. Now $b$ has Jordan form $(J_2^8) \oplus (J_2^3,J_1^2)$ on $\mathcal{L}(\bar{H}^0)$ and it interchanges the two $27$-dimensional summands in \eqref{e:adj}, so $b$ has Jordan form $(J_2^{38},J_1^2)$ on $V$. By inspecting \cite[Table 6]{Lawunip}, we conclude that every involution in $\bar{H} \setminus \bar{H}^0$ is contained in the class labelled $A_1^3$.

Next we turn to the semisimple elements in $\bar{H}$. Suppose $x \in \bar{H}^0$ has prime order $r \ne p$. By working with the decomposition in \eqref{e:adj}, it is easy to compute the eigenvalues of $x$ on $V$, which allows us to read off $\dim C_{V}(x) = \dim C_{\bar{G}}(x)$. For example, if $x = (x_1,x_2,x_3) \in \bar{H}^0$ has order $2$, where $x_1 = x_2 \ne 1$ and $x_3 = 1$ then the dimension of the $1$-eigenspace of $x$ on $\mathcal{L}(\bar{H}^0)$ is equal to $\dim C_{\bar{H}^0}(x) = 4+4+8 = 16$ and we calculate that $x$ acts as $(-I_{12},I_{15})$ on the two $27$-dimensional summands in \eqref{e:adj}. Therefore, $\dim C_{\bar{G}}(x) = 16+30 = 46$ and thus $C_{\bar{G}}(x)^0 = D_5T_1$. Now assume $x \in \bar{H}\setminus \bar{H}^0$, so $r \in \{2,3\}$. If $r=2$ then $x$ is $\bar{H}$-conjugate to $b$ and we calculate that $x$ acts as $(-I_{40},I_{38})$ on $V$, so $C_{\bar{G}}(x)^0 = A_5A_1$. Similarly, if $r=3$ then $x$ is conjugate to $a$ and we have $\dim C_{\bar{G}}(x) = 30$, so $C_{\bar{G}}(x)^0 = D_4T_2$. 

Finally, let $\gamma$ be an involutory graph automorphism of $\bar{G}$. Then the normalizer of $\bar{H}$ in $\bar{G}.2 = \bar{G}.\la \gamma \ra$ is $\bar{H}.2 = \bar{H}^0.({\rm Sym}_{3} \times 2)$, where ${\rm Sym}_3$ acts naturally on the three factors of $\bar{H}^0$ and a generator $c$ for the cyclic group of order $2$ acts as a simultaneous  graph automorphism on all three factors. There are three classes of involutions in $\bar{H}.2 \setminus \bar{H}$, represented by $c$, $bc$ and $(1,1,t)bc$, where $t \in A_2$ is an involution. Here $bc$ acts as a transposition on the three factors of $\bar{H}^0$, swapping the first two and centralizing the third. Working with \eqref{e:adj}, we calculate that $bc$ has Jordan form $(J_2^{26},J_1^{26})$ or $(-I_{26},I_{52})$ on $V$, according to the parity of $p$, and we deduce that $C_{\bar{G}}(bc) = F_4$. On the other hand, if $x = c$ or $(1,1,t)bc$ then $x$ has Jordan form $(J_2^{36},J_1^6)$ or $(-I_{42},I_{36})$ and we see that $C_{\bar{G}}(x) \ne F_4$. 

\vs

\noindent \emph{Case 1. $H$ is of type $A_2^{\e}(q)^3$.}

\vs

We are now ready to begin the proof of the lemma. First assume $H$ is of type $A_2^{\e}(q)^3$, so 
\[
H_0 = e.{\rm L}_{3}^{\e}(q)^3.e.{\rm Sym}_{3} \; \mbox{ and } \; K = N_{L}(H_0) = e.{\rm L}_{3}^{\e}(q)^3.e^2.{\rm Sym}_{3},
\]
where $e = (3,q-\e)$ and $L = {\rm Inndiag}(G_0) = G_0.e$. We proceed by estimating the contribution to $\mathcal{Q}(G,H)$ from the different types of elements of prime order in $G$.

First we consider the contribution from unipotent elements. Suppose $p=2$ and let $x \in H$ be a unipotent involution. Both ${\rm SL}_{3}^{\e}(q)$ and ${\rm L}_{3}^{\e}(q)$ contain $(q^2-1)(q^2+\e q+1)< 2q^4$ involutions, which form a single conjugacy class, and we recall from above that every involution in $\bar{H}\setminus \bar{H}^0$ is contained in the class labelled $A_1^3$. If $x$ is in the $G_0$-class labelled $A_1$ then $|x^G \cap H| < 3.2q^4=a_1$ and $|x^G|>(q-1)q^{21}=b_1$. 
Similarly, if $x$ is in the $A_1^2$ class, then
$|x^G \cap H| <3.(2q^4)^2 = 12q^8=a_2$ and $|x^G|>(q-1)q^{31} = b_2$.
Finally, if $x$ is in the class labelled $A_1^3$, then
\[
|x^G \cap H| < (2q^{4})^3 + 3\cdot \frac{|{\rm SL}_{3}^{\e}(q)|^3}{|{\rm SL}_{3}^{\e}(q)||{\rm SL}_{2}(q)|} < 8q^{12}+6q^{13} = a_3,\;\; |x^G|>(q-1)q^{39} = b_3
\]
and thus the unipotent contribution when $p=2$ is less than $\sum_{i=1}^3a_i^2/b_i$.

Now assume $p \geqs 3$. As above, the contribution to $\mathcal{Q}(G,H)$ from elements in the classes labelled $A_1$, $A_1^2$ and $A_1^3$ is less than $\sum_{i=1}^3a_i^2/b_i$. Now $|x^G|>\frac{1}{4}q^{42} = b_4$ for all other nontrivial elements of order $p$ in $G_0$ (see \cite[Table 22.2.3]{LieS}) and using Proposition \ref{p:i23}(ii) we note that 
\[
i_p(e.{\rm L}_{3}^{\e}(q)^3.e)  = i_p({\rm L}_{3}^{\e}(q)^3)< q^{18} = a_4.
\]
In addition, if $p=3$ then there are elements of order $p$ in $\bar{H} \setminus \bar{H}^0$ which transitively permute the factors of $\bar{H}^0$. As explained above, such an element $x$ is in the $\bar{G}$-class labelled $A_2^2$, so $|x^G|>\frac{1}{6}q^{48} = b_5$ and we note that there are at most $2|{\rm SL}_{3}^{\e}(q)|^2< 2q^{16}=a_5$ of these elements in $H_0$.

We conclude that the unipotent contribution is less than
$\sum_{i=1}^{5}a_i^2/b_i < q^{-4}+q^{-5}$ for all $p$. 

Next let us turn to the contribution to $\mathcal{Q}(G,H)$ from semisimple elements. Let $x \in H$ be a semisimple element of prime order $r$. First assume $r=2$, so $C_{\bar{G}}(x)^0 = A_5A_1$ or $D_5T_1$. Recall that the involutions in $\bar{H} \setminus \bar{H}^0$ are of type $A_5A_1$ and note that both ${\rm SL}_{3}^{\e}(q)$ and ${\rm L}_{3}^{\e}(q)$ contain $q^2(q^2+\e q +1) < 2q^4$ involutions, which form a single class. By considering the decomposition in \eqref{e:adj}, we can calculate $\dim C_{V}(x)$ for each involution $x \in H$ and this allows us to determine the $G_0$-class of $x$. In this way, we deduce that if $x$ is of type $D_5T_1$ then $|x^G\cap H| < 3.(2q^4)^2 = 12q^8 = a_6$ and $|x^G|> (q-1)q^{31} = b_6$, whereas 
\[
|x^G \cap H| < 3.2q^4 + (2q^4)^3 + 3\cdot \frac{|{\rm SL}_{3}^{\e}(q)|^3}{|{\rm SL}_{3}^{\e}(q)||{\rm SL}_{2}(q)|} < 6q^4+8q^{12}+6q^{13} = a_7
\]
and $|x^G|>(q-1)q^{39} = b_7$ if $x$ is of type $A_5A_1$. It follows that the contribution to $\mathcal{Q}(G,H)$ from semisimple involutions is less than $a_6^2/b_6+a_7^2/b_7<q^{-9}$.

Next assume $r=3$, so $p \ne 3$. First  assume $e=3$, so we have $H_0 = 3.{\rm L}_{3}^{\e}(q)^3.3.{\rm Sym}_{3}$ and $K = 3.{\rm L}_{3}^{\e}(q)^3.3^2.{\rm Sym}_{3}$.
Let $x \in K$ be a semisimple element of order $3$, whence 
\[
\mbox{$C_{\bar{G}}(x)^0 = A_5T_1$, $D_4T_2$, $A_2^3$, $D_5T_1$ or $A_1A_4T_1$}
\]
(see \cite[Table 4.7.1]{GLS}) and we recall that $C_{\bar{G}}(x)^0 = D_4T_2$ if $x \in \bar{H} \setminus \bar{H}^0$. Let $Z$ be the normal subgroup of $K$ of order $3$, so  
\[
K/Z = {\rm L}_{3}^{\e}(q)^3.3^2.{\rm Sym}_{3} < {\rm PGL}_{3}^{\e}(q)^3.{\rm Sym}_{3}
\]
and
\[
i_3(K) < |Z|\cdot (1+i_3({\rm PGL}_{3}^{\e}(q) \wr Z_3)).
\]
By applying Proposition \ref{p:i23}(i), we deduce that
\[
1+i_3({\rm PGL}_{3}^{\e}(q) \wr Z_3) \leqs (i_3({\rm Aut}({\rm L}_{3}^{\e}(q)))+1)^3+2|{\rm PGL}_{3}^{\e}(q)|^2 < 8(q+1)^3q^{15}+2q^{16}
\]
and thus $i_3(K) < 24(q+1)^3q^{15}+6q^{16}$. If $q=2$ then $\e=-$ and we can use {\sc Magma} to construct $K$ as a subgroup of $E_6(4)$ (see \cite[Lemma 2.3]{BTh_comp}); in this way, we calculate that $i_3(K) = 492074$. Set
\[
a_{8} = \left\{\begin{array}{ll}
24(q+1)^3q^{15}+6q^{16} & \mbox{if $q>2$,} \\
492074 & \mbox{if $q=2$.}
\end{array}\right.
\]
If $\dim x^{\bar{G}}>42$ then $|x^G|>\frac{1}{6}q^{48} = b_{8}$ and the contribution to $\mathcal{Q}(G,H)$ from these elements is less than $a_{8}^2/b_{8}$. 

Now assume $\dim x^{\bar{G}} \leqs 42$, so $x \in \bar{H}^0$ and we have $C_{\bar{G}}(x)^0 = D_5T_1$ or $A_5T_1$. By considering \eqref{e:adj}, one can check that $C_{\bar{G}}(x)^0 = D_5T_1$  if and only if $Zx  \in {\rm PGL}_{3}^{\e}(q)^3$ is of the form $(x_1,x_2,1)$ (up to permutations of the coordinates), where $x_1 \in {\rm PGL}_{3}^{\e}(q)$ is conjugate to the image modulo scalars of a diagonal matrix ${\rm diag}(1,1,\omega) \in {\rm GL}_{3}^{\e}(q)$ with $\omega$ a primitive cube root of unity (so $x_1 \in {\rm PGL}_{3}^{\e}(q) \setminus {\rm L}_{3}^{\e}(q)$ is a diagonal automorphism) and $x_2$ is conjugate to $x_1^{-1}$. Now $|x^G|>(q-1)q^{31} = b_{9}$ and there are at most 
\[
3! \cdot \left(\frac{|{\rm SL}_{3}^{\e}(q)|}{|{\rm GL}_{2}^{\e}(q)|}\right)^2 < 24q^8 = a_{9}
\]
such elements in $K$. Now assume $C_{\bar{G}}(x)^0 = A_5T_1$, so $|x^G|>(q-1)q^{41} = b_{10}$. Here we find that $C_{\bar{G}}(x)^0 = A_5T_1$ if and only if $Zx  \in {\rm PGL}_{3}^{\e}(q)^3$ is one of the following, up to permutations:
\begin{itemize}\addtolength{\itemsep}{0.2\baselineskip}
\item[{\rm (a)}] $(x_1,1,1)$, where $x_1$ is conjugate to the image of ${\rm diag}(1,\omega,\omega^2) \in {\rm GL}_{3}^{\e}(q)$; or
\item[{\rm (b)}] $(x_1,x_2,x_3)$, where $x_1$ and $x_2$ are conjugate to the image of ${\rm diag}(1,1,\omega)$ and $x_3$ is conjugate to $x_1^{-1}$.
\end{itemize}
Therefore, there are at most 
\[
3 \cdot \frac{|{\rm SL}_{3}^{\e}(q)|}{|{\rm GL}_{1}^{\e}(q)|^2} + 3\cdot 2 \cdot 
\left(\frac{|{\rm SL}_{3}^{\e}(q)|}{|{\rm GL}_{2}^{\e}(q)|}\right)^3 < 12q^6 + 48q^{12}= a_{10}
\]
such elements in $K$.
 
Bringing these estimates together, we conclude that if $e=3$, then the contribution to $\mathcal{Q}(G,H)$ from semisimple elements of order $3$ is less than $\sum_{i=8}^{10}a_{i}^2/b_i< q^{-4}$. 

Now assume $e = 1$, so $H_0 = {\rm L}_{3}^{\e}(q)^3.{\rm Sym}_{3}$ and we note that ${\rm L}_{3}^{\e}(q)$ contains $q^3(q^3+\e)< 2q^6$ elements of order $3$, which form a single class (note that these elements are regular). The possibilities for $C_{\bar{G}}(x)^0$ are $A_5T_1$, $A_2^3$ and $D_4T_2$. By working with \eqref{e:adj} we calculate that if $C_{\bar{G}}(x)^0 = A_5T_1$ then $|x^G\cap H|< 3.2q^6 = 6q^6$ and $|x^G|>(q-1)q^{41}$. Similarly, if $C_{\bar{G}}(x)^0 = A_2^3$ then $|x^G\cap H|< (2q^6)^3 = 8q^{18}$ and $|x^G|>\frac{1}{6}q^{54}$. And for $C_{\bar{G}}(x)^0 = D_4T_2$ we get 
\[
|x^G\cap H|<3.(2q^6)^2 + 2|{\rm SL}_{3}^{\e}(q)|^2 < 12q^{12}+2q^{16},\;\; |x^G|>\frac{1}{6}q^{48}.
\]
From the above estimates, it is straightforward to check that the combined contribution to $\mathcal{Q}(G,H)$ from semisimple elements of order $3$ is less than $q^{-4}$ for all $q$. 

To complete the analysis of semisimple elements, let us assume $r \geqs 5$. If $\dim x^{\bar{G}} \geqs 50$ then $|x^G|>(q-1)q^{49} = b_{11}$ and there are clearly fewer than $q^{24} = a_{11}$ such elements in $H$. Now assume $\dim x^{\bar{G}} < 50$, so $C_{\bar{G}}(x)^0 = D_5T_1$, $A_5T_1$ or $D_4T_2$. Since $r \geqs 5$ we have $x \in \bar{H}^0$, say $x = (x_1, x_2, x_3)$, and by working with the decomposition in \eqref{e:adj} we deduce that $\dim C_{\bar{G}}(x) \leqs 36$. Therefore, $C_{\bar{G}}(x)^0 = A_5T_1$ or $D_4T_2$ and thus $|x^G|>(q-1)q^{41} = b_{12}$. If each $x_i$ is nontrivial then we calculate that $\dim C_{\bar{G}}(x) \leqs 24$, which is a contradiction. Therefore, there are fewer than $3|{\rm SL}_{3}^{\e}(q)| + 3|{\rm SL}_{3}^{\e}(q)|^2 < 3q^8(q^8+1) = a_{12}$ such elements in $H$ and we conclude that the entire contribution to $\mathcal{Q}(G,H)$ from semisimple elements of order at least $5$ is less than $a_{11}^2/b_{11}+a_{12}^2/b_{12}<q^{-1}+q^{-4}$.

Finally, let us assume $x \in G$ is a field, graph or graph-field automorphism of $G_0$. First assume $x$ is a field or graph-field automorphism of order $r$, so $q=q_0^r$. If $r \geqs 3$ then $|x^G|>\frac{1}{6}q^{52} = b_{13}$ and we note that there are fewer than 
$\log_2q.6q^{24} = a_{13}$ such elements in $H$. Now assume $r=2$, so $\e=+$ and it is convenient to use the bound
\[
{\rm fpr}(x,G/H)< \frac{|\bar{H}:\bar{H}^0|\cdot q^{24}}{(q^{1/2}-1)^6q^9|x^{G_0}|}< 36q^{-24}(q^{1/2}-1)^{-6},
\]
which is explained in the proof of \cite[Lemma 6.1]{LLS2}. Since $|x^G|<2q^{39}$, it follows that the contribution to $\mathcal{Q}(G,H)$ from field and graph-field automorphisms is less than
\[
\eta\left(a_{13}^2/b_{13}\right)+2\cdot 2q^{39} \cdot \left(36q^{-24}(q^{1/2}-1)^{-6} \right)^2 < 4\eta q^{-1}+\mu q^{-2},
\]
where $\eta=1$ if $q=p^f$ and $f$ is divisible by an odd prime, otherwise $\eta = 0$, and similarly $\mu = 1$ if $q$ is a square, otherwise $\mu=0$.

Now assume $x \in G$ is an involutory graph automorphism. As explained earlier, if $C_{\bar{G}}(x) \ne F_4$ then either $x$ induces a graph automorphism on each $A_2$ factor of $\bar{H}^0$, or it swaps two of the factors and acts nontrivially on the third. Therefore, $|x^G|>\frac{1}{6}q^{42} = b_{14}$ and there are at most
\[
\left(\frac{|{\rm SL}_{3}^{\e}(q)|}{|{\rm SL}_{2}(q)|}\right)^3 + 3|{\rm SL}_{3}^{\e}(q)|\cdot (q^2-1)(q^2+\e q+1) < 2q^{15} = a_{14}
\] 
such elements in $H$. On the other hand, if $C_{\bar{G}}(x) = F_4$ then $|x^G|>\frac{1}{6}q^{26} = b_{15}$ and $x$ acts as a transposition on the factors of $\bar{H}^0$, centralizing the fixed factor, whence 
$H$ contains at most $3|{\rm SL}_{3}^{\e}(q)|< 3q^8 = a_{15}$ such elements. Therefore, the contribution to $\mathcal{Q}(G,H)$ from graph automorphisms is less than $a_{14}^2/b_{14}+a_{15}^2/b_{15}<q^{-4}$.

Bringing together all of the above bounds, we conclude that
\[
\mathcal{Q}(G,H)<(1+4\eta)q^{-1}+\mu q^{-2}+4q^{-4}+q^{-5}+q^{-9}<1
\]
and the result follows.

\vs

\noindent \emph{Case 2. $H$ is of type $A_2(q^2)A_2^{-\e}(q)$.}

\vs

As before, set $L = {\rm Inndiag}(G_0)$ and observe that 
\[
K = N_{L}(H_0) = g.({\rm L}_{3}(q^2) \times {\rm L}_{3}^{-\e}(q)).h.2,
\]
where $g = (3,q+\e)$ and $h = (3,q^2-1)$. In particular, note that if $e = (3,q-\e)=3$ then $H_0 = ({\rm L}_{3}(q^2) \times {\rm L}_{3}^{-\e}(q)).2$. Let us also observe that there is a unique class of involutions in $K \setminus (g.({\rm L}_{3}(q^2) \times {\rm L}_{3}^{-\e}(q)).h)$, which acts on ${\rm L}_{3}(q^2) \times {\rm L}_{3}^{-\e}(q)$ by inducing a graph-field automorphism on the first factor and a graph automorphism on the second. This corresponds to the involution $b \in \bar{H}$ discussed earlier (see \eqref{e:action}). 

First assume $p=2$ and $x \in G_0$ is an involution. If $x$ is in the class labelled $A_1$, then $|x^G|>(q-1)q^{21}=b_1$ and $|x^G \cap H| = i_2({\rm L}_{3}^{-\e}(q)) < 2q^4 = a_1$. Similarly, if $x$ is in the $A_1^2$ class then $|x^G \cap H| = i_2({\rm L}_{3}(q^2)) < 2q^8 = a_2$ and $|x^G|>(q-1)q^{31}=b_2$. Now assume $x$ is in the class labelled $A_1^3$, so $|x^G|>(q-1)q^{39} = b_3$. Here we get
\[
|x^G \cap H| \leqs i_2({\rm L}_{3}(q^2)) \cdot i_2({\rm L}_{3}^{-\e}(q)) + \frac{|{\rm SL}_{3}(q^2)|}{|{\rm SU}_{3}(q)|} \cdot \frac{|{\rm SL}_{3}^{-\e}(q)|}{|{\rm SL}_{2}(q)|} < 4q^{13} = a_3.
\]
By arguing as in Case 1, we see that the unipotent contribution to $\mathcal{Q}(G,H)$ for any $p$ is less than $\sum_{i=1}^{4}a_i^2/b_i< q^{-4}+q^{-5}$, where $a_4 = q^{18}$ and $b_4 = \frac{1}{4}q^{42}$.

Next assume $x \in G$ is a semisimple element of prime order $r$. Suppose $r=2$, so $C_{\bar{G}}(x)^0 = D_5T_1$ or $A_5A_1$. If $x$ is a $D_5T_1$ involution, then $|x^G\cap H| = i_{2}({\rm L}_{3}(q^2))<2q^8 = a_5$ and $|x^G|>(q-1)q^{31}=b_5$. Similarly, if $x$ is an $A_5A_1$ involution then $|x^G|>(q-1)q^{39} = b_6$ and we see that 
\[
|x^G \cap H| \leqs i_2({\rm L}_{3}^{-\e}(q)) + i_2({\rm L}_{3}(q^2)) \cdot i_2({\rm L}_{3}^{-\e}(q)) + \frac{|{\rm SL}_{3}(q^2)|}{|{\rm SU}_{3}(q)|} \cdot \frac{|{\rm SL}_{3}^{-\e}(q)|}{|{\rm SL}_{2}(q)|} < 4q^{13} = a_6.
\]

Now suppose $r=3$. Let $Z$ be the normal subgroup of $K$ of order $g = (3,q+\e)$, so $K/Z$ is a subgroup of $({\rm PGL}_{3}(q^2) \times {\rm PGL}_{3}^{-\e}(q)).2$ and thus
\[
i_3(K) < |Z| \cdot (1+i_3({\rm PGL}_{3}(q^2) \times {\rm PGL}_{3}^{-\e}(q))).
\]
For $q \in \{2,3\}$ we compute 
\[
i_3({\rm PGL}_{3}(q^2) \times {\rm PGL}_{3}^{-\e}(q)) \leqs \left\{\begin{array}{ll}
387420488 & \mbox{if $q=3$} \\
391472 & \mbox{if $q=2$} 
\end{array}\right.
\]
and by applying Proposition \ref{p:i23}(i) we deduce that  
\[
a_7 = \left\{\begin{array}{ll}
12(q+1)(q^2+1)q^{15} & \mbox{if $q > 3$} \\
387420489 & \mbox{if $q=3$} \\
1174419 & \mbox{if $q=2$}
\end{array}\right.
\]
is an upper bound on $i_3(K)$. If $\dim x^{\bar{G}} \geqs 42$ then $|x^G|>(q-1)q^{41} = b_7$ and so the contribution to $\mathcal{Q}(G,H)$ from these elements is less than $a_7^2/b_7$. Now assume $\dim x^{\bar{G}}<42$, so $C_{\bar{G}}(x)^0 = D_5T_1$ is the only option and thus $|x^G|>(q-1)q^{31}=b_8$. By arguing as in Case 1, we see that there are at most $|Z|\cdot i_3({\rm PGL}_{3}(q^2)) < 6(q^2+1)q^{10}$ such elements in $H$. Set $a_8 = 6(q^2+1)q^{10}$ if $q>2$ and $a_8 = 3\cdot i_3({\rm PGL}_{3}(4)) = 14496$ if $q=2$.

Now assume $r \geqs 5$ and note that $i_r(K) = i_r({\rm L}_{3}(q^2) \times {\rm L}_{3}^{-\e}(q))$. As in Case 1, if $\dim x^{\bar{G}} \geqs 50$ then $|x^G|>(q-1)q^{49} = b_{9}$ and $H$ contains fewer than $q^{24}$ such elements. In fact, for $q=2$ we calculate that ${\rm L}_{3}(q^2) \times {\rm L}_{3}^{-\e}(q)$ contains at most $290352$ elements of prime order at least $5$, so we set $a_{9}=q^{24}$ if $q>2$ and $a_{9} = 290352$ if $q=2$. By arguing as before, if $\dim x^{\bar{G}}<50$ then $|x^G|>(q-1)q^{41} = b_{10}$ and there are fewer than $|{\rm SL}_{3}^{\e}(q^2)|+|{\rm SL}_{3}^{-\e}(q)|<q^{16} = a_{10}$ such elements in $H$.

Putting all of the above estimates together, we conclude that the contribution to $\mathcal{Q}(G,H)$ from semisimple elements is less than $\sum_{i=5}^{10}a_i^2/b_i<q^{-1}+q^{-2}$.

Finally, let us assume $x$ is a field, graph or graph-field automorphism. By repeating the argument in Case 1, we see that the contribution from field and graph-field automorphisms is less than $4\eta q^{-1}+\mu q^{-2}$, where $\eta$ and $\mu$ are defined as in Case 1. Now assume $x$ is an involutory graph automorphism. Recall that if $C_{\bar{G}}(x) = F_4$ then $|x^G|>\frac{1}{6}q^{26} = b_{11}$ and $x$ acts on the factors of $\bar{H}^0$ by swapping the first two and centralizing the third. Given the structure of $H_0$, this implies that $x$ induces a field automorphism on the ${\rm L}_{3}(q^2)$ factor and thus 
\[
|x^G\cap H| \leqs \frac{|{\rm SL}_{3}(q^2)|}{|{\rm SL}_{3}(q)|}< q^8 = a_{11}.
\]
Similarly, if $C_{\bar{G}}(x) \ne F_4$ then $|x^G|>\frac{1}{6}q^{42} = b_{12}$ and $x$ either induces graph automorphisms on both ${\rm L}_{3}(q^2)$ and ${\rm L}_{3}^{-\e}(q)$, or it acts as a field automorphism on ${\rm L}_{3}(q^2)$ and as an involutory inner automorphism on ${\rm L}_{3}^{-\e}(q)$. It follows that 
\[
|x^G \cap H| \leqs \frac{|{\rm SL}_{3}(q^2)|}{|{\rm SL}_{2}(q^2)|}\cdot \frac{|{\rm SL}_{3}^{-\e}(q)|}{|{\rm SL}_{2}(q)|} + \frac{|{\rm SL}_{3}(q^2)|}{|{\rm SL}_{3}(q)|}\cdot (q^2-1)(q^2-\e q+1) < 2q^{15} = a_{12}
\]
and thus the total contribution to $\mathcal{Q}(G,H)$ from graph automorphisms is less than $a_{11}^2/b_{11}+a_{12}^2/b_{12}<q^{-6}$.

In conclusion, the above estimates show that
\[
\mathcal{Q}(G,H) < (1+4\eta)q^{-1}+(1+\mu)q^{-2}+q^{-4}+q^{-5}+q^{-6}
\]
and one checks that this upper bound is always less than $1$.  

\vs

\noindent \emph{Case 3. $H$ is of type $A_2^{\e}(q^3)$.}

\vs

To complete the proof, we may assume $H$ is of type $A_2^{\e}(q^3)$, so $H_0 = {\rm L}_{3}^{\e}(q^3).3$ and 
\[
K = N_{L}(H_0) = {\rm L}_{3}^{\e}(q^3).(e \times 3) \leqs {\rm PGL}_{3}^{\e}(q^3).3
\]
where $L = {\rm Inndiag}(G_0)$ and $e=(3,q-\e)$ as before. Note that $H_0$ contains an element of order $3$, which acts as a field automorphism on ${\rm soc}(H_0) = {\rm L}_{3}^{\e}(q^3)$; this corresponds to the element $a \in \bar{H}$, which transitively permutes the three factors of $\bar{H}^0$ (see \eqref{e:action}).

First assume $x \in H$ is a unipotent element of $G$. If $x$ is a long root element in ${\rm L}_{3}^{\e}(q^3)$ then $x$ is contained in the $A_1^3$ class of $G_0$, so $|x^G|>(q-1)q^{39} = b_1$ and 
\[
|x^G\cap H| = \frac{|{\rm SL}_{3}^{\e}(q^3)|}{q^9(q^3-\e)} < 2q^{12}=a_1.
\]
If $p \geqs 3$ and $x \in {\rm L}_{3}^{\e}(q^3)$ is regular, then \cite[Section 4.9]{Law09} implies that $x$ is contained in the $\bar{G}$-class labelled $A_2^2A_1$ (if $p=3$) or $D_4(a_1)$ (if $p >3$). Similarly, if $p=3$ and $x \in H_0$ is a field automorphism of order $3$, then $x$ is in the class $A_2^2$. In both cases, $|x^G|>\frac{1}{6}q^{48}=b_2$ and there are fewer than
\[
i_p({\rm L}_{3}^{\e}(q^3)) + 2\cdot \frac{|{\rm SL}_{3}^{\e}(q^3)|}{|{\rm SL}_{3}^{\e}(q)|}<2q^{18}=a_2
\]
such elements in $H$. Therefore, the unipotent contribution to $\mathcal{Q}(G,H)$ is less than $a_1^2/b_1+a_2^2/b_2<q^{-7}$.

Next assume $x \in H$ is a semisimple element of prime order $r$. If $r=2$ then $|x^G \cap H| = i_2({\rm L}_{3}^{\e}(q^3))<2q^{12}=a_3$ and $C_{\bar{G}}(x)^0 = A_5A_1$, so $|x^G|>(q-1)q^{39} = b_3$. Now suppose $r=3$. By considering the structure of $K$ and the action of $\bar{H}^0$ on the adjoint module for $\bar{G}$ (see \eqref{e:adj}), it is easy to see that $C_{\bar{G}}(x)^0 \ne D_5T_1$, so $|x^G|>(q-1)q^{41}=b_4$ and by applying Proposition \ref{p:i23}(i) we deduce that 
\[
i_3(K) \leqs i_3({\rm Aut}({\rm L}_{3}^{\e}(q^3))) < 2(q^3+1)q^{15} = a_4.
\]
Similarly, if $r \geqs 5$ then $\dim x^{\bar{G}} \geqs 50$, so $|x^G|>(q-1)q^{49}=b_5$ and we note that $|{\rm L}_{3}^{\e}(q^3)|<q^{24}=a_5$. Therefore, the contribution to $\mathcal{Q}(G,H)$ from semisimple elements is less than $\sum_{i=3}^{5}a_i^2/b_i<q^{-1}+q^{-2}$.

Finally, let us assume $x \in G$ is a field, graph or graph-field automorphism. Suppose $x$ is a field automorphism of prime order $r$, so $q=q_0^r$ and we may assume $r \ne 3$ since every element in $H$ of order $3$ is contained in $K$. If $r=2$ then $\e=+$ and $x$ acts as an involutory field automorphism on ${\rm L}_{3}(q^3)$, so $|x^G|>\frac{1}{6}q^{39}=b_6$ and
\[
|x^G \cap H| \leqs \frac{|{\rm SL}_{3}(q^3)|}{|{\rm SL}_{3}(q^{3/2})|}<2q^{12} = a_6.
\]
If $r \geqs 5$ then $|x^G|>\frac{1}{6}q^{312/5}=b_7$ and we note that $|H|<6\log_2q.q^{24} = a_7$. Next assume $x \in G$ is an involutory graph-field automorphism, so $\e=+$ and $q=q_0^2$. Here $|x^G|>\frac{1}{6}q^{39}=b_8$ and $x$ induces a graph-field automorphism on ${\rm L}_{3}(q^3)$, so $|x^G \cap H|<2q^{12}=a_8$. Finally, if $x$ is an involutory graph automorphism then $x$ acts as a graph automorphism on ${\rm L}_{3}(q^3)$, so
\[
|x^G \cap H| \leqs \frac{|{\rm SL}_{3}^{\e}(q^3)|}{|{\rm SL}_{2}(q^{3})|}<2q^{15} = a_9.
\]
In terms of the ambient algebraic groups, $x$ must act as a simultaneous graph automorphism on the three factors of $\bar{H}^0$ (since ${\rm L}_{3}^{\e}(q^3)$ is not normalized by a graph automorphism that also induces a nontrivial permutation of the factors of $\bar{H}^0$). Therefore, $C_{\bar{G}}(x) \ne F_4$ and thus $|x^G|>\frac{1}{6}q^{42}=b_9$. It follows that the contribution to $\mathcal{Q}(G,H)$ from field, graph and graph-field automorphisms is less than $\sum_{i=6}^{9}a_i^2/b_i<q^{-5}$.

To conclude, it follows that 
\[
\mathcal{Q}(G,H) < \sum_{i=1}^{9}a_i^2/b_i < q^{-1}+q^{-2}+q^{-5}+q^{-7}
\]
and the proof of the lemma is complete.
\end{proof}

\begin{lem}\label{l:maxrank_e6_3}
Suppose $G_0=E_6^{\e}(q)$ and $H$ is of type 
\[
A_1(q)A_5^{\e}(q), \; D_4(q).(q-\e)^2, \; {}^3D_4(q).(q^2+\e q+1), \; D_5^{\e}(q).(q-\e).
\]
Then $G$ is not extremely primitive. 
\end{lem}

\begin{proof}
Write $H_0 = H \cap G_0$ and $L = {\rm Inndiag}(G_0)$. The structure of $N_{L}(H_0)$ is presented in \cite[Table 5.1]{LSS}.

First assume $H$ is of type $A_1(q)A_5^{\e}(q)$, so $H_0 = d.({\rm L}_{2}(q) \times {\rm L}_{6}^{\e}(q)).d$, where $d=(2,q-1)$. If $q$ is odd then $Z(H) \ne 1$, whereas ${\rm soc}(H)$ is not a direct product of isomorphic simple groups if $q$ is even. Therefore, $G$ is not extremely primitive by Lemma \ref{l:structure}. Similar reasoning handles the cases where $H$ is of type $D_4(q).(q-\e)^2$ or ${}^3D_4(q).(q^2+\e q+1)$.

Finally, let us assume $H$ is of type $D_5^{\e}(q).(q-\e)$, so 
\[
H_0 = h.({\rm P\O}_{10}^{\e}(q) \times (q-\e)/eh).h
\]
with $e=(3,q-\e)$ and $h=(4,q-\e)$. If $q$ is odd then $F(H) = Z_h \neq 1$ and $G$ is not extremely primitive by Lemma \ref{l:structure}(iv). So for the remainder of the proof we may assume $q$ is even and thus $H_0 = \O_{10}^{\e}(q) \times (q-\e)/e$. If $(q,\e) \ne (2,-),(2,+),(4,+)$ then the structure of ${\rm soc}(H)$ is incompatible with extreme primitivity, so it remains to consider the three special cases. 

First assume $(q,\e) = (2,-)$. If $G$ contains ${\rm Inndiag}(G_0)$ then ${\rm soc}(H)$ is incompatible, so assume otherwise, in which case $(G,H)=({}^2E_6(2),\O_{10}^{-}(2))$ or $({}^2E_6(2).2,{\rm O}_{10}^{-}(2))$. Let $r$ be the rank of $G$ on $\O = G/H$. In both cases, the character tables of $G$ and $H$ are available in the \textsf{GAP} Character Table Library \cite{GAPCTL}. For $G = G_0$, the fusion map from $H$-classes to $G$-classes is also stored and this allows us to compute the number of fixed points of each $x \in H$ on $\O = G/H$. In turn, we deduce that $r=13$ via the  Orbit Counting Lemma. On the other hand, if $G = G_0.2$ then there are two possible fusion maps and they both give $r=12$. With the aid of {\sc Magma}, we can determine the indices $n_1, \ldots, n_k$ of a set of representatives of the $H$-classes of core-free maximal subgroups of $H$ and then it is routine to rule out extreme primitivity via Lemma \ref{l:char}. For example, if $G = {}^2E_6(2).2$ and $H = {\rm O}_{10}^{-}(2)$ then $n_i \leqs |H:{\rm M}_{12}.2| = 263208960$ and one checks that $1+11.263208960 < |G:H|$. We refer the reader to \cite[Lemma 2.4]{BTh_comp} for further details on the \textsf{GAP} and {\sc Magma} computations used in this case.

Next assume $(q,\e) = (2,+)$, so $G = {\rm Aut}(E_6(2)) = E_6(2).2$ and $H={\rm O}_{10}^{+}(2)$ (as noted in \cite[Table 5.1]{LSS}, if $\e=+$ then $H$ is maximal only if $G$ contains a graph automorphism). We can handle this case in a similar fashion. Working with the character tables of $G_0$ and $H_0 = \O_{10}^{+}(2)$ in \cite{GAPCTL}, we calculate that $G_0$ has rank $35$ and thus $G$ has rank at most $35$ (the character table of $G$ is not available in \cite{GAPCTL} and we have not computed the exact rank of $G$). Let $M$ be a core-free maximal subgroup of $G$. Then using {\sc Magma}, we find that $|H:M|$ is maximal when $M = {\rm O}_{6}^{+}(2) \times {\rm O}_{4}^{+}(2)$, giving $|H:M| \leqs  16189440$. But $1+34.16189440< |G:H|$, so $G$ is not extremely primitive by Lemma \ref{l:char}. 

Finally, let us turn to the case $(q,\e) = (4,+)$. Here we may assume $G$ contains a graph automorphism (so that $H$ is maximal) but does not contain ${\rm Inndiag}(G_0)$ (so that the structure of ${\rm soc}(H)$ is compatible with extreme primitivity). Therefore, $(G,H) = (E_6(4).2, {\rm O}_{10}^{+}(4))$ or $(E_6(4).2^2, {\rm O}_{10}^{+}(4).2)$. 

Fix a set of simple roots $\a_1, \ldots, \a_6$ for $\bar{G}$ and let $X_{\a} = \{x_{\a}(c)\,:\, c \in \mathbb{F}_4\}$ be the root subgroup of $G_0$ corresponding to the root $\a$. By replacing $H$ by a suitable conjugate, we may assume that
\[
H_0 = \la X_{\pm \a_1}, X_{\pm \a_2}, X_{\pm \a_3}, X_{\pm \a_4}, X_{\pm \a_5} \ra = \O_{10}^{+}(4).
\]
Let $g \in G_0 \setminus H_0$ be the involution $x_{\alpha_6}(1)$. By applying  Chevalley's commutator relations (see \cite[Theorem 5.2.2]{Carter}), we see that $g$ centralizes $X_{\a}$ for all $\a \in \{\pm \alpha_1, \pm\alpha_2, \pm\alpha_3, \pm\alpha_4, -\alpha_5\}$ and thus $g$ centralizes the subgroup $K = 4^{10}{:}\text{SL}_5(4) < H_0$. By inspecting \cite[Table 8.66]{BHR}, we see that the only maximal subgroup of $H_0.2 = \text{O}^+_{10}(4)$ containing $K$ is $H_0$ itself (note that the maximal parabolic subgroup $4^{10}{:}\text{GL}_5(4)$ of $H_0$ does not extend to a maximal subgroup of $H_0.2$). Similarly, the only maximal subgroups of ${\rm O}_{10}^{+}(4).2 = H_0.2^2$ containing $K$ are the three index-two subgroups of the form $H_0.2$.

Suppose $H \cap H^g$ is a maximal subgroup of $H$. Then since $g$ normalizes $K$, we have $K \leqs H \cap H^g < H$ and thus $H \cap H^g = H_0$ (if $G = G_0.2$) or $H_0.2$ (if $G = G_0.2^2$). Since $g$ is an involution, it normalizes $H \cap H^g$ and so it must also normalize the characteristic subgroup $H_0$. But from the explicit  description of $H_0$ above in terms of root subgroups, it is easy to see that $g$ does not normalize $H_0$ and so we have reached a contradiction. We conclude that $H \cap H^g$ is non-maximal in $H$ and the proof is complete.
\end{proof}

\begin{rem}
In the proof of the previous lemma, we applied Lemma \ref{l:char} to handle the case $(G,H) = (E_6(2).2, {\rm O}_{10}^{+}(2))$. It is worth noting that this case can also be treated by  making minor modifications to the argument we used for $(q,\e) = (4,+)$.
\end{rem}

\subsection{$G_0=F_4(q)$}\label{ss:mr_f4}

\begin{lem}\label{l:maxrank_f4_1}
If $G_0=F_4(q)$ and $H$ is the normalizer of a maximal torus, then $b(G,H)=2$.
\end{lem}

\begin{proof}
Let $W(\bar{G}) = {\rm O}_{4}^{+}(3)$ be the Weyl group of $\bar{G}$ and note that $q$ is even and $G$ contains graph automorphisms (see \cite[Table 5.2]{LSS}). It will be useful to note that if $x \in G$ has prime order, and $x$ is not a long or short root element, then $|x^G|>q^{22}$ (minimal if $x$ is an involution in the class labelled $(\tilde{A}_1)_2$). 

If $x \in G$ is a long (or short) root element, then $|x^G| > q^{16} = b_1$ and Corollary \ref{c:root} gives $|x^G \cap H| \leqs 24(q+1)^4 = a_1$. As noted above, for all other nontrivial elements we have $|x^G|>q^{22}$ and we observe that 
\[
|H| \leqs (q+1)^4.|W(\bar{G})|.2\log_2q = a_2.
\]
This yields $\mathcal{Q}(G,H) < a_1^2/b_1 + a_2^2/b_2$, which is less than $1$ for $q \geqs 8$ (it is also less than $q^{-1}$ for $q \geqs 8$). Therefore, to complete the proof we may assume $q \in \{2,4\}$. 

Suppose $q=2$, so $G = F_4(2).2$ and $H = 7^2{:}(3 \times 2.{\rm Sym}_4)$. Here $H_0 = H \cap G_0 = 7^2{:}(3 \times {\rm SL}_{2}(3))$ has a unique involution, so by \cite[Corollary 4.4]{BTh} we see that there are precisely $49$ involutions in $H_0$ and they are all contained in the largest $G_0$-class of involutions (this is the class labelled  $A_1\tilde{A}_1$). Therefore, $|x^G|>2^{26} = b_1$ for all $x \in H$ of prime order and thus $\mathcal{Q}(G,H) < a_1^2/b_1<1$, where $a_1 = |H|$. In particular, $b(G,H) = 2$. 

Now assume $q=4$, so $G = F_4(4).2$ or $F_4(4).4$. Up to conjugacy, there are $5$ possibilities for $H_0$ and we will inspect each case in turn. If $H_0 = (4^2\pm 4 +1)^2{:}(3 \times {\rm SL}_{2}(3))$ then by arguing as above we see that $|x^G|>4^{26}=b_1$ and the result follows since $\mathcal{Q}(G,H) < a_1^2/b_1<1$ with $a_1 = 4|H_0|$. The case $H_0 = 241{:}12$ is entirely similar. Next assume $H_0 = 17^2{:}(4 \circ {\rm GL}_{2}(3))$. If $x \in G$ is not a root element, then $|x^G|>4^{22} = b_1$ and the contribution from these elements is less than $a_1^2/b_1$, where $a_1 = 4|H_0|$.
On the other hand, if $x$ is a root element, then $|x^G|>4^{16} = b_2$ and we have the trivial bound $|x^G \cap H| \leqs |H_0| = a_2$. This gives $\mathcal{Q}(G,H) < a_1^2/b_1 + a_2^2/b_2 <1$ as required.

Finally, let us assume $H_0 = 5^4{:}W(\bar{G})$. As in the previous case, the contribution from non-long root elements is less than $a_1^2/b_1$, where $a_1 = 4|H_0|$ and $b_1 = 4^{22}$. Now assume $x$ is a root element, so $|x^G|>4^{16}=b_2$. Here we can use {\sc Magma} to construct $H_0$ as a subgroup of $F_4(16)$ (see \cite[Proposition 2.2]{BTh_comp}) and this allows us to compute the Jordan form of each involution in $H_0$ on the adjoint module for $\bar{G}$. By inspecting \cite[Table 4]{Lawunip}, we deduce that $|x^G \cap H| = 120 = a_2$. Therefore, $\mathcal{Q}(G,H) < a_1^2/b_1 + a_2^2/b_2 <1$ and thus $b(G,H)=2$.
\end{proof}

\begin{lem}\label{l:maxrank_f4_2}
If $G_0 = F_4(q)$ and $H$ is of type $A_2^{\e}(q)^2$, then $b(G,H) = 2$.
\end{lem}

\begin{proof}
Here $\bar{H}^0 = A_2\tilde{A}_2$, where the first factor is generated by long root subgroups and the second by short root subgroups. Set $e=(3,q-\e)$ and observe that
\[
H_0 = ({\rm SL}_{3}^{\e}(q) \circ {\rm SL}_{3}^{\e}(q)).e.2 = e.{\rm L}_{3}^{\e}(q)^2.e.2,
\]
where the outer involution acts as a graph automorphism on both copies of ${\rm SL}_{3}^{\e}(q)$. Let $V = \mathcal{L}(\bar{G})$ be the adjoint module for $\bar{G}$ and note that
\begin{equation}\label{e:a2a2}
V{\downarrow}A_2\tilde{A}_2 = \mathcal{L}(A_2\tilde{A}_2) \oplus (W \otimes S^2(W^*)) \oplus (W^* \otimes S^2(W)),
\end{equation}
where $W$ is the natural module for $A_2$, $W$ its dual and $S^2(W)$ is the symmetric-square of $W$ (see \cite[Table 2]{Thomas}). 

Let $x \in H$ be an element of prime order $r$. To begin with, let us assume $r=p=2$. Note that if $x \in \bar{H}^0$ then the $\bar{G}$-class and $\bar{H}^0$-class of $x$ have the same label (see \cite[Section 4.7]{Law09}). If $x$ acts as a graph automorphism on the two ${\rm SL}_{3}^{\e}(q)$ factors, then $x$ has Jordan form $(J_2^6,J_1^4)$ on 
$\mathcal{L}(A_2\tilde{A}_2)$ and it interchanges the two $18$-dimensional summands in \eqref{e:a2a2}. Therefore, $x$ has Jordan form $(J_2^{24},J_1^4)$ on $V$ and by inspecting \cite[Table 4]{Lawunip} we deduce that $x$ is in the $\bar{G}$-class labelled $A_1\tilde{A}_1$. Putting this together, it follows that if $x$ is a long or short root element, then $|x^G|>q^{16}=b_1$ and there are $2\cdot i_2({\rm L}_{3}^{\e}(q)) < 3q^4 = a_1$ such elements in $H$. On the other hand, if $x$ is in the class labelled $A_1\tilde{A}_1$ then $|x^G|>q^{28}=b_2$ and $H$ contains 
\[
i_2({\rm L}_{3}^{\e}(q))^2 + \left(\frac{|{\rm SL}_{3}^{\e}(q)|}{|{\rm SL}_{2}(q)|}\right)^2  = (q+\e)^2(q^3-\e)^2 + q^4(q^3-\e)^2 < 2q^{10} = a_2
\]
such elements (no involution in $H$ is contained in the class labelled $(\tilde{A}_1)_2$).

Next assume $r=p>2$. If $\dim x^{\bar{G}} \geqs 30$ then $|x^G|>\frac{1}{4}q^{30}=b_3$ and we note that $H_0$ contains $q^{12}=a_3$ unipotent elements (see Proposition \ref{p:i23}(ii)). Now assume $\dim x^{\bar{G}} < 30$, so $x$ is contained in one of the $G_0$-classes labelled $A_1$, $\tilde{A}_1$ or $A_1\tilde{A}_1$. As in the previous paragraph, the contribution to $\mathcal{Q}(G,H)$ from these elements is less than $a_1^2/b_1+a_2^2/b_2$. 

Now let us turn to semisimple elements. Suppose $x \in H$ is a semisimple element of prime order $r$. First assume $r=2$, so $C_{\bar{G}}(x) = A_1C_3$ or $B_4$. If $x$ acts as a graph automorphism on both ${\rm SL}_{3}^{\e}(q)$ factors, then $x$ has Jordan form $(-I_{10},I_6)$ on $\mathcal{L}(A_2\tilde{A}_2)$ and it swaps the two $18$-dimensional summands in \eqref{e:a2a2}, so $\dim C_{V}(x) = 24$ and thus $C_{\bar{G}}(x) = A_1C_3$. Now assume $x = (x_1,x_2) \in \bar{H}^0$. If either $x_2=1$, or $x_1$ and $x_2$ are both nontrivial, then $C_{\bar{G}}(x) = A_1C_3$. On the other hand, if $x_1=1$ then $C_{\bar{G}}(x) = B_4$. We conclude that if $x$ is a $B_4$-type involution, then $|x^G|>q^{16}=b_4$ and $|x^G \cap H| = i_2({\rm L}_{3}^{\e}(q))<2q^4 = a_4$. Similarly, if $C_{\bar{G}}(x) = A_1C_3$ then $|x^G|>q^{28}=b_5$ and
\[
|x^G \cap H| \leqs i_2({\rm L}_{3}^{\e}(q))+i_2({\rm L}_{3}^{\e}(q))^2 + \left(\frac{|{\rm SL}_{3}^{\e}(q)|}{|{\rm SL}_{2}(q)|}\right)^2 < 2q^{10}=a_5.
\]

Now assume $r \geqs 3$. If $\dim x^{\bar{G}} \geqs 36$ then $|x^G|>\frac{1}{2}q^{36}=b_6$ and we note that there are fewer than $q^{16}=a_6$ such elements in $H$. So we may assume $\dim x^{\bar{G}} < 36$, in which case $C_{\bar{G}}(x) = B_3T_1$ or $C_3T_1$ and $|x^G|>(q-1)q^{29}=b_7 = b_8$. Let us also observe that $r$ divides $|Z(C_{G_0}(x))| = q \pm 1$.

Suppose $r=3$ and let $Z$ be the normal subgroup of $H_0$ of order $e$. Then by applying Proposition \ref{p:i23}(i) we deduce that
\[
i_3(H_0) < |Z| \cdot (1+i_3({\rm PGL}_{3}^{\e}(q)^2)) \leqs 12(q+1)^2q^{10}.
\]
In fact, one checks that $i_3({\rm PGL}_{3}^{\e}(2)^2) \leqs 6560$ and thus $i_3(H_0) \leqs 19683$ when $q=2$. It follows that the contribution to $\mathcal{Q}(G,H)$ from elements $x \in G$ of order $3$ with $\dim x^{\G} < 36$ is less than $a_7^2/b_7$, where $a_7 = 12(q+1)^2q^{10}$ if $q \geqs 4$ and $a_7 = 19683$ if $q=2$. Now assume $r \geqs 5$ (and $C_{\bar{G}}(x) = B_3T_1$ or $C_3T_1$). An easy calculation using \eqref{e:a2a2} shows that an element $x = (x_1,x_2) \in \bar{H}^0$ of order $r$ has the appropriate centralizer in $\bar{G}$ if and only if $x_1$ or $x_2$ is trivial. Therefore, there are fewer than $2|{\rm L}_{3}^{\e}(q)|<2q^8=a_8$ such elements in $H$.

To complete the proof, let us assume $p=2$ and $x$ is an involutory graph automorphism, so $|x^G|>q^{26}=b_9$. In terms of the ambient algebraic groups, $x$ interchanges the two $A_2$ factors of $\bar{H}^0$, so there are two classes of involutions in $\bar{H}.2\setminus \bar{H}$, represented by $x$ and $yx$, where $y \in \bar{H}$ acts as a simultaneous graph automorphism on both factors. Now $C_{\bar{H}^0}(x)$ and $C_{\bar{H}^0}(yx)$ are both isomorphic to $A_2$ and we deduce that $|x^G \cap H| \leqs 2|{\rm SL}_{3}^{\e}(q)|<2q^8 = a_9$.

We conclude that $\mathcal{Q}(G,H)< \sum_{i=1}^{9}a_i^2/b_i$ and it is routine to verify that this upper bound is less than $1$ for all $q \geqs 2$ (and it tends to zero as $q$ tends to infinity).
\end{proof}

\begin{rem}\label{r:f4}
It is worth noting that the cases with $q=2$ in Lemma \ref{l:maxrank_f4_2} can also be handled using {\sc Magma}. We refer the reader to \cite[Lemma 2.5]{BTh_comp} for the details.
\end{rem}

\begin{lem}\label{l:maxrank_f4_20}
If $G_0 = F_4(q)$ and $H$ is of type $C_2(q)^2$, then $b(G,H) = 2$.
\end{lem}

\begin{proof}
Here $p=2$, $G$ contains a graph automorphism and 
\[
H_0 = {\rm Sp}_{4}(q) \wr {\rm Sym}_2 < {\rm Sp}_{8}(q) < G_0.
\]
In particular, $H_0$ is a non-maximal subgroup of $G_0$. Set $\bar{H}^0 = B_2^2 < B_4 < \bar{G}$ and let $W$ be the $4$-dimensional natural module for $B_2$, with $\rho:B_2 \to {\rm GL}(W)$ the corresponding representation. Let $\tau$ be the standard graph automorphism of $B_2$ and let $W^{\tau}$ be the $B_2$-module $W$ afforded by the representation $\rho\tau$. Then by inspecting \cite[Table 2]{Thomas} we see that the restriction of the adjoint module $V = \mathcal{L}(\bar{G})$ to $\bar{H}^0$ has the form 
\begin{equation}\label{e:f4adj}
V{\downarrow}B_2^2 = \mathcal{L}(B_2^2) \oplus (W \otimes W) \oplus (W^{\tau} \otimes W^{\tau})
\end{equation}

The case $q=2$ can be checked using {\sc Magma} (see \cite[Lemma 2.5]{BTh_comp}) and so we may assume $q \geqs 4$. Our goal is to show that $\mathcal{Q}(G,H)<1$.

Let $x \in H$ be an element of prime order $r$. First assume $x \in H_0$ and $r=2$. Recall that ${\rm Sp}_{4}(q)$ has three classes of involutions, labelled $b_1$, $a_2$ and $c_2$ in \cite{AS}. It is straightforward to determine the ${\rm Sp}_{8}(q)$-class of each involution in ${\rm Sp}_{4}(q)^2$ and then the $G_0$-class can be read off from \cite[p.373]{LLS} (also see \cite[Section 4.4]{Law09}). For example, if $x = (x_1,x_2) \in {\rm Sp}_{4}(q)^2$, where $x_1$ is of type $a_2$ and $x_2$ is of type $b_1$, then $x$ is a $b_3$ involution in ${\rm Sp}_{8}(q)$ and is therefore contained in the $G_0$-class labelled $A_1\tilde{A}_1$. For the reader's convenience, the $G_0$-class of each involution in ${\rm Sp}_{4}(q)^2$ is recorded in Table \ref{tab:uni}.
Similarly, if $x \in H_0$ interchanges the two copies of ${\rm Sp}_{4}(q)$ then $x$ embeds in ${\rm Sp}_{8}(q)$ as an $a_4$-type involution and thus $x$ is in the $G_0$-class labelled 
$(\tilde{A}_1)_2$. 

\renewcommand{\arraystretch}{1.2}
\begin{table}
\begin{center}
\[
\begin{array}{c|cccc}
 & 1 & b_1 & a_2 & c_2 \\ \hline
 1 & 1 & \tilde{A}_1 & \tilde{A}_1 & (\tilde{A}_1)_2 \\
 b_1 & A_1 & (\tilde{A}_1)_2 & A_1\tilde{A}_1 & A_1\tilde{A}_1 \\
 a_2 & A_1 & A_1\tilde{A}_1 & (\tilde{A}_1)_2 & A_1\tilde{A}_1 \\
 c_2 & (\tilde{A}_1)_2 & A_1\tilde{A}_1  & A_1\tilde{A}_1 & A_1\tilde{A}_1 \\
\end{array}
\]
\caption{The involutions in ${\rm Sp}_{4}(q)^2 < F_4(q)$, $p=2$}
\label{tab:uni}
\end{center}
\end{table}
\renewcommand{\arraystretch}{1}

To summarise, if $x \in G_0$ is a long or short root element, then $|x^G|>q^{16} = b_1$ and there are precisely $4(q^4-1) = a_1$ such elements in $H$. Similarly, if $x$ is in the $G_0$-class $(\tilde{A}_1)_2$ then $|x^G|>q^{22}=b_2$ and we have
\[
|x^G \cap H| = 2|c_2^{{\rm Sp}_{4}(q)}| + |b_1^{{\rm Sp}_{4}(q)}|^2 + |{\rm Sp}_{4}(q)| < q^6(q^4+2q^2+2) = a_2.
\]
Finally, if $x$ is in the $A_1\tilde{A}_1$ class then $|x^G|>q^{28}=b_3$ and $|x^G \cap H|<q^8(q^4+4q^2+2) = a_3$.

Next assume $x \in H_0$ has order $r \geqs 3$, so $|x^G|>(q-1)q^{29}=b_4$. Since 
\[
i_3({\rm Sp}_{4}(q)) \leqs 2\cdot \frac{|{\rm Sp}_{4}(q)|}{|{\rm GL}_{2}(q)|} = 2q^3(q+1)(q^2+1)
\]
it follows that
\[
i_3(H_0) = i_3({\rm Sp}_{4}(q)^2) \leqs (2q^3(q+1)(q^2+1)+1)^2 - 1 < 5q^{12}=a_4
\]
and thus the contribution to $\mathcal{Q}(G,H)$ from semisimple elements of order $3$ is less than $a_4^2/b_4$. 

Now assume $x \in {\rm Sp}_{4}(q)^2$ has prime order $r \geqs 5$. If $\dim x^{\bar{G}} \geqs 42$ then $|x^G|>\frac{1}{2}q^{42}=b_5$ and we note that there are fewer than $|{\rm Sp}_{4}(q)|^2<q^{20}=a_5$ semisimple elements in $H$. Therefore, to complete the analysis for semisimple elements, we may assume $r \geqs 5$ and $\dim x^{\bar{G}}<42$, in which case 
\begin{equation}\label{e:poss}
\mbox{$C_{\bar{G}}(x) = B_2T_2$, $A_2A_1T_1$, $B_3T_1$ or $C_3T_1$.}
\end{equation}
In order to work with the decomposition in \eqref{e:f4adj}, it will be useful to observe that the action of $\tau$ on semisimple elements of $B_2$ is as follows (up to conjugacy, and with respect to the natural module): 
\begin{equation}\label{e:tau}
\tau \,:\, {\rm diag}(\l,\mu,\mu^{-1},\l^{-1}) \mapsto {\rm diag}(\l\mu,\l^{-1}\mu,\l\mu^{-1},\l^{-1}\mu^{-1}).
\end{equation}

Let $i \geqs 1$ be minimal such that $r$ divides $q^i-1$, so $i \in \{1,2,4\}$. First let us consider the  contribution to $\mathcal{Q}(G,H)$ from the elements with $i=4$ and $\dim x^{\bar{G}}<42$. Note that $C_{\bar{G}}(x) = B_2T_2$ is the only option since $r$ divides $q^2+1$ and thus $\dim Z(C_{\bar{G}}(x)) \geqs 2$. In particular, $|x^G|>\frac{1}{2}q^{40} = b_6$. Write $x=(x_1,x_2) \in {\rm Sp}_{4}(q)^2$. Using \eqref{e:tau} and the decomposition in \eqref{e:f4adj}, we deduce that $\dim C_{V}(x) =12$ if and only if $x_1$ and $x_2$ are ${\rm Sp}_{4}(q)$-conjugate, or if $x_j=1$ for some $j$. If $\pi$ denotes the set of primes $s \geqs 5$ dividing $q^2+1$, then $|\pi| < 2\log_2q$ and we deduce that there are at most
\[
2|{\rm Sp}_{4}(q)| + \sum_{r \in \pi} \frac{1}{4}(r-1) \cdot \left(\frac{|{\rm Sp}_{4}(q)|}{q^2+1}\right)^2 < 2q^{10}+\frac{1}{2}\log_2q.q^{10}(q^2-1)^4 = a_6
\]
such elements in $H$. Therefore, the contribution to $\mathcal{Q}(G,H)$ from these elements is less than $a_6^2/b_6$.

Now assume $i \in \{1,2\}$ and $\dim x^{\bar{G}}<42$. As above, write $x = (x_1,x_2) \in {\rm Sp}_{4}(q)^2$ and let us assume $C_{\bar{G}}(x) = B_3T_1$ or $C_3T_1$. By considering the various possibilities for $x_1$ and $x_2$ we deduce that $\dim C_{V}(x) = 22$ if and only if $x_1$ is non-regular and $x_2$ is either trivial or conjugate to $x_1$ (or vice versa). Therefore, if $\a$ denotes the number of such elements in $H$, then 
\[
\a \leqs \sum_{r \in \pi} \frac{1}{2}(r-1) \cdot \left(\left(2\frac{|{\rm Sp}_{4}(q)|}{|{\rm GL}_{2}(q)|}\right)^2 +  4\frac{|{\rm Sp}_{4}(q)|}{|{\rm GL}_{2}(q)|}\right),
\]
where $\pi$ is the set of primes $s \geqs 5$ dividing $q^2-1$. If $q=4$, then $\pi=\{5\}$ and the above bound yields $\a \leqs 236792320$. Now assume $q \geqs 8$. 
Since $r \leqs q+1$, it follows that 
\begin{equation}\label{e:alp}
\a < 2\log_2q.q(4q^6(q+1)^2(q^2+1)^2+4q^3(q+1)(q^2+1)).
\end{equation}
Now $|x^G|>(q-1)q^{29}=b_7$ and so the contribution to $\mathcal{Q}(G,H)$ from the semisimple elements $x \in H$ with $r \geqs 5$ and $C_{\bar{G}}(x) = B_3T_1$ or $C_3T_1$ is less than $a_7^2/b_7$, where $a_7 = 236792320$ if $q=4$ and $a_7$ is the upper bound on $\a$ in \eqref{e:alp} if $q \geqs 8$.

Finally, let us assume $i \in \{1,2\}$ and $C_{\bar{G}}(x) = B_2T_2$ or $A_2A_1T_1$, so $|x^G|>\frac{1}{2}q^{40} = b_8$. First observe that ${\rm Sp}_{4}(q)$ contains at most
\[
\frac{1}{2}(r-1) \cdot 2\frac{|{\rm Sp}_{4}(q)|}{|{\rm GL}_{2}(q)|} < 2(r-1)q^6
\]
non-regular semisimple elements of order $r$. Therefore, there are fewer than $4(r-1)^2q^{12}$ elements in ${\rm Sp}_{4}(q)^2$ of order $r$ of the form $(x_1,x_2)$, where neither $x_1$ nor $x_2$ is regular. Since $r$ divides $q^2-1$, there are less than $8\log_2q.q^{14}$ semisimple elements of this form in $H$. 

Now assume $x = (x_1,x_2) \in {\rm Sp}_{4}(q)^2$ has order $r$ and $x_1$ is regular. By considering \eqref{e:f4adj}, we deduce that $\dim C_{\bar{G}}(x) = 12$ if and only if one of the following holds, where $\sim$ denotes $B_2$-conjugacy:
\begin{itemize}\addtolength{\itemsep}{0.2\baselineskip}
\item[{\rm (a)}] $x_2$ is either trivial or conjugate to $x_1$;
\item[{\rm (b)}] $x_1 \sim {\rm diag}(\l,\l^2,\l^{-2},\l^{-1})$ and $x_2 \sim {\rm diag}(\l, 1,1,\l^{-1})$ or ${\rm diag}(\l,\l,\l^{-1},\l^{-1})$ for some primitive $r$th root of unity $\l$ (with $r=5$ in the latter case).
\end{itemize}
Let $\a$ and $\b$ the total number of elements in $H$ satisfying the conditions in (a) and (b), respectively (allowing for the conditions on $x_1$ and $x_2$ to be interchanged) and let $\pi$ be the set of primes $s \geqs 5$ dividing $q^2-1$. Then  
\[
\a \leqs 2|{\rm Sp}_{4}(q)| + \sum_{r \in \pi}\binom{\frac{1}{2}(r-1)}{2}\left(\frac{|{\rm Sp}_{4}(q)|}{(q-1)^2}\right)^2 < 2q^{10}+2\log_2q.\frac{1}{8}q^2.(2q^{8})^2 = 2q^{10}+\log_2q.q^{18}
\]
and
\[
\b \leqs \sum_{r \in \pi}\frac{1}{2}(r-1) \cdot 2\cdot \frac{|{\rm Sp}_{4}(q)|}{(q-1)^2} \cdot 2\frac{|{\rm Sp}_{4}(q)|}{|{\rm GL}_{2}(q)|} <16\log_2q.q^{15}.
\]

By combining the above estimates, we conclude that $H$ contains fewer than
\[
a_{8} = 2q^{10}+ \log_2q.q^{14}(q^4+16q+8)
\]
semisimple elements of prime order $r \geqs 5$ with $i \in \{1,2\}$ and $C_{\bar{G}}(x) = B_2T_2$ or $A_2A_1T_1$.

To complete the proof of the lemma, we need to estimate the contribution to $\mathcal{Q}(G,H)$ from field and graph automorphisms. First assume $x \in H$ is a field automorphism of order $r$, so $q=q_0^r$ and $|x^G|>\frac{1}{2}q^{52(1-r^{-1})}$. If $r$ is odd then $x$ acts on ${\rm Sp}_{4}(q)^2$ as a field automorphism on both factors, so 
$H$ contains 
\[
(r-1) \cdot \left(\frac{|{\rm Sp}_{4}(q)|}{|{\rm Sp}_{4}(q^{1/r})|}\right)^2 < 4(r-1)q^{20\left(1-r^{-1}\right)}
\]
such elements. Since $|H|< 4\log_2q.q^{20}$, we see that the contribution from odd order field automorphisms is less than $\sum_{i=9}^{11}a_i^2/b_i$, where
\[
a_9 = 8q^{40/3},\; b_9 = \frac{1}{2}q^{104/3},\; a_{10} = 16q^{16},\; b_{10} = \frac{1}{2}q^{208/5},\; a_{11} = 4\log_2q.q^{20},\; b_{11} = \frac{1}{2}q^{312/7}.
\]
Now assume $r=2$, so $|x^G|>q^{26} = b_{12}$ and there are two $H_0$-classes of involutions in the coset $H_0x$. It follows that 
\[
|x^G \cap H| = \left(\frac{|{\rm Sp}_{4}(q)|}{|{\rm Sp}_{4}(q^{1/2})|}\right)^2 + |{\rm Sp}_{4}(q)| < 5q^{10}=a_{12}.
\]

Finally, suppose $x$ is an involutory graph automorphism of $G_0$, so $|x^G|>q^{26}=b_{13}$ and $q=2^{2m+1}$ with $m \geqs 1$. Now $x$ induces a graph automorphism on both ${\rm Sp}_{4}(q)$ factors and we note that there are two classes of involutions in $H_0x$. In particular,
\[
|x^G \cap H| = \left(\frac{|{\rm Sp}_{4}(q)|}{|{}^2B_2(q)|}\right)^2 + |{\rm Sp}_{4}(q)| < 5q^{10}=a_{13}.
\]

Bringing the above estimates together, we conclude that
\[
\mathcal{Q}(G,H) < \sum_{i=1}^{13}a_i^2/b_i < 1
\]
for all $q \geqs 4$ (and the upper bound tends to $0$ as $q$ tends to infinity). The result follows.
\end{proof}

\begin{lem}\label{l:maxrank_f4_3}
If $G_0 = F_4(q)$ and $H$ is of type $C_2(q^2)$, then $b(G,H) = 2$.
\end{lem}

\begin{proof}
Here $p=2$, $H_0 = {\rm Sp}_{4}(q^2).2 < {\rm Sp}_{8}(q) < G_0$ and $G$ contains graph automorphisms. Note that
\begin{equation}\label{e:twist}
(H_0)' = {\rm Sp}_{4}(q^2) = \{ (x,x^{\varphi}) \,:\, x \in {\rm Sp}_{4}(q^2) \} < \bar{H}^0 = B_2^2,
\end{equation}
where $\varphi$ is an involutory field automorphism of ${\rm Sp}_{4}(q^2)$. In particular, the outer involution in $H_0$, which acts as a field automorphism on ${\rm Sp}_{4}(q^2)$, corresponds to an involution in $\bar{H}$ that interchanges the two $B_2$ factors of $\bar{H}^0$. The case $q=2$ can be checked using {\sc Magma} (see \cite[Lemma 2.5]{BTh_comp} for the details of this computation) and so we may assume $q \geqs 4$.

Let $x \in H$ be an element of prime order $r$. First assume $r=2$ and $x \in H_0$.  There are three classes of involutions in ${\rm Sp}_{4}(q^2)$ and it is easy to determine the corresponding class in ${\rm Sp}_{8}(q)$, which in turn allows us to identify the $G_0$-class of $x$ (see \cite[p.373]{LLS}, for example). Using the notation from \cite{AS}, we find that the $b_1$-involutions in ${\rm Sp}_{4}(q^2)$ are contained in the ${\rm Sp}_{8}(q)$-class labelled $c_2$, which is contained in the $G_0$-class $(\tilde{A}_1)_2$. Similarly, the $a_2$-type involutions in ${\rm Sp}_{4}(q^2)$ are also in the $(\tilde{A}_1)_2$ class, while the $c_2$ involutions are in the class $A_1\tilde{A}_1$. In addition, the involutory field automorphisms in $H_0$ embed in ${\rm Sp}_{8}(q)$ as $a_4$-type involutions, so they are contained in the $G_0$-class $(\tilde{A}_1)_2$. 

In conclusion, if $x$ is an involution in the class $(\tilde{A}_1)_2$, then $|x^G|>q^{22}=b_1$ and
\[
|x^G \cap H| = 2(q^8-1) + \frac{|{\rm Sp}_{4}(q^2)|}{|{\rm Sp}_{4}(q)|} < 2q^{10}=a_1.
\]
And for $x$ in the class $A_1\tilde{A}_1$ we have $|x^G \cap H|< q^{12}=a_2$ and $|x^G|>q^{28}=b_2$.

Next assume $x \in G_0$ has order $r \geqs 3$, so $|x^G|>(q-1)q^{29}=b_3$. Since
\[
i_3(H_0) \leqs 2\frac{|{\rm Sp}_{4}(q^2)|}{|{\rm GL}_{2}(q^2)|} = 2q^6(q^2+1)(q^4+1)=a_3,
\]
it follows that the contribution to $\mathcal{Q}(G,H)$ from semisimple elements of order $3$ is less than $a_3^2/b_3$. Now assume $r \geqs 5$. If $\dim x^{\bar{G}} \geqs 42$ then $|x^G|>\frac{1}{2}q^{42}=b_4$ and there are fewer than $q^{20}=a_4$ semisimple elements in $H$, so the contribution from these elements is less than $a_4^2/b_4 = 2q^{-2}$. 

To complete the analysis of semisimple elements, we may assume $r \geqs 5$ and $\dim x^{\bar{G}} < 42$, so the possibilities for $C_{\bar{G}}(x)$ are listed in \eqref{e:poss}. In terms of the embedding of ${\rm Sp}_{4}(q^2)$ in $\bar{H}^0$ (see \eqref{e:twist}), we may write $x = (x_1,x_2) \in \bar{H}^0$, where up to conjugacy we have
\[
x_1 = {\rm diag}(\l,\mu,\mu^{-1},\l^{-1}),\;\; x_2 = x_1^{\varphi} = {\rm diag}(\l^q,\mu^q,\mu^{-q},\l^{-q}).
\]
By expressing $x$ in this form, we can use \eqref{e:f4adj} to determine the dimension of the $1$-eigenspace of $x$ on the adjoint module $V=\mathcal{L}(\bar{G})$. Note that $r$ divides $q^8-1$.

Suppose $C_{\bar{G}}(x) = B_3T_1$ or $C_3T_1$. By considering \eqref{e:f4adj} and the various possibilities for $x_1$ and $x_2$, we deduce that $x_1$ and $x_2$ must be ${\rm Sp}_{4}(q^2)$-conjugate and non-regular, so $r$ divides $q^2-1$. Let $\pi$ be the set of primes $s \geqs 5$ dividing $q^2-1$ and note that $|x^G|>(q-1)q^{29}=b_5$. Then $H$ contains at most
\[
\sum_{r \in \pi}\frac{1}{2}(r-1)\cdot 2\frac{|{\rm Sp}_{4}(q^2)|}{|{\rm GL}_{2}(q^2)|} < f(q).q^6(q^2+1)(q^4+1) = a_5
\]
such elements, where $f(q) = 2\log_2q.q^2$ if $q \geqs 8$ and $f(4) = 4$ (note that $\pi = \{5\}$ if $q=4$).

Now assume $C_{\bar{G}}(x) = B_2T_2$ or $A_2A_1T_1$, so $|x^G|>\frac{1}{2}q^{40}=b_6$. Suppose $x \in {\rm Sp}_{4}(q^2)$ is regular. If $r$ divides $q^4+1$ then by considering the possibilities for $|Z(C_{G_0}(x))|$ we deduce that $C_{\bar{G}}(x) = T_4$, which is a contradiction. Now assume $r$ divides $q^4-1$. If $x_1$ and $x_2$ are not ${\rm Sp}_{4}(q^2)$-conjugate, then by considering the $1$-eigenspace of $x$ on $V$ we deduce that $\dim x^{\bar{G}} \geqs 42$. On the other hand, if $x_1$ and $x_2$ are conjugate (which is always the case if $r$ divides $q^2-1$) then $\dim C_{\bar{G}}(x)=12$. Therefore, if $\pi$ is the set of primes $s \geqs 5$ dividing $q^4-1$, then there are at most 
\[
\sum_{r \in \pi}\binom{\frac{1}{2}(r-1)}{2}\frac{|{\rm Sp}_{4}(q^2)|}{(q^2-1)^2} < 4\log_2q.\frac{1}{8}q^2.q^8(q^2+1)^2(q^4+1) = \frac{1}{2}\log_2q.q^{10}(q^2+1)^2(q^4+1)
\]
such elements in $H$. Finally, by arguing as in the previous paragraph, we calculate that there are fewer than $2a_5$ non-regular semisimple elements in ${\rm Sp}_{4}(q^2)$ of order $r$, where $r \geqs 5$ is a prime divisor of $q^4-1$. We conclude that there are less than
\[
\frac{1}{2}\log_2q.q^{10}(q^2+1)^2(q^4+1)+2a_5 = a_6
\]
semisimple elements $x \in H$ of prime order $r \geqs 5$ with $C_{\bar{G}}(x) = B_2T_2$ or $A_2A_1T_1$.

Now suppose $x \in G$ is a field automorphism of prime order $r$, so $q=q_0^r$. If $r=2$ and $L = \la G_0, x\ra$, then every involution in $N_{L}(H_0) = {\rm Sp}_{4}(q^2).4$ is contained in $H_0$, so we may assume $r$ is odd. Since $x$ induces a field automorphism of order $r$ on ${\rm Sp}_{4}(q^2)$, we see that the contribution to $\mathcal{Q}(G,H)$ from field automorphisms is less than $\sum_{i=7}^{9}a_i^2/b_i$, where
\[
a_7 = 4q^{40/3},\; b_7 = \frac{1}{2}q^{104/3},\; a_{8} = 8q^{16},\; b_{8} = \frac{1}{2}q^{208/5},\; a_{9} = 4\log_2q.q^{20},\; b_{9} = \frac{1}{2}q^{312/7}.
\]
Finally, suppose $x \in G$ is an involutory graph automorphism, so $q=2^{2m+1}$ with $m \geqs 1$. As noted above for involutory field automorphisms, if $G = G_0.2$ then every involution in $H = H_0.2 = {\rm Sp}_{4}(q^2).4$ is contained in $H_0$ (note that ${\rm Out}({\rm Sp}_{4}(q^2))$ is cyclic). In particular, there is no contribution to $\mathcal{Q}(G,H)$ from graph automorphisms.

In view of the above estimates, we conclude that
\[
\mathcal{Q}(G,H) < \sum_{i=1}^{9}a_i^2/b_i
\]
and one checks that this bound is sufficient.
\end{proof}

\begin{lem}\label{l:maxrank_f4_4}
Suppose $G_0=F_4(q)$ and $H$ is of type 
\[
A_1(q)C_3(q) \, (p \ne 2), \; B_4(q), \; D_4(q), \; {}^3D_4(q).
\]
Then $G$ is not extremely primitive. 
\end{lem}

\begin{proof}
If $H$ is of type $A_1(q)C_3(q)$ (with $q$ odd) then $H$ is the centralizer of an involution, so $Z(H) \ne 1$ and $G$ is not extremely primitive by Lemma \ref{l:structure}(i).

In the three remaining cases, the maximality of $H$ implies that $G$ does not contain any graph automorphisms (see \cite[Table 5.1]{LSS}). In other words, we may assume that $G = G_0.A$ and $H = H_0.A$, where $A$ is a group of field automorphisms of $G_0$.

Suppose $H$ is of type $B_4(q)$, so $H_0 = d.\O_9(q)$ with $d=(2,q-1)$ and we may assume $q$ is even (otherwise $Z(H) \ne 1$). To handle this case, we will appeal to Lemma \ref{l:simple}. Let $K = \Omega_8^+(q) < H_0$. By inspecting \cite[Tables 8.58 and 8.59]{BHR}, for example, it is easy to see that $M = N_{H_0}(K) = K.2$ is the only maximal overgroup of $K$ in $H_0$. Since $M$ is a non-normal subgroup of $N_{G_0}(K) = K.{\rm Sym}_3$ (see \cite[Table 5.1]{LSS}), we may choose $g \in N_{G_0}(K)$ such that $M^g \ne M$. If $M^g \leqs H_0$ then $K = K^g < M^g$ and thus $M = M^g$ since $M$ is the unique maximal overgroup of $K$ in $H_0$, which is a contradiction. Therefore, $M^g \not\leqs H_0$ and the result follows by applying Lemma \ref{l:simple} (noting that $K$ and $M$ are both $A$-stable).

Now assume $H$ is of type $D_4(q)$, so $H_0 = d^2.{\rm P\O}_{8}^{+}(q).{\rm Sym}_3$ and once again we may assume $q$ is even. Consider the subgroup $\bar{K} = B_3 < \bar{H}^0 = D_4$ and set 
\[
K = \bar{K}_{\s} = {\rm Sp}_6(q) < S  = (\bar{H}^0)_{\s} = \O_8^{+}(q) <H_0.
\]
Let $\mathcal{M}$ be the set of maximal overgroups of $K$ in $H_0$. By inspecting \cite[Table 8.50]{BHR} we see that $N_{H_0}(K) = K \times Z_2$ and $C_{H_0}(K) = Z_2$. In particular, each $M \in \mathcal{M}$ is of the form $S.2$ (three such groups) or $S.3$. Note that $K$ and each subgroup in $\mathcal{M}$ is $A$-stable. By inspecting \cite[Table 2]{LS94} we see that $C_{\bar{G}}(\bar{K})^0 = A_1$ and thus 
$C_{G_0}(K) \geqs {\rm SL}_{2}(q)$. Therefore, we can choose $g \in C_{G_0}(K) \setminus C_{H_0}(K)$. Then $g$ normalizes $K$, but it does not normalize $S$ (since $H_0 = N_{G_0}(S)$), so $M^g \not\leqs H_0$ for all $M \in \mathcal{M}$. Now apply Lemma \ref{l:simple}.

Finally, let us assume $H$ is of type ${}^3D_4(q)$, so $H_0 = {}^3D_4(q).3$. Write $H_0 = S.\la \tau \ra$, where $S = {}^3D_4(q)$ and $\tau$ is a triality graph automorphism of $S$ with $C_S(\tau) = G_2(q) = K$. Note that $M=N_{H_0}(K) = K \times \la \tau \ra$ and $S$ are the only maximal overgroups of $K$ in $H_0$. Let us also observe that $K$ and $M$ are $A$-stable.

We claim that 
\[
C_{H_0}(K) = \la \tau \ra < {\rm PGL}_{2}(q) \leqs C_{G_0}(K).
\]
To see this, first observe that $K = \bar{K}_\sigma$ for a $\sigma$-stable subgroup $\bar{K} = C_{\bar{G}}(\tau)$ of type $G_2$, which contains long root subgroups of $\bar{G}$. So by \cite[Table~3]{LS94}, $C_{\bar{G}}(\bar{K})^0 = \bar{J}$ is of type $A_1$. In addition, note that if $p \neq 2$ then $\bar{J} \times \bar{K}$ is a maximal subgroup of $\bar{G}$ by \cite[Theorem~1]{LS04} and thus $\bar{J}$ must be of adjoint type. It remains to prove that $\tau \in \bar{J}$, since then $\tau$ will be contained in $\bar{J}_\sigma \cong \text{PGL}_2(q)$. Suppose this is not the case. Then $\tau$ must centralize $\bar{J}$ and so it centralizes $\bar{J} \times \bar{K}$. As noted above, if $p \neq 2$ then $\bar{J} \times \bar{K}$ is maximal, so this is not possible. On the other hand, if $p=2$ then \cite[Table~4.7.1]{GLS} implies that $C_{\bar{G}}(\tau) = B_3 T_1$, $C_3 T_1$ or $A_2 \tilde{A}_2$. But once again we reach a contradiction since none of these groups contain a subgroup of type $A_1 G_2$. This justifies the claim.

Suppose $q \geqs 3$. Then $\la \tau \ra$ is non-normal in ${\rm PGL}_{2}(q)$, so we can choose $g \in C_{G_0}(K) \setminus C_{H_0}(K)$ that does not normalize $Z(M) = \la \tau \ra$. In particular, $g$ does not normalize $M$. If $M^g \leqs H_0$ then $K = K^g < M^g \leqs H_0$ and thus $M = M^g$, which is a contradiction. Similarly, if $S^g \leqs H_0$ then $S = S^g$ and thus $g \in N_{G_0}(S) = H_0$. But we have $g \in C_{G_0}(K) \setminus C_{H_0}(K)$, so once again we have reached a contradiction. Therefore, $M^g \not\leqs H_0$ and $S^g \not\leqs H_0$, so Lemma \ref{l:simple} implies that $G$ is not extremely primitive.

To complete the proof, let us assume $q=2$. The character tables of $G$ and $H$ are available in the \textsf{GAP} Character Table Library \cite{GAPCTL} and we calculate that $G$ has rank $7$ (there are four possible fusion maps from $H$-classes to $G$-classes, but this does not affect the computation of the rank; see \cite[Lemma 2.6]{BTh_comp} for the details). In addition, the indices of the core-free maximal subgroups of $H$ are as follows:
\[
\{n_1, \ldots, n_9\} = \{4064256, 978432, 179712, 163072, 89856, 69888, 17472, 2457, 819\}.
\]
Now $|G:H| = 5222400$ and it is routine to check that there is no tuple $[a_1, \ldots, a_9]$ of non-negative integers such that $\sum_{i}a_i = 6$ and $1+\sum_{i}a_in_i = |G:H|$. Therefore, $G$ is not extremely primitive by Lemma \ref{l:char}.
\end{proof}

\subsection{$G_0=G_2(q)$}\label{ss:mr_g2}

\begin{lem}\label{l:maxrank_g2_1}
If $G_0=G_2(q)$ and $H$ is the normalizer of a maximal torus, then $b(G,H)=2$.
\end{lem}

\begin{proof}
Here $p=3$, $q \geqs 9$ and $G$ contains graph automorphisms (see \cite[Table 5.2]{LSS}). By Corollary \ref{c:root}, there are no root elements in $H$, so $|x^G| \geqs q^3(q+1)(q^3-1)=b_1$ for all $x \in H$ of prime order (see \cite{LieS,Lubeck}) and we note that $|H| \leqs 12(q+1)^2.2\log_3q = a_1$. Therefore, $\mathcal{Q}(G,H) < a_1^2/b_1$ and this upper bound is less than $1$ if $q>9$ (and it is less than $q^{-1}$ for $q>81$). 

To complete the proof, we may assume $q=9$. Here $|x^G| \geqs q^3(q+1)(q^3+1)$ for all $x \in H$ of prime order (minimal if $x$ is an involutory field automorphism) and by arguing as above we reduce to the cases where $H_0 = (q \pm 1)^2.D_{12}$. If $x \in H$ is not an involutory field automorphism, then $|x^G| \geqs (q^6-1)(q^2-1) = b_1$ and we note that $|H| \leqs (q+1)^2.12.4 = a_1$. Now assume $x$ is an involutory field automorphism, so $|x^G| = q^3(q+1)(q^3+1)=b_2$ and there are at most $|H_0x| \leqs (q+1)^2.12 = a_2$ such elements in $H$. Therefore, $\mathcal{Q}(G,H) < a_1^2/b_1 + a_2^2/b_2 < 1$ and the result follows.
\end{proof}

\begin{lem}\label{l:maxrank_g2_2}
If $G_0=G_2(q)$ and $H$ is of type $A_1(q)^2$, then $G$ is not extremely primitive. 
\end{lem}

\begin{proof}
If $q$ is odd then $H$ is the centralizer of an involution, so $Z(H) \ne 1$ and $G$ is not extremely primitive by Lemma \ref{l:structure}(i). For the remainder, let us assume $q$ is even. Write $G=G_0.A$ and $H = H_0.A$, where $A = \la \varphi \ra$ and $\varphi$ is either trivial or a field automorphism.

Let $\a_1,\a_2$ be simple roots for $G_0$ with $\a_1$ short, $\a_2$ long and let $X_{\a}  = \{ x_{\a}(c)\,:\, c \in \mathbb{F}_q\}$ be the root subgroup corresponding to the root $\a$. Then up to conjugacy, we may assume that 
\[
H_0 = \la X_{\pm(3\a_1+2\a_2)}, X_{\pm \a_1} \ra  = {\rm L}_2(q) \times {\rm L}_{2}(q).
\]
Set $g = x_{\a_2}(1)x_{\a_1+\a_2}(1)x_{-\a_2}(1) \in G_2(2) \leqs G_0$. By \cite[Theorem A.1]{BGL}, we have $H_0 \cap H_0^g = 1$ and thus $b(G_0,H_0) = 2$. Moreover, since 
$g \in G_2(2) \leqs C_G(\varphi)$ it follows that $H \cap H^g = \la \varphi \ra$. Clearly, this is not a maximal subgroup of $H$ and thus $G$ is not extremely primitive.
\end{proof}

\begin{lem}\label{l:maxrank_g2_3}
If $G_0=G_2(q)$ and $H$ is of type $A_2^{\e}(q)$, then $G$ is not extremely primitive. 
\end{lem}

\begin{proof}
Here $H_0 = {\rm SL}_{3}^{\e}(q).2$ and the maximality of $H$ implies that $G = G_0.A$ and $H = H_0.A$, where $A$ is a group of field automorphisms (see \cite[Table 5.1]{LSS}). In view of Theorem \ref{t:small}, we may assume that $q \geqs 7$. If $q \equiv \e \imod{3}$ then $F(H)=Z_3$ and extreme primitivity is ruled out by Lemma \ref{l:structure}(iv), so we may assume that $(3,q-\e)=1$. Let $S = {\rm soc}(H_0) = {\rm SL}_{3}^{\e}(q)$.

By inspecting \cite[Tables 8.3 and 8.5]{BHR}, we see that $H_0$ has a maximal subgroup $M = {\rm GL}^\epsilon_2(q).2$. Set 
\[
K = \text{SL}_2(q) < L = \text{GL}^\epsilon_2(q) < M
\]
and let $\mathcal{M}$ be the set of maximal overgroups of $K$ in $H_0$. Note that $N_{H_0}(K) = M$. We claim that $\mathcal{M} =\{M,S\}$. Plainly $S \in \mathcal{M}$ and using \cite{BHR} it is clear that every other subgroup in $\mathcal{M}$ is a conjugate of $M$. Suppose $K$ is contained in $M^h$ for some $h \in H_0$. Then $K, K^{h^{-1}} \leqs M$, but $K$ is the only subgroup of $M$ isomorphic to ${\rm SL}_{2}(q)$, so $K = K^{h^{-1}}$ and thus $h \in N_{H_0}(K) = M$. Therefore, $M = M^h$ and this justifies the claim. Note that $K$, $M$ and $S$ are all $A$-stable. In addition, let us observe that $C_{S}(K) = C_{L}(K) = Z(L)=Z_{q-\epsilon}$ is a characteristic subgroup of $C_{H_0}(K) = C_{M}(K) = Z_{q-\epsilon}.2$ and $M$.  

We claim that $C_S(K) < \text{SL}_2(q) < G_0$. Firstly, we note that $L = \bar{L}_\sigma$ for a $\sigma$-stable subgroup $\bar{L} < \bar{G}$ of type $A_1 T_1$ and so $K = (\bar{L}')_\sigma$. By inspecting \cite[Table 2]{LS94}, we see that $C_{\bar{G}}(\bar{L}')^0 = \bar{J}$, where $\bar{J}$ is a subgroup of type $A_1$, whence $T_1 = Z(\bar{L}) < \bar{J}$ and $C_{S}(K) = Z(\bar{L})_\sigma < \bar{J}_\sigma = \text{SL}_2(q)$, as required. 

By the claim, it follows that $C_{L}(K)$ is a non-normal subgroup of $C_{G_0}(K)$. Therefore, we may choose an element $g \in C_{G_0}(K)$ which does not normalize $C_L(K)$. Suppose $M^g \leqs H_0$. Then $K = K^g < M^g \leqs H_0$ and thus $M = M^g$ (since $M$ and $S$ are the only maximal overgroups of $K$ in $H_0$). But $C_L(K) = Z(L)$ is a characteristic subgroup of $M$, so this would imply that $g$ normalizes $C_L(K)$, which is a contradiction. Similarly, if $S^g \leqs H_0$ then $S^g = S$ and thus $C_S(K)^g = C_S(K)$. But $C_S(K) = C_L(K)$ and $g$ does not normalize $C_L(K)$, so once again we have reached a contradiction. We conclude that $M^g \not\leqs H_0$ and $S^g \not\leqs H_0$, so the desired result follows from Lemma \ref{l:simple}.
\end{proof}

\subsection{$G_0={}^3D_4(q)$}\label{ss:mr_3d4}

\begin{lem}\label{l:maxrank_3d4_1}
If $G_0={}^3D_4(q)$ and $H$ is the normalizer of a maximal torus, then $b(G,H)=2$.
\end{lem}

\begin{proof}
The possibilities for $H$ are described in \cite[Table 5.2]{LSS} and in each case we observe that $H_0 = S.N$, where $S$ is a torus of odd order and $N$ has a unique involution. If $q=2$ then we refer the reader to \cite[Table 12]{BLS}. For the remainder, we may assume $q \geqs 3$.

Suppose $H_0$ contains a long root element $x$. Then $p=2$ by Corollary \ref{c:root}, so $x$ is an involution. But \cite[Corollary 4.4]{BTh} implies that every involution in $H_0$ is contained in the largest class of involutions in $G_0$ and it follows that there are no long root elements in $H_0$. Therefore, Proposition \ref{p:bounds} gives $|x^G|>q^{14}=b_1$ for all $x \in H$ of prime order and we note that 
\[
|H| \leqs (q^2+q+1)^2.|{\rm SL}_{2}(3)|.3\log_{2}q = a_1.
\]
This gives $\mathcal{Q}(G,H) < a_1^2/b_1$, which is less than $1$ if $q \geqs 7$ (and it is less than $q^{-1}$ for $q \geqs 11$).

Now assume $q \in \{3,4,5\}$. If $H_0 = (q^4-q^2+1){:}4$, then we can replace $a_1$ in the previous bound by $3\log_2q.|H_0|$ and this is sufficient. Next suppose $H_0 = (q^2-q+1)^2{:}{\rm SL}_{2}(3)$. If $q=5$ then $|H| \leqs 21^2.|{\rm SL}_{2}(3)|.3 = a_1$ and we get $\mathcal{Q}(G,H) < a_1^2/b_1 < 1$ with $b_1 = 5^{14}$. If $q \in \{3,4\}$ then 
$H = N_G(L)$, where $L$ is a Sylow $(q^2-q+1)$-subgroup of $G_0$, and we can use {\sc Magma} to show that $b(G,H)=2$. Finally, the case 
$H_0 = (q^2+q+1)^2{:}{\rm SL}_{2}(3)$ can be handled in a similar fashion, noting that $H = N_G(L)$ for a Sylow $r$-subgroup of $G_0$ with $r=13$, $7$ and $31$ when $q=3$, $4$ and $5$, respectively. We refer the reader to \cite[Lemmas 2.7 an 2.8]{BTh_comp} for the details of these {\sc Magma} computations.
\end{proof}

\begin{lem}\label{l:maxrank_3d4_2}
Suppose $G_0={}^3D_4(q)$ and $H$ is of type $A_1(q)A_1(q^3)$ or $A_2^{\e}(q).(q^2+\e q+1)$. Then $G$ is not extremely primitive. 
\end{lem}

\begin{proof}
In view of Theorem \ref{t:small}, we may assume $q \geqs 3$. First assume $H$ is of type $A_1(q)A_1(q^3)$, so $H_0 = d.({\rm L}_{2}(q) \times {\rm L}_{2}(q^3)).d$ with $d=(2,q-1)$. If $q$ is odd then $H$ is the centralizer of an involution, so $Z(H) \ne 1$ and $G$ is not extremely primitive. On the other hand, if $q$ is even then the structure of ${\rm soc}(H)$ is incompatible with extreme primitivity. Similarly, if $H$ is of type $A_2^{\e}(q).(q^2+\e q+1)$, then
$H_0 = ({\rm SL}_{3}^{\e}(q) \circ (q^2+\e q+1)).h.2$
with $h = (q^2+\e q +1,3)$ and the result follows since ${\rm soc}(H)$ is not a direct product of isomorphic simple groups.
\end{proof}

\subsection{$G_0={}^2F_4(q)'$}\label{ss:mr_2f4}

\begin{lem}\label{l:maxrank_2f4_1}
If $G_0={}^2F_4(q)'$ and $H$ is the normalizer of a maximal torus, then $b(G,H)=2$.
\end{lem}

\begin{proof}
If $q=2$ then the result follows from \cite[Table 12]{BLS}, so we may assume $q \geqs 8$. By inspecting \cite[Table 5.2]{LSS}, we see that 
\[
|H| \leqs (q+\sqrt{2q}+1)^2.|4 \circ {\rm GL}_2(3)|.2\log_2q = a_1.
\]
Now $|x^G|>(q-1)q^{10} = b_1$ for all $x \in G$ of prime order, so $\mathcal{Q}(G,H) < a_1^2/b_1$ and one checks that this upper bound is less than $1$ for all $q \geqs 8$ (in addition, it is less than $q^{-1}$ for all $q \geqs 32$).
\end{proof}

\begin{lem}\label{l:maxrank_2f4_2}
If $G_0={}^2F_4(q)'$ and $H$  is of type $A_2^{-}(q)$, then $b(G,H)=2$.
\end{lem}

\begin{proof}
Here $H_0 = {\rm SU}_{3}(q).2$ or ${\rm PGU}_{3}(q).2$. In view of Remark \ref{r:small}, we may assume $q \geqs 8$. Let $x \in H$ be an element of prime order $r$.

First assume $r=2$, so $x \in H_0$ and $|C_{H_0}(x)|$ is divisible by $3$. In the notation of \cite[Table II]{Shin}, it follows that each involution in $H_0$ is $G_0$-conjugate to $u_2$, whence $|x^G|>(q-1)q^{13}=b_1$ and 
\[
|x^G \cap H| = i_2(H) = \frac{|{\rm SU}_{3}(q)|}{q^3(q+1)^2}+\frac{|{\rm SU}_{3}(q)|}{|{\rm SL}_{2}(q)|} < 2q^5 = a_1.
\]
Similarly, if $x \in H_0$ has order $3$ then $|x^G|>(q-1)q^{17}=b_2$ and since $G_0$ has a unique class of elements of order $3$, it follows that
\[
|x^G \cap H| = i_3(H_0) \leqs i_3({\rm PGU}_{3}(q)) < 2(q+1)q^5 = a_2.
\]
If $x \in H_0$ has order $r \geqs 5$ then $|x^G|>\frac{1}{2}q^{20} = b_3$ (minimal if $x$ is conjugate to the element denoted $t_9$ in \cite[Table IV]{Shin}) and we record the bound $|{\rm SU}_{3}(q)|<q^8 = a_3$.

Finally, suppose $x \in H$ is a field automorphism of order $r$. If $r=3$ then $|x^G|>\frac{1}{2}q^{52/3}=b_4$ and $H$ contains fewer than 
\[
2\cdot \frac{|{\rm SU}_{3}(q)|}{|{\rm SU}_{3}(q^{1/3})|}< 4q^{16/3}=a_4
\]
such elements. For $r \geqs 5$, $|x^G|>\frac{1}{2}q^{104/5} = b_5$ and we note that $|H|<2\log_2q.q^{8} = a_5$.

We conclude that $\mathcal{Q}(G,H)<\sum_{i=1}^{5}a_i^2/b_i < q^{-1}$ and the result follows.
\end{proof}

\begin{lem}\label{l:maxrank_2f4_4}
If $G_0={}^2F_4(q)'$ and $H$  is of type $C_2(q)$, then $b(G,H)=2$.
\end{lem}

\begin{proof}
Here $H_0 = {\rm Sp}_{4}(q).2$, where the outer involution acts as a graph automorphism on ${\rm Sp}_{4}(q)$. For $q=2$, we refer the reader to \cite[Table 12]{BLS}. For the remainder, let us assume $q \geqs 8$.

It will be convenient to view $H_0$ as a subgroup of $F_4(q)$. Then in terms of the ambient algebraic groups, we have 
\begin{equation}\label{e:b2}
B_2 = \{ (x,x^{\tau}) \,:\, x \in B_2\} < B_2^2 < B_4 < F_4,
\end{equation}
where $\tau$ is an involutory graph automorphism of $B_2$. Let $V$ be the adjoint module for $\bar{G}$ and note that the restriction of $V$ to $B_2^2$ is given in \eqref{e:f4adj}.

Let $x \in H$ be an element of prime order $r$. First assume $r=2$, so $x \in H_0$. If $x$ acts as a graph automorphism on ${\rm Sp}_{4}(q)$, then $C_{{\rm Sp}_{4}(q)}(x) = {}^2B_2(q)$ and we deduce that $x$ is $G_0$-conjugate to $u_1$ (see \cite[Table II]{Shin}). Similarly, if $x$ is a long or short root element in ${\rm Sp}_{4}(q)$, then $x$ is conjugate to $u_2$. Now assume $x \in {\rm Sp}_{4}(q)$ is a $c_2$-type involution, in the notation of \cite{AS}. From the embedding $B_2 < B_4$ in \eqref{e:b2}, we see that $x$ is in the $B_4$-class labelled $c_4$, and by considering the fusion of $B_4$-classes of involutions in $F_4$ (see \cite[p.373]{LLS}, for example), we deduce that $x$ is in the $F_4$-class labelled $A_1\tilde{A}_1$. It follows that $x$ is $G_0$-conjugate to $u_1$ (the involutions in ${}^2F_4(q)$ conjugate to $u_2$ are contained in the $F_4$-class $(\tilde{A}_1)_2$). To summarise: if $x$ is $G_0$-conjugate to $u_1$, then
\[
|x^G\cap H| = \frac{|{\rm Sp}_{4}(q)|}{|{}^2B_2(q)|} < q^5 = a_1, \;\; |x^G|>(q-1)q^{10}=b_1,
\]
whereas 
\[
|x^G \cap H| = 2(q^4-1)+(q^2-1)(q^4-1)<2q^6 = a_2,\;\; |x^G|>(q-1)q^{13}=b_2
\]
if $x$ is conjugate to $u_2$.

Next assume $x \in H_0$ has order $3$. Here $|x^G|>(q-1)q^{17}=b_3$ and 
\[
|x^G\cap H| = i_3(H_0) = 2q^3(q^2+1)(q-1) < 2q^6 = a_3.
\]

Now suppose $x \in H_0$ has order $r \geqs 5$, so \cite[Table IV]{Shin} indicates that $C_{\bar{G}}(x) = B_2T_2$, $A_1\tilde{A}_1T_2$ or $T_4$. It will be useful to recall the action of $\tau$ on semisimple elements of $B_2$ in \eqref{e:tau}.

If $x$ is not regular in ${\rm Sp}_{4}(q)$ then by working with the decomposition in \eqref{e:f4adj}, we calculate that $\dim C_{V}(x) = 12$ and thus $\dim C_{\bar{G}}(x) = 12$. For example, if $r$ divides $q-1$ and $x = {\rm diag}(1,1,\omega, \omega^{-1}) \in {\rm Sp}_{4}(q)$, then $x^{\tau} = {\rm diag}(\omega,\omega,\omega^{-1},\omega^{-1})$ and we find that $x$ has an $8$-dimensional $1$-eigenspace on $\mathcal{L}(B_2B_2)$ and a $4$-dimensional $1$-eigenspace on $W \otimes W$, where $W$ is the natural $4$-dimensional module for $B_2$. Furthermore, $x^{\tau}$ acts on $W^{\tau}$ as ${\rm diag}(1,1,\omega^2,\omega^{-2})$ and thus $\dim C_{W^{\tau} \otimes W^{\tau}}(x) = 0$, giving $\dim C_{V}(x) = 12$ as claimed. It follows that $C_{\bar{G}}(x) = B_2T_2$ and by inspecting the relevant tables in \cite{Shin2,Shin} we deduce that $|x^G|>\frac{1}{2}q^{20}=b_4$ (note that $x$ is of type $t_1$, $t_7$ or $t_9$ in \cite[Table IV]{Shin}). Similarly, we find that $C_{\bar{G}}(x) = B_2T_2$ if $r=5$ (note that $5$ divides $q^2+1$, so every element of order $5$ is regular). Since $r$ must divide $q^2-1$ if $x$ is non-regular, it follows that there are at most 
\[
\frac{|{\rm Sp}_{4}(q)|}{q^2+1} + \sum_{r \in \pi}\frac{1}{2}(r-1) \cdot 2\cdot \frac{|{\rm Sp}_{4}(q)|}{|{\rm GL}_{2}(q)|} < q^8+8\log_2q.q^7 = a_4
\]
semisimple elements in $H_0$ with $C_{\bar{G}}(x) = B_2T_2$, where $\pi$ is the set of primes $s \geqs 5$ dividing $q^2-1$.

Now assume $x \in H_0$ is a regular semisimple element with $r \geqs 7$. Then by considering \eqref{e:f4adj}, we deduce that $\dim C_{\bar{G}}(x) \leqs 8$ and thus $C_{\bar{G}}(x) = A_1\tilde{A}_1T_2$ or $T_4$. This implies that $|x^G|>\frac{1}{2}q^{22}=b_5$ and we note that $|{\rm Sp}_{4}(q)|<q^{10}=a_5$.

Finally, let us assume $x \in G$ is a field automorphism of prime order $r$, so $q=q_0^r$, $r \geqs 3$ and $x$ acts as a field automorphism on ${\rm Sp}_{4}(q)$. If $r=3$ then $|x^G|>\frac{1}{2}q^{52/3}=b_6$ and there are fewer than $2|{\rm Sp}_{4}(q):{\rm Sp}_{4}(q^{1/3})|<4q^{20/3}=a_6$ such elements in $H$. Similarly, if $r=5$ then $|x^G|>\frac{1}{2}q^{104/5}=b_7$ and $H$ contains fewer than $8q^8 = a_7$ such elements. For $r \geqs 7$ we have $|x^G|>\frac{1}{2}q^{156/7}=b_8$ and we note that $|H|<2\log_2q.q^{10}=a_8$. 

Bringing together the above bounds, we conclude that $\mathcal{Q}(G,H)<\sum_{i=1}^8a_i^2/b_i< 1$ and the result follows.
\end{proof}

\begin{lem}\label{l:maxrank_2f4_3}
If $G_0={}^2F_4(q)'$ and $H$  is of type ${}^2B_2(q)^2$, then $b(G,H)=2$.
\end{lem}

\begin{proof}
Here $H_0 = {}^2B_2(q) \wr {\rm Sym}_2$ and $q \geqs 8$ (see \cite[Table 5.1]{LSS}). As in the previous case, it will be useful to view $H_0$ as a subgroup of $F_4(q)$ via $B_2^2 < B_4 < F_4$. Let $x \in H$ be an element of prime order $r$ and let $V = \mathcal{L}(\bar{G})$ be the adjoint module for $\bar{G}=F_4$.

First assume $r=2$, so $x \in H_0$ and we note that ${}^2B_2(q)$ contains $(q^2+1)(q-1)$ involutions, which form a single conjugacy class. If $x$ interchanges the two ${}^2B_2(q)$ factors, then $C_{H_0}(x)$ contains ${}^2B_2(q)$ and thus $x$ is $G_0$-conjugate to $u_1$ in the notation of \cite[Table II]{Shin}. Similarly, each involution in ${}^2B_2(q)^2$ of the form $(x_1,1)$ or $(1,x_1)$ is of type $u_1$. However, if $x = (x_1,x_2) \in {}^2B_2(q)^2$ and each $x_i$ is an involution, then $x$ is a $c_4$-type involution in $B_4$ (since each $x_i$ is a $c_2$-type involution in $B_2$) and thus $x$ is in the $\bar{G}$-class labelled $A_1\tilde{A}_1$. In particular, we deduce that $x$ is $G_0$-conjugate to $u_2$. It follows that if $x$ is of type $u_1$, then
\[
|x^G \cap H| = 2(q^2+1)(q-1)+|{}^2B_2(q)|<q^5 = a_1,\;\; |x^G|>(q-1)q^{10}=b_1
\]
and for $x$ of type of $u_2$ we get $|x^G \cap H|<q^6=a_2$ and $|x^G|>(q-1)q^{13}=b_2$.

Next assume $x \in H_0$ and $r \geqs 5$ (observe that $H_0$ does not contain any elements of order $3$).  Write $x=(x_1,x_2) \in {}^2B_2(q)^2$ and note that if $x_i \ne 1$ then $x_i$ is a regular element of $B_2$ (since $x_i \in C_{B_2}(\tau)$, this follows from \eqref{e:tau}). If $x_2=1$ then using \eqref{e:f4adj} we deduce that $\dim C_{V}(x) = 12$, so $C_{\bar{G}}(x) = B_2T_2$. Similarly, if $x_1$ and $x_2$ are conjugate, then $C_{\bar{G}}(x) = B_2T_2$. On the other hand, if the $x_i$ are nontrivial and non-conjugate, then $C_{\bar{G}}(x) = A_1\tilde{A}_1T_2$ or $T_4$, so $|x^G|>\frac{1}{2}q^{22}=b_3$ and we note that $|{}^2B_2(q)^2|<q^{10}=a_3$. 

Now assume $C_{\bar{G}}(x) = B_2T_2$, so $x$ is of type $t_1$, $t_7$ or $t_9$ with respect to the notation in \cite[Table IV]{Shin}, where
$|C_{G_0}(t_i)| =  q^2(q^2+1)(q-1)f_i(q)$ and 
\[
f_1(q)=q-1,\;\; f_7(q) = q-\sqrt{2q}+1,\;\; f_9(q) = q+\sqrt{2q}+1.
\]
In particular, $|x^G|>q^{20}$ if $x$ is of type $t_1$ or $t_7$, whereas $|x^G|>\frac{1}{2}q^{20}$ if $x$ is of type $t_9$. Let us also observe that if $y \in {}^2B_2(q)$ has order $r$, then $|C_{{}^2B_2(q)}(y)| = q-1$ or $q \pm \sqrt{2q}+1$ (see \cite{Suz}). Moreover, there are $\frac{1}{2}(q-2)$ distinct ${}^2B_2(q)$-classes of semisimple elements with centralizer of order $q-1$, and $\frac{1}{4}(q \pm \sqrt{2q})$ classes with a centralizer of order $q\pm \sqrt{2q}+1$. By Lagrange's theorem, it follows that if $x = (x_1,x_2) \in {}^2B_2(q)^2$ and the $x_i$ are conjugate elements of order $r$, then $x$ is of type $t_1$ if $|C_{{}^2B_2(q)}(x_i)| = q-1$. Similarly, $x$ is of type $t_7$ if  $|C_{{}^2B_2(q)}(x_i)| = q-\sqrt{2q}+1$, and type $t_9$ if $|C_{{}^2B_2(q)}(x_i)| = q+\sqrt{2q}+1$. Therefore, if $x$ is of type $t_1$ then $|x^G|>q^{20}=b_4$ and there are at most
\[
2|{}^2B_2(q)| + \frac{1}{2}(q-2)\cdot (q^2(q^2+1))^2 = a_4
\]
such elements in $H_0$. Similarly, if $x$ is of type $t_7$ then $|x^G|>q^{20}=b_5$ and $H_0$ contains no more than 
\[
2|{}^2B_2(q)| + \frac{1}{4}(q-\sqrt{2q})\cdot (q^2(q-1)(q+\sqrt{2q}+1))^2 = a_5
\]
such elements. Finally, if $x$ is of type $t_9$ then $|x^G|>\frac{1}{2}q^{20}=b_6$ and there are at most
\[
2|{}^2B_2(q)| + \frac{1}{4}(q+\sqrt{2q})\cdot (q^2(q-1)(q-\sqrt{2q}+1))^2 = a_6
\]
such elements in $H_0$.

Finally, suppose $x \in G$ is a field automorphism of prime order $r$, so $q=q_0^r$, $r \geqs 3$ and $x$ acts as a field automorphism on both ${}^2B_2(q)$ factors. 
If $r=3$ then $|x^G|>\frac{1}{2}q^{52/3}=b_7$ and there are at most $2|{}^2B_2(q):{}^2B_2(q^{1/3})|^2<8q^{20/3}=a_7$ such elements in $H$. Similarly, if $r=5$ then $|x^G|>\frac{1}{2}q^{104/5}=b_8$ and $H$ contains fewer than $16q^8 = a_8$ such elements. Finally, $|x^G|>\frac{1}{2}q^{156/7}=b_9$ if $r \geqs 7$ and we note that $|H|<2\log_2q.q^{10}=a_9$. 

We conclude that
\[
\mathcal{Q}(G,H) < \sum_{i=1}^{9}a_i^2/b_i < 1
\]
and the result follows. 
\end{proof}

\subsection{$G_0={}^2G_2(q)'$ or ${}^2B_2(q)$}\label{ss:mr_2g2}

\begin{lem}\label{l:maxrank_2g2_1}
Suppose $G_0={}^2G_2(q)'$ or ${}^2B_2(q)$ and $H$ is the normalizer of a maximal torus. Then $b(G,H)=2$.
\end{lem}

\begin{proof}
See \cite[Lemmas 4.37 and 4.39]{BLS}. In both cases, it is easy to check that the upper bound on $\mathcal{Q}(G,H)$ given in \cite{BLS} tends to $0$ as $q$ tends to infinity.
\end{proof}

\begin{lem}\label{l:maxrank_2g2_2}
If $G_0={}^2G_2(q)'$ and $H$  is of type $2 \times {\rm L}_{2}(q)$, then $G$ is not extremely primitive. 
\end{lem}

\begin{proof}
Here $q>3$ and $H$ is the centralizer of an involution, so $Z(H) \ne 1$ and $G$ is not extremely primitive by Lemma \ref{l:structure}(i).
\end{proof}

\vs

We have now completed the proof of Theorem \ref{t:maxrank} and Proposition \ref{p:mr_base}.

\section{Lower rank subgroups}\label{s:part1}

At this point, we have now established Theorem \ref{t:main} in the cases where $H$ contains a maximal torus of $G_0$ and we will now consider the remaining possibilities for $H$. It will be convenient to postpone the analysis of the twisted groups ${}^2B_2(q)$, ${}^2G_2(q)$, ${}^2F_4(q)$ and ${}^3D_4(q)$ to Section \ref{s:part4}, so in the next three  sections we will assume 
\begin{equation}\label{e:T}
G_0 \in \{E_8(q), E_7(q), E_6^{\e}(q), F_4(q), G_2(q)'\},
\end{equation}
where $q=p^f$ with $p$ a prime. 

For $G_0 = G_2(q)'$, the maximal subgroups of $G$ have been determined up to conjugacy by Cooperstein \cite{Coop} (for $p=2$) and Kleidman \cite{K88} (for $p \ne 2$). In the remaining cases, there is a complete description of the maximal subgroups when $G_0 = F_4(2)$, $E_6(2)$, ${}^2E_6(2)$ or $E_7(2)$; see \cite{NW}, \cite{KW}, \cite{ATLAS, Wil2} and \cite{BBR}, respectively. 

To proceed in the general case, we will apply the following fundamental result (see \cite[Theorem 2]{LS90}), which partitions the remaining maximal subgroups into various types.

\begin{thm}\label{t:types}
Let $G$ be an almost simple group with socle $G_0 = (\bar{G}_{\s})'$ as in \eqref{e:T}. Let $H$ be a maximal subgroup of $G$ with $G = HG_0$ and assume $H$ does not contain a maximal torus of $G_0$. Set $H_0 = H \cap G_0$. Then one of the following holds:
\begin{itemize}\addtolength{\itemsep}{0.2\baselineskip}
\item[{\rm (I)}] $H = N_G(\bar{H}_{\s})$, where $\bar{H}$ is a maximal closed $\sigma$-stable positive dimensional subgroup of $G$ (not parabolic nor maximal rank);
\item[{\rm (II)}] $H$ is of the same type as $G$ (possibly twisted) over a subfield of $\mathbb{F}_{q}$;
\item[{\rm (III)}] $H$ is an exotic local subgroup (determined in \cite{CLSS});
\item[{\rm (IV)}] $G_0=E_8(q)$, $p \geqs 7$ and $H_0 = ({\rm Alt}_5 \times {\rm Alt}_6).2^2$;
\item[{\rm (V)}] $H$ is almost simple, and not of type (I) or (II).
\end{itemize}
\end{thm}

In view of this result, for the remainder of the paper we will refer to Type I and Type II subgroups of $G$, etc. We direct the reader to Theorem \ref{t:simples} for more detailed information on Type V subgroups. 

We begin by handling the subgroups of Type III and IV, which are easily eliminated.

\begin{thm}\label{t:34}
If $H$ is a Type III or IV subgroup of $G$, then $G$ is not extremely primitive.
\end{thm}

\begin{proof}
In view of Lemma \ref{l:structure} (specifically, parts (ii) and (v)), we may assume that one of the following holds, where $H_0 = H \cap G_0$:
\begin{itemize}\addtolength{\itemsep}{0.2\baselineskip}
\item[{\rm (a)}] $G_0 = G_2(q)$, $H_0 = 2^3.{\rm SL}_{3}(2)$, $q=p \geqs 3$; 
\item[{\rm (b)}] $G_0 = F_4(q)$, $H_0 = 3^3.{\rm SL}_{3}(3)$, $p \geqs 5$; 
\item[{\rm (c)}] $G_0 = E_8(q)$, $H_0 = 5^3.{\rm SL}_{3}(5)$, $p \ne 2,5$. 
\end{itemize}

First consider case (a). For $q \in \{3,5\}$, we appeal to Theorem \ref{t:small} (note that $b(G,H) = 3$ when $q=3$). Now assume $q \geqs 7$, so $G=G_2(q)$ and $H = 2^3.{\rm SL}_{3}(2)$. We claim that $b(G,H)=2$. For $q=7$ we use {\sc Magma} to construct $G$ as a permutation group of degree $19608$ and we construct $H = N_G(L)$ by identifying an appropriate elementary abelian subgroup $L$ of order $8$ in a Sylow $2$-subgroup of $G$. It is then routine to find an element $g \in G$ by random search such that $H \cap H^g = 1$ (see \cite[Lemma 2.9]{BTh_comp} for more details). Now assume $q \geqs 11$ and let $x \in H$ be an element of prime order $r$, so $r \in \{2,3,7\}$ and we note that $i_2(H) = 91=a_1$, $i_3(H)=224=a_2$ and $i_7(H)=384=a_3$. If $r=2$ then $|x^G| = q^4(q^4+q^2+1)=b_1$ and similarly $|x^G| \geqs q^3(q^3-1)=b_2$ if $r=3$ and $|x^G|>\frac{1}{2}q^{10} = b_3$ if $r=7$. Therefore,
\[
\mathcal{Q}(G,H) < \sum_{i=1}^{3}a_i^2/b_i < q^{-1}
\]
and the result follows. 

In cases (b) and (c) we also claim that $b(G,H) =2$. For example, in (c) we note that 
$|H| \leqs \log_2q.5^3|{\rm SL}_{3}(5)| = a_1$ and $|x^G|>q^{58}=b_1$ for all nontrivial $x \in G$, which yields $\mathcal{Q}(G,H) < a_1^2/b_1< q^{-1}$. Case (b) is entirely similar, noting that $|x^G|>q^{16}$ for all nontrivial $x \in G$ (see Proposition \ref{p:bounds}).
\end{proof}

For the remainder of this section, we will focus on the subgroups arising in part (I) of Theorem \ref{t:types}. Following \cite[Theorem 8]{LS03}, we partition these subgroups into three cases:
\begin{itemize}\addtolength{\itemsep}{0.2\baselineskip}
\item[{\rm (a)}] $G_0 = E_7(q)$, $p \geqs 3$ and $\bar{H}_{\s} = (2^2 \times {\rm P\Omega}_{8}^{+}(q).2^2).{\rm Sym}_3$ or ${}^3D_4(q).3$;
\item[{\rm (b)}] $G_0 = E_8(q)$, $p \geqs 7$ and $\bar{H}_{\s} = {\rm PGL}_{2}(q) \times {\rm Sym}_{5}$; \hfill\refstepcounter{equation}(\theequation)\label{e:list0}
\item[{\rm (c)}] $(G_0, {\rm soc}(H_0))$ is one of the cases listed in \cite[Table 3]{LS03}.
\end{itemize}

Our main theorem is the following.

\begin{thm}\label{t:typeI}
If $H$ is a Type I subgroup of $G$, then $G$ is not extremely primitive.
\end{thm}

Suppose $G$ is extremely primitive and $H$ is a Type I subgroup of $G$. By Lemma \ref{l:structure}(v), the socle of $H$ is a direct product of isomorphic simple groups. Therefore, by considering the groups in (a) and (b) above, together with the cases in \cite[Table 3]{LS03}, we deduce that the only subgroups of this type are the ones listed in Table \ref{tab:typeI}.

\begin{table}
\begin{center}
$$\begin{array}{ll}\hline
G_0 & {\rm soc}(H_0) \\ \hline
E_8(q) & {\rm L}_{2}(q) \, (\mbox{$3$ classes, $p \geqs 23,29,31$}), \Omega_5(q) \, (p \geqs 5) \\

E_7(q) &  {\rm L}_{2}(q) \, (\mbox{$2$ classes, $p \geqs 17,19$}),  {\rm L}_{3}^{\e}(q) \, (p \geqs 5), {\rm L}_{2}(q)^2 \, (p \geqs 5), {}^3D_4(q) \, (p \geqs 3) \\

E_6^{\e}(q) & {\rm L}_{3}^{\pm}(q) \, (p \geqs 5),  G_2(q) \, (p \ne 7), {\rm PSp}_{8}(q) \, (p \geqs 3), F_4(q),\,  \\

F_4(q) &  {\rm L}_{2}(q) \, (p \geqs 13),  G_2(q) \, (p=7)  \\

G_2(q) &  {\rm L}_{2}(q) \, (p \geqs 7) \\ \hline
\end{array}$$
\caption{Type I subgroups with a compatible socle}
\label{tab:typeI}
\end{center}
\end{table}

First we handle two special cases with $G_0 = E_6^{\e}(q)$.

\begin{lem}\label{t:e6typeI_1}
Suppose $G_0 = E_6^{\e}(q)$ and ${\rm soc}(H) = F_4(q)$ or ${\rm PSp}_{8}(q)$. Then $G$ is not extremely primitive.
\end{lem}

\begin{proof}
Set $H_0 = H \cap G_0$ and first assume ${\rm soc}(H) = F_4(q)$, so $H_0 = F_4(q)$ is the centralizer in $G_0$ of a graph automorphism. Note that ${\rm Inndiag}(G_0) \cap G = G_0$ by the maximality of $H$ in $G$. In addition, if $G$ contains a graph automorphism, then ${\rm soc}(H)$ is not a direct product of isomorphic simple groups and thus $G$ is not extremely primitive by Lemma \ref{l:structure}(v). Write $G = G_0.A$ and $H = H_0.A$.
Set $e=(3,q-\e)$.

The action of the quasisimple group $e.G_0$ on the set of cosets of $F_4(q)$ is studied in some detail by Lawther in \cite{Law93} and he computes all the subdegrees (see \cite[Table 1]{Law93} for $\e=+$ and \cite[Table 3]{Law93} for $\e=-$). In both cases, there is a subdegree $s(q) = q^{4}(q^8-1)(q^{12}-1)$. This is also a subdegree for the action of $G_0$, so there exists $g \in G_0$ such that 
\[
|H_0 \cap H_0^g| = |H_0| / s(q) = q^{20}(q^2-1)(q^6-1).
\]

Suppose $G$ is extremely primitive, so $M = H \cap H^g$ is a maximal subgroup of $H$ such that $|M \cap H_0| = |H_0 \cap H_0^g|$. First we observe that $M$ is non-parabolic since $|H_0 \cap H_0^g|$ is indivisible by $q^{24}$. Then since $|M|>q^{22}$, it follows that $M$ is one of the subgroups listed in \cite[Lemma 4.23]{BLS}. But none of these subgroups have order divisible by $q^{20}$, so we have reached a contradiction and 
we conclude that $G$ is not extremely primitive. 

To complete the proof, we may assume $p$ is odd and ${\rm soc}(H) = {\rm PSp}_{8}(q)$. As in the previous case, $H_0 = {\rm PGSp}_{8}(q) = \text{PSp}_8(q).2$ is the centralizer in $G_0$ of an involutory graph automorphism and we may assume that ${\rm Inndiag}(G_0) \cap G_0 = G_0$ and $G$ does not contain any graph automorphisms. Once again, write $G = G_0.A$ and $H = H_0.A$.

Set $\bar{H} = C_4 < \bar{G} = E_6$ and let $\bar{L}$ be a Levi subgroup of $\bar{H}$ of type $A_3T_1$. Then $N_{\bar{H}}(\bar{L}) = \bar{L}.2$ and the outer involution induces a graph automorphism on $\bar{K} = \bar{L}' = A_3$ and inverts the torus $Z(\bar{L})^0=T_1$. It is straightforward to check that $C_{\bar{H}}(\bar{K}) = Z(\bar{L})^0$. We claim that 
$\bar{K}.2$ is a subgroup of a maximal subsystem subgroup $A_1A_5$ of $\bar{G}$. 
By inspecting \cite[Table 8.3]{LS96}, we see that $\bar{G}$ has two conjugacy classes of $A_3$ subgroups, both of which are contained in a maximal subsystem subgroup $A_1A_5$. The subgroups in the two classes differ in the structure of their connected centralizers in $\bar{G}$, which is either $A_1$ or $A_1^2T_1$. By considering the composition factors of the $A_3$ subgroups on the adjoint module for $\bar{G}$ (see \cite[Table 8.3]{LS96}), we deduce that 
$C_{\bar{G}}(\bar{K})^0 = \bar{J}$ is of type $A_1$. It follows that $\bar{K}$ is contained in the $A_5$ factor of a subsystem subgroup $A_1A_5$, with $\bar{K}$ acting irreducibly on the natural module for $A_5$. In particular, $\bar{K}.2 < A_5 < A_1A_5$ and $C_{\bar{G}}(\bar{K}.2)^0 = \bar{J}$, which contains $Z(\bar{L})^0$ as a maximal torus.  

As usual, in order to descend from the algebraic groups discussed above to the corresponding finite groups we will take the fixed points of an appropriate Steinberg endomorphism $\s$. Specifically, we define $\s$ to be the product of the standard Frobenius map and a lift of the central involution in the Weyl group of $\bar{H}$. Referring to the subgroups introduced above, we may assume that $\bar{H}$, $\bar{L}$, $\bar{K}$ and $\bar{J}$ are all $\s$-stable. Then $(\bar{G}_{\s})' = G_0$, $\bar{H}_{\s} = H_0$ and $\bar{K}_\sigma = h.\text{U}_4(q).h$, where $h = (4,q+1)/2$. Set $K = (\bar{K}_{\s})' = h.\text{U}_4(q)$.
By considering \cite[Tables 8.48 and 8.49]{BHR}, we deduce that 
\[
M = N_{H_0}(K) = h.((q+1)/2h \times K).2^2.2
\]
is the unique maximal overgroup of $K$ in $H_0$. 

Now $C_{H_0}(K) = (Z(\bar{L})^0)_\sigma = Z_{q+1}$ is contained in the subgroup $J = \bar{J}_\sigma  = \text{SL}_2(q)$ of $C_{G_0}(K)$. Since $C_{H_0}(K)$ is non-normal in $J$, we can choose $g \in J$ which does not normalize $C_{H_0}(K)$. If $M^g \leqs H_0$ then $M^g$ is a maximal subgroup of $H_0$ containing $K$, so $M^g = M$ and thus $g$ normalizes $C_M(K) = C_{H_0}(K)$, which is a contradiction. Therefore, $M^g \not\leqs H_0$ and we complete the proof by applying Lemma \ref{l:simple}, noting that both $K$ and $M$ are $A$-stable.
\end{proof}

In each of the remaining cases, we claim that $b(G,H)=2$.

\begin{prop}\label{p:typeIbase}
If $(G,H)$ is one of the cases in Table \ref{tab:typeI} and $(G_0,{\rm soc}(H)) \ne (E_6^{\e}(q),F_4(q))$, $(E_6^{\e}(q),{\rm PSp}_{8}(q))$, then $b(G,H) = 2$. 
\end{prop}

We will prove Proposition \ref{p:typeIbase} in a sequence of lemmas.

\begin{lem}\label{l:typeI_1}
If $G_0 = E_7(q)$ and ${\rm soc}(H) = {}^3D_4(q)$, then $b(G,H) = 2$.
\end{lem}

\begin{proof}
Here $p \geqs 3$ and $H_0 = H \cap G_0 = {}^3D_4(q).3$. We proceed by estimating the various contributions to $\mathcal{Q}(G,H)$ from elements of prime order.

Let $x \in H_0$ be an element of prime order $r$. If $r=2$ then $|x^G \cap H| = q^8(q^8+q^4+1) = a_1$ and $|x^G|>\frac{1}{2}(q-1)q^{53} = b_1$ (note that $H_0$ has a unique conjugacy class of involutions). Similarly, if $x$ is unipotent then $x$ is not a long root element of $G_0$ (see the proof of \cite[Proposition 5.12]{BGS}), so $|x^G|>q^{52} = b_2$ and we note that $H_0$ contains fewer than $q^{24}=a_2$ elements of order $p$ (see Proposition \ref{p:i23}(ii)). 

Next assume $x \in H_0$ and $r \ne 2,p$. If $r=3$ then $|x^G|>(q-1)q^{53}=b_3$ and Proposition \ref{p:i23}(i) gives $i_3(H_0)< 2(q+1)q^{19}=a_3$. Similarly, if $r \geqs 5$ and $C_{\bar{G}}(x)^0 \ne E_6T_1$, then $|x^G|>\frac{1}{2}q^{66} = b_3$ and we note that $|{}^3D_4(q)|< q^{28}= a_3$. Now assume $C_{\bar{G}}(x)^0 = E_6T_1$, so $|x^G|>\frac{1}{2}q^{54}=b_4$. We claim that $x$ is not a regular element of ${}^3D_4(q)$, so there are fewer than $q^{25}=a_4$ such elements in $H_0$. Suppose $x$ is regular and let $V = \mathcal{L}(\bar{G})$ be the adjoint module for $\bar{G}$. Let $\l_1, \ldots, \l_4$ be fundamental dominant weights for $\bar{H}^0=D_4$, labelled in the usual manner. As noted in \cite[Table 4]{Thomas}, we have
\[
V{\downarrow}\bar{H}^0 = \mathcal{L}(\bar{H}^0) \oplus V(2\l_1) \oplus V(2\l_3) \oplus V(2\l_4),
\]
where $V(2\l_i)$ is an irreducible module for $\bar{H}^0$ with highest weight $2\l_i$. Now $V(2\l_i)$ is the unique nontrivial composition factor of $S^2(V(\l_i))$ and by considering the action of a regular semisimple element on this latter module we deduce that $\dim C_{V(2\l_i)}(x) \leqs 5$. Since $\dim C_{\mathcal{L}(\bar{H}^0)}(x) = 4$, it follows that $\dim C_{\bar{G}}(x) = \dim C_{V}(x) \leqs 19$ and the claim follows.

Finally, suppose $x \in G$ is a field automorphism. Then $|x^G|>\frac{1}{2}q^{133/2}=b_5$ and we note that $|H|<3\log_3q.q^{28} = a_5$. We conclude that
\begin{equation}\label{e:qgh5}
\mathcal{Q}(G,H) < \sum_{i=1}^{5}a_i^2/b_i < q^{-1}
\end{equation}
as required.
\end{proof}

\begin{lem}\label{l:typeI_2}
If $G_0 = F_4(q)$ and ${\rm soc}(H) = G_2(q)$, then $b(G,H) = 2$.
\end{lem}

\begin{proof}
Here $p=7$, $H_0 = G_2(q)$ and we proceed as in the proof of the previous lemma. To get started, let $x \in H_0$ be an element of prime order $r$. First assume $r=7$, so $x$ is unipotent. The $G$-class of $x$ is determined in \cite[Section 5.2]{Law09} and we deduce that $|x^G|>q^{28} = b_1$ (minimal if $x$ is contained in the class labelled $A_1\tilde{A}_1$), so the contribution to $\mathcal{Q}(G,H)$ is less than $a_1^2/b_1$, where $a_1 = q^{12}$ is the total number of unipotent elements in $H_0$ (see Proposition \ref{p:i23}(ii)). If $r \ne 2,7$ then Proposition \ref{p:bounds} gives $|x^G|>(q-1)q^{29}=b_2$ and we note that $|H_0|<q^{14}=a_2$. 

Now assume $r=2$, so $|x^G \cap H| = q^4(q^4+q^2+1)=a_3$ since $H_0$ has a unique class of involutions. There are two classes of involutions in $G_0$ and we claim that $C_{\bar{G}}(x) = A_1C_3$ rather than $B_4$. To see this, let $V_{26}$ be the
minimal module for $\bar{G}$. By considering the restriction of $V_{26}$ to $C_{G_2}(x) = A_1\tilde{A}_1$, one checks that $x$ has trace $2$ on $V$ (see \cite[Section 6.6]{Thomas}) and thus $C_{\bar{G}}(x) = A_1C_3$ as claimed (as noted in \cite[Proposition 1.2]{LS99}, the involutions $y \in \bar{G}$ with $C_{\bar{G}}(y) = B_4$ have trace $-6$ on $V_{26}$). Therefore, $|x^G|>q^{28} = b_3$.

Finally, let us assume $x \in G$ is a field automorphism of order $r$. If $r=2$ then $|x^G|>\frac{1}{2}q^{26} = b_4$ and $|x^G\cap H| = |G_2(q):G_2(q^{1/2})|<2q^{7} = a_4$. On the other hand, if $r\geqs 3$ then $|x^G| > \frac{1}{2}q^{34} = b_5$ and we note that $|H|<\log_7q.q^{14} = a_5$. 

By bringing these estimates together, we conclude that \eqref{e:qgh5} holds and the result follows.
\end{proof}

\begin{lem}\label{l:typeI_3}
If $(G_0, {\rm soc}(H)) = (E_6^{\e}(q), G_2(q))$ or $(G_2(q),{\rm L}_{2}(q))$, then $b(G,H)=2$.
\end{lem}

\begin{proof}
First assume $G_0 = G_2(q)$ and ${\rm soc}(H) = {\rm L}_{2}(q)$, in which case $p \geqs 7$. If $x \in H_0$ has order $3$ then $|x^G| \geqs q^3(q^3-1) = b_1$ and we record that $i_3(H_0) \leqs q(q+1) = a_1$. As explained in \cite[Section 5.1]{Law09}, the elements of order $p$ in $H_0$ are regular in $G$. Therefore, if $x \in H$ has prime order and $x$ is not a semisimple element of order $3$, then $|x^G| \geqs q^3(q+1)(q^3+1)=b_2$ (minimal if $x$ is an involutory field automorphism) and we note that $|H| \leqs q(q^2-1).\log_7q = a_2$. This gives $\mathcal{Q}(G,H) < a_1^2/b_1+a_2^2/b_2<q^{-1/2}$ and we are done. 

Finally, let us assume $G_0 = E_6^{\e}(q)$ and ${\rm soc}(H) = G_2(q)$. From \cite[Section 5.5]{Law09}, we see that $H$ does not contain any long root elements of $G_0$. Let $x \in H$ be an element of prime order $r$. If $r \geqs 3$ then $|x^G|>(q-1)q^{31}=b_1$ and we note that $|H| < 2\log_2q.q^{14}=a_1$. On the other hand, if  $r=2$ then $|x^G|>\frac{1}{3}(q-1)q^{25}=b_2$ and Proposition \ref{p:i23}(i) implies that $i_2(H) \leqs i_2({\rm Aut}(G_2(q))) < 2(q+1)q^{7} = a_2$. Therefore, $\mathcal{Q}(G,H) < a_1^2/b_1+a_2^2/b_2<q^{-1/2}$ and the result follows.
\end{proof}

We are now in a position to complete the proof of Proposition \ref{p:typeIbase}.

\begin{proof}[Proof of Proposition \ref{p:typeIbase}]
In each of the remaining cases we find that the trivial bound $\mathcal{Q}(G,H) < a_1^2/b_1$ is sufficient, where $a_1 = |{\rm Aut}({\rm soc}(H))|$ is an upper bound on $|H|$ and $b_1$ is the size of the smallest nontrivial conjugacy class in $G$, which can be read off from Proposition \ref{p:bounds}. For example, if $G_0 = E_6^{\e}(q)$ and ${\rm soc}(H) = {\rm L}_{3}^{\pm}(q)$, then $p \geqs 5$ and the bounds
\[
|H| \leqs |{\rm Aut}({\rm L}_{3}^{\pm}(q))| < 2\log_5q.q^8 = a_1,\;\; |x^G| > (q-1)q^{21} = b_1
\]
are sufficient. 
\end{proof}

\section{Subfield subgroups}\label{s:part2}

In this section, we prove Theorem \ref{t:main} when $H$ is one of the subgroups in part (II) of Theorem \ref{t:types}. In particular, $G_0$ is one of the groups in \eqref{e:T} and $H$ is either a subfield subgroup corresponding to a maximal subfield of $\mathbb{F}_q$, or $H$ is a twisted subgroup of type ${}^2E_6(q^{1/2})$ or ${}^2F_4(q)$ with $G_0 = E_6(q)$ or $F_4(q)$ (for $q=2^{2m+1}$), respectively. The remaining subfield subgroups will be handled in Section \ref{s:part4}.

Our main result is the following.

\begin{thm}\label{t:sub}
If $H$ is a Type II subgroup of $G$, then $G$ is not extremely primitive.
\end{thm}

We partition the proof of Theorem \ref{t:sub} into a sequence of lemmas.

\begin{lem}\label{g2:subfield}
The conclusion to Theorem \ref{t:sub} holds when $G_0 = G_2(q)$ and $H$ is of type $G_2(q_0)$.
\end{lem}

\begin{proof}
Here $H_0 = G_2(q_0)$ and $q=q_0^k$, where $k$ is a prime. For $k$ odd, we claim that $b(G,H) = 2$, so $G$ is not extremely primitive by Lemma \ref{l:base2}. 

First assume $k \geqs 5$. If $x \in G_0$ is either a long root element (or a short root element when $p=3$) or a semisimple element of order $3$ with $C_{\bar{G}}(x) = A_2$, then $|x^G|>(q-1)q^5=b_1$ and we note that $H$ contains fewer than $2q_0^6 \leqs 2q^{6/5}=a_1$ such elements. For all other nontrivial elements $x \in G$ we have $|x^G|>q^{7}=b_2$ and thus $\mathcal{Q}(G,H) < a_1^2/b_1+a_2^2/b_2$, where $a_2 = 2\log_2q.q^{14/5}$ is an upper bound on $|H|$. One checks that the given upper bound on $\mathcal{Q}(G,H)$ is always less than $1$ (and it tends to zero as $q$ tends to infinity).

A more detailed analysis is required when $k=3$. Let $x \in G_0$ be an element of prime order $r$ and first assume $r=2$. If $p \ne 2$ then $G_0$ and $H_0$ both have a unique conjugacy class of involutions and we get $|x^G \cap H|<2q^{8/3}=a_1$ and $|x^G|>q^8 = b_1$. Similarly, if $p=2$ and $x$ is a long root element, then $|x^G \cap H|<q^2 = a_2$ and $|x^G|>(q-1)q^5=b_2$, whereas $|x^G \cap H|<q^{8/3} = a_3$ and $|x^G|>(q-1)q^7=b_3$ if $x$ is a short root element. 

Next assume $r=p \geqs 3$. As above, the contribution from long root elements is less than $a_2^2/b_2$. Similarly, short root elements contribute at most $a_2^2/b_2$ if $p=3$ and at most $a_3^2/b_3$ if $p \geqs 5$. Since $a_3^2/b_3<a_2^2/b_2$, it follows that the combined contribution from long and short root elements is less than $2a_2^2/b_2$ for all $p$. If $p=3$ and $x$ is in the class labelled $(\tilde{A}_1)_3$ in \cite[Table 22.2.6]{LieS} then $|x^G\cap H|<a_3$ and $|x^G|>b_3$. For all other unipotent elements, we have $|x^G|>\frac{1}{7}q^{10}=b_4$ and we note that $H_0$ contains precisely $q^4 = a_4$ unipotent elements (see Proposition \ref{p:i23}(ii)). 

Now assume $x \in G_0$ is semisimple with $r \geqs 3$. If $r=3$ and $C_{\bar{G}}(x) = A_2$ then $|x^G \cap H| < 2q^2 = a_5$ and $|x^G|>(q-1)q^5=b_5$. In every other case, $|x^G|>(q-1)q^9 = b_6$ and we note that $|H_0|<q^{14/3}=a_6$.

To complete the analysis for $k=3$, we may assume $x$ is a field or graph automorphism of $G_0$. Suppose $x$ is a field automorphism of order $r$, so $q=q_1^r$. If $r=2$ then $|x^G|>q^7=b_7$ and there are fewer than $2q^{7/3}=a_7$ such elements in $H$. Similarly, if $r \geqs 5$ then $|x^G|>\frac{1}{2}q^{56/5}=b_8$ and we note that $|H|<2\log_2q.q^{14/3}=a_8$. If $r=3$ then $|x^G|>\frac{1}{2}q^{28/3}=b_9$ and we may assume $x$ centralizes $H_0$, so there are precisely $2(i_3(H_0)+1)$ such elements in $H$. By applying Proposition \ref{p:i23}(i), we deduce that 
$2(i_3(H_0)+1) \leqs 4(q^{1/3}+1)q^3=a_9$. Finally, if $x$ is an involutory graph-field automorphism of $G_0$ then $x$ acts as an involutory graph-field automorphism on $H_0$, whence $|x^G \cap H|<2q^{7/3} = a_{10}$ and $|x^G|>q^{7}=b_{10}$.

Bringing the above bounds together, we conclude that
\[
\mathcal{Q}(G,H) < a_2^2/b_2 + \sum_{i=1}^{10}a_i^2/b_i
\]
and one checks that this bound is sufficient. 

Finally, let us assume $k=2$. In \cite{Law90}, Lawther calculates the subdegrees for $G_0$ and by inspecting \cite[Tables 2-4]{Law90} we find that $H_0$ has an orbit of length $\frac{1}{2}q^2 (q^6-1)(q^2-1)$. In particular, there exists $g \in G_0$ such that $|H_0 \cap H_0^g| = 2q^4$. By inspecting \cite[Tables 8.30, 8.41 and 8.42]{BHR}, we see that $H$ does not have a maximal subgroup $M$ with $|M \cap H_0| = 2q^4$. Therefore, $H \cap H^g$ is a non-maximal subgroup of $H$ and we conclude that $G$ is not extremely primitive.
\end{proof}

\begin{lem}\label{g2:ree}
The conclusion to Theorem \ref{t:sub} holds when $G_0 = G_2(q)$ and $H$ is of type ${}^2G_2(q)$.
\end{lem}

\begin{proof}
Here $p=3$ and $H_0 = {}^2G_2(q)$. The subdegrees for the action of $G_0$ are recorded in \cite[Table 1]{Law90} and we see that $H_0$ has an orbit of length $(q^3+1)(q-1)$. Therefore, there exists an element $g \in G_0$ such that $|H_0 \cap H_0^g| = q^3$. By inspecting \cite[Table 8.43]{BHR}, we deduce that there is no maximal subgroup $M$ of $H$ with $|M\cap H_0| = q^3$. The result follows.
\end{proof}

Now let us turn to the remaining subfield subgroups that we are handling in this section, so one of the following holds:
\begin{itemize}\addtolength{\itemsep}{0.2\baselineskip}
\item[{\rm (a)}] $G_0 = L(q)$ and $H_0$ is of type $L(q_0)$, where $L \in \{F_4,E_6^{\e},E_7,E_8\}$ and $q=q_0^k$ with $k$ a prime.
\item[{\rm (b)}] $G_0 = E_6(q)$, $q=q_0^2$ and $H_0$ is of type ${}^2E_6(q_0)$.
\item[{\rm (c)}] $G_0 = F_4(q)$, $q=2^{2m+1}$ and $H_0 = {}^2F_4(q)$.
\end{itemize}

First let us consider the cases in (a). As before, let $\bar{G}$ be the ambient simple algebraic group and fix a Steinberg endomorphism $\psi$ of $\bar{G}$ such that ${\rm soc}(H_0) = (\bar{G}_{\psi})'$ and $G_0 = (\bar{G}_{\psi^k})'$. Here $\psi = \sigma\tau$, where $\sigma$ is a standard Frobenius morphism corresponding to the map $\l \mapsto \l^{q_0}$ on $\mathbb{F}_q$ and either $\tau = 1$, or $G_0 = {}^2E_6(q)$, $k$ is odd and $\tau$ is an involutory graph automorphism of $\bar{G}$. 

\begin{rem}\label{r:h0}
Let us record that $H_0$ is simple, unless one of the following holds:
\begin{itemize}\addtolength{\itemsep}{0.2\baselineskip}
\item[{\rm (i)}] $G_0 = E_6^{\e}(q)$, $k=3$ and $q \equiv \e \imod{3}$, in which case $H_0 = {\rm Inndiag}(E_6^{\e}(q_0))$.
\item[{\rm (ii)}] $G_0 = E_7(q)$, $k=2$ and $q$ is odd, in which case $H_0 = {\rm Inndiag}(E_7(q_0))$.
\end{itemize}
\end{rem}

Set $\bar{X} = \la X_{\pm \a_0} \ra$ and $\bar{Y} = C_{\bar{G}}(\bar{X})^0$, where $\a_0$ is the highest root of $\bar{G}$. Then $\bar{M} = \bar{X}\bar{Y}$ is a maximal rank subgroup of $\bar{G}$ of type $A_1C_3$, $A_1A_5$, $A_1D_6$ or $A_1E_7$ for $\bar{G} = F_4, E_6, E_7$ or $E_8$, respectively.
Since $\bar{X}$ is $\psi$-stable, it follows that $\bar{M}$ is $\psi$-stable and by taking fixed points we get
\[
\bar{M}_\psi 
= \left\{\begin{array}{ll}
d.(\text{L}_2(q_0) \times \text{PSp}_6(q_0)).d & \mbox{if $\bar{G}=F_4$} \\
d.(\text{L}_2(q_0) \times \text{L}^\epsilon_6(q_0)).de & \mbox{if $\bar{G}=E_6$} \\
d.(\text{L}_2(q_0) \times \text{P}\Omega_{12}^+(q_0)).d & \mbox{if $\bar{G}=E_7$} \\
d.(\text{L}_2(q_0) \times E_7(q_0)).d & \mbox{if $\bar{G}=E_8$,} 
\end{array}\right.
\]
where $d=(2,q-1)$ and $e=(3,q-\e)$. By inspecting \cite[Table 5.1]{LSS}, it follows that $\bar{M}_{\psi}$ is a maximal subgroup of $\bar{G}_{\psi}$, unless $(\bar{G},p) = (F_4,2)$, in which case $\bar{M}_{\psi}<{\rm Sp}_{8}(q_0)<\bar{G}_{\psi}$. Set 
\[
M = \left\{\begin{array}{ll}
{\rm Sp}_{8}(q_0) & \mbox{if $(\bar{G},p) = (F_4,2)$} \\
\bar{M}_\psi \cap H_0 & \mbox{otherwise}
\end{array}\right.
\]
and $Y = (\bar{Y}_\psi)'$, so that $\text{SL}_2(q_0) \circ Y$ is a subgroup of $M$. We also set $K = T \circ Y$, where $T = Z_{q_0+1}$ is a maximal torus of ${\rm SL}_{2}(q_0)$. 

\begin{lem}\label{l:claim1}
The only maximal overgroup of $K$ in $H_0$ is $M$. In addition, if $(\bar{G},p) = (F_4,2)$ then $\bar{M}_\psi$ is the unique maximal overgroup of $K$ in $M$. 
\end{lem}

\begin{proof}
By construction, $K$ is contained in $M$, which in turn is a maximal subgroup of $H_0$ by \cite{LSS}. By inspection, we observe that $K$ is not contained in a parabolic subgroup of $H_0$. 

Suppose $K$ is contained in a non-parabolic maximal subgroup $J$ of $H_0$. If $H_0 = E_8(q_0)$ then $|J| \geqs |K|>q_0^{88}$ and thus \cite[Lemma 4.2]{BLS} implies that $J$ is of type $E_{8}(q_0^{1/2})$, $A_1(q_0)E_7(q_0)$ or $D_8(q_0)$ and we immediately deduce that $J$ must be $H_0$-conjugate to $M$. A very similar argument applies if $H_0 = E_7(q_0)$ or $E_6^{\e}(q_0)$, using \cite[Lemmas 4.6 and 4.13]{BLS}. For $H_0 = F_4(q_0)$ with $q_0 \geqs 3$ we appeal to \cite[Lemma 3.15]{Bur18}, noting that $|J| > q_0^{17}$, and we use \cite{NW} for $q_0=2$. Note that
if $H_0 = F_4(q_0)$ and $q_0$ is even, then $H_0$ has two classes of subgroups isomorphic to ${\rm Sp}_{8}(q_0)$, but $K$ is only contained in conjugates of $M$. To see this, first note that the subgroups in the other class are of the form $\bar{L}_{\psi}$, where $\bar{L}$ is of type $B_4$. Let $V = V_{26}$ be the irreducible $26$-dimensional module for $\bar{G} = F_4$ with highest weight $\l_4$ (in terms of the usual labelling). By \cite[Chapter 12, Table 2]{Thomas}, $K$ acts on $V$ with composition factors of dimension $14,6,6$ and this is incompatible with the action of $\bar{L}$, which has composition factors of dimension $16,8,1,1$. 

Next we claim that $K$ is contained in a unique conjugate of $M$. If $(\bar{G},p) \neq (F_4,2)$ then $Y$ is normal in $M$, which in turn is a maximal subgroup of $H_0$, so $M = N_{H_0}(Y)$. Since $N_{H_0}(K) \leqs N_{H_0}(Y)$, we quickly deduce that $K$ is contained in a unique conjugate of $M$. Indeed, suppose $K < M^g$, so $K, K^{g^{-1}} < M$. Then since $M$ contains a unique class of subgroups $H_0$-conjugate to $K$ we have $K = K^{mg}$ for some $m \in M$. But $mg \in N_{H_0}(K) \leqs M$ and thus $g \in M$.

Finally, let us assume that $(\bar{G},p) = (F_4,2)$, so 
\[
K = Z_{q_0+1} \times {\rm Sp}_{6}(q_0) < \bar{M}_{\psi} = {\rm SL}_{2}(q_0) \times {\rm Sp}_{6}(q_0) < M = {\rm Sp}_{8}(q_0) < H_0 = F_4(q_0).  
\]
We claim that $N_{H_0}(K) \leqs \bar{M}_\psi$ and the $H_0$-class of $K$ does not split in either $\bar{M}_\psi$ or $M$. Then by repeating the argument in the previous paragraph, we deduce that $M$ is the unique maximal overgroup of $K$ in $H_0$, and $\bar{M}_{\psi}$ is the unique maximal overgroup of $K$ in $M$, which completes the proof of the lemma.

To prove the claim, we first show that $M$ contains a unique class of subgroups isomorphic to $K$ and that $N_{M}(K) = K.2$. To do this, we use \cite[Tables 8.48 and 8.49]{BHR} to check that any maximal subgroup of $M$ containing a subgroup isomorphic to $K$ is necessarily conjugate to $\bar{M}_\psi$. It follows that $N_{M}(K)$ must be contained in a conjugate of $\bar{M}_\psi$ and since $N_{\bar{M}_\psi}(K) = K.2$ we conclude that 
$N_{M}(K) = K.2 \leqs \bar{M}_\psi$. Moreover, it is clear that $\bar{M}_\psi$ contains a unique class of subgroups isomorphic to $K$, so it follows that the $H_0$-class of $K$ does not split in $M$ nor in $\bar{M}_\psi$.  
Finally, let us consider $N_{H_0}(K)$. Since this is a proper subgroup of $H_0$ containing $K$, it is therefore contained in a conjugate of $M$. So $N_{H_0}(K) = N_{M^{g}}(K)$ for some $g \in H_0$ and we deduce that 
\[
N_{H_0}(K) = N_{M}(K^{g^{-1}})^g = (K^{g^{-1}}.2)^g = K.2 = N_{\bar{M}_\psi}(K) \leqs \bar{M}_{\psi}.
\]
This justifies the claim and the proof of the lemma is complete.
\end{proof} 

\begin{prop}\label{p:sub}
The conclusion to Theorem \ref{t:sub} holds if $G_0 = F_4(q)$, $E_6^{\e}(q)$, $E_7(q)$ or $E_8(q)$ and $H$ is of type $F_4(q_0)$, $E_6^{\e}(q_0)$, $E_7(q_0)$ or $E_8(q_0)$, respectively.
\end{prop}

\begin{proof}
Write $G = G_0.A$, $H = H_0.A$ and let us adopt the notation introduced above. 

Suppose there exists $g \in N_{G_0}(K)$ that does not normalize $M$. Then 
$M^g \not\leqs H_0$ and the result now follows via Lemma \ref{l:simple}, unless $(\bar{G},p) = (F_4,2)$ and $G$ contains graph automorphisms (indeed, we may assume $K$ and $M$ are $A$-stable in all the other cases). Let us assume we are in the latter situation, so $K = Z_{q_0+1} \times Y$ with $Y = {\rm Sp}_{6}(q_0)$. Suppose $H \cap H^g$ is a maximal subgroup of $H$. We claim that $H \cap H^g = H_0.B$ 
for some maximal subgroup $B$ of $A$. Now $A$ contains a graph automorphism $\tau$ such that $Y$ and $Y^{\tau}$ are non-conjugate subgroups of $H_0$, so $\la Y, Y^{\tau}\ra = H_0$ and thus $H_0$ is the only $A$-stable subgroup of $H_0$ containing $K$. This justifies the claim.
Therefore, $H^g$ contains $H_0$ and $H_0^g$. But $H_0$ and $H_0^g$ generate $G_0$ since they are both maximal subgroups of $G_0$ and $g$ does not normalize $H_0$ (this follows from Lemma \ref{l:claim1}; if $g$ normalizes $H_0$ then it must also normalize the unique maximal overgroup of $K$ in $H_0$, which is $M$). This is absurd since $H^g$ does not contain $G_0$. So we have reached a contradiction and we conclude that $H \cap H^g$ is not maximal in $H$, whence $G$ is not extremely primitive.
 
In view of the above remarks, in order to prove the proposition it suffices to show that there exists $g \in N_{G_0}(K)$ with $M \ne M^g$. To do this, take an element $g \in \bar{X}_{\psi^k} = {\rm SL}_{2}(q)$ that centralizes $T = Z_{q_0+1}$ but does not normalize $\bar{X}_\psi = \text{SL}_2(q_0)$. Such an element $g$ exists because the centralizer of $T$ in $\bar{X}_{\psi^k}$ is a cyclic torus of order $q_0^k - (-1)^k$ which is clearly not a subgroup of $N_{\bar{X}_{\psi^k}}(\bar{X}_{\psi}) = \bar{X}_{\psi}.a$ (where $a \in \{1,2\}$).

Since $g \in \bar{X}$, it follows that $g$ centralizes $Y$ and $K = T \circ Y$. In addition, since $\bar{X}_\psi \circ Y$ is characteristic in $\bar{M}_\psi \cap H_0$, it follows that $g$ does not normalize $\bar{M}_\psi \cap H_0$. If $(\bar{G},p) \neq (F_4,2)$ then $\bar{M}_\psi \cap H_0 = M$ and we have proved the claim. So suppose that $(\bar{G},p) = (F_4,2)$. Then Lemma \ref{l:claim1} shows that $\bar{M}_\psi$ is the unique maximal overgroup of $K$ in $M$. Therefore, if $g$ normalizes $M$ then it must also normalize the unique maximal overgroup of $K$ in $M$. But we have just observed that $g$ does not normalize $\bar{M}_\psi \cap H_0 = \bar{M}_\psi$ and this completes the proof of the proposition.
\end{proof}

Finally, we handle the twisted maximal subgroups that arise when $G_0 = E_6(q)$ or $F_4(q)$ (with $p=2$ in the latter case). 

\begin{prop}\label{p:twist}
The conclusion to Theorem \ref{t:sub} holds if $G_0 = E_6(q)$ or $F_4(q)$ and $H$ is of type ${}^2E_6(q^{1/2})$ or ${}^2F_4(q)$, respectively.
\end{prop}

\begin{proof}
First assume $G_0 = E_6(q)$, where $q=q_0^2$. Set $\psi = \s\tau$, where $\s$ is a standard Frobenius morphism of $\bar{G} = E_6$ and $\tau$ is a graph automorphism such that $(\bar{G}_{\psi})' = {\rm soc}(H_0) = {}^2E_6(q_0)$ and $(\bar{G}_{\psi^2})' = G_0$. 

Set $\bar{M} = \bar{X}\bar{Y} = A_1A_5$, where $\bar{X} = \langle X_{\pm \alpha_0} \rangle$ and $\bar{Y} = C_{\bar{G}}(\bar{X})^0$. Then $\bar{M}$ is $\psi$-stable and 
\cite{LSS} implies that 
\[
M = \bar{M}_{\psi} \cap H_0 = ({\rm SL}_{2}(q_0) \circ d.{\rm U}_{6}(q_0)).d = d.(\text{L}_2(q_0) \times \text{U}_6(q_0)).d
\]
is a maximal subgroup of $H_0$, where $d=(2,q-1)$. Set $K = Z_{q_0+1} \circ \text{U}_6(q_0) < M$. By arguing as in the proof of Lemma \ref{l:claim1}, we deduce that $M$ is the unique maximal overgroup of $K$ in $H_0$. Similarly, by repeating the argument in the proof of Proposition \ref{p:sub}, we see that there exists $g \in N_{G_0}(K)$ such that $M \ne M^g$ and we conclude that $G$ is not extremely primitive by applying Lemma \ref{l:simple}. 

Finally, let us assume $G_0 = F_4(q)$ and $H_0 = {}^2F_4(q)$, where $q=2^{2m+1}$ and $m \geqs 0$. Let $\bar{G} = F_4$ and set $\psi = \sigma\tau$, where $\sigma$ is the standard Frobenius morphism of $\bar{G}$ corresponding to the map $\l \mapsto \l^{2^m}$ on $\mathbb{F}_q$ and $\tau$ is the standard graph automorphism of $\bar{G}$, which interchanges long root and short root subgroups. Then $G_0 = \bar{G}_{\psi^2}$ and $H_0 = \bar{G}_{\psi}$. 

The case $q=2$ can be handled in {\sc Magma} (see \cite[Lemma 2.10]{BTh_comp}). More precisely, we construct $H_0 < G_0$ as permutation groups and we use random search to find an element $g \in G_0$ such that $H_0 \cap H_0^g = 1$. Therefore, $b(G_0,H_0)=2$ and more generally we have $|H \cap H^g| \leqs 2$. In particular, $H \cap H^g$ is not a maximal subgroup of $H$ and thus $G$ is not extremely primitive. For the remainder, we will assume $q \geqs 8$. 

Fix a set of simple roots $\a_1,\a_2,\a_3,\a_4$ for $\bar{G}=F_4$ and consider the following roots
\[
\begin{array}{ll}
\beta_1 = \alpha_1 + \alpha_2 + \alpha_3 & \beta_2 = \alpha_2 + 2\alpha_3 + 2\alpha_4 \\
\beta_3 = \alpha_1 + 2\alpha_2 + 3\alpha_3 + 2\alpha_4 & \beta_4 = 2\alpha_1 + 3\alpha_2 + 4\alpha_3 + 2\alpha_4 = \alpha_0
\end{array}
\]
where $\beta_1, \beta_3$ are short and $\beta_2, \beta_4$ are long. Note that $\psi$ acts on root elements as follows
\[
x_{\beta_1}(c) \mapsto x_{\beta_2}(c^{2^{m+1}}),\; x_{\beta_2}(c) \mapsto x_{\beta_1}(c^{2^{m}}),\; x_{\beta_3}(c) \mapsto x_{\beta_4}(c^{2^{m+1}}), \; x_{\beta_4}(c) \mapsto x_{\beta_3}(c^{2^{m}})
\]
for all $c \in \bar{\mathbb{F}}_2$. Let $\bar{P} = P_{1,4}$ be the standard parabolic subgroup of $\bar{G}$ with $\bar{L}' = 
\la X_{\pm \alpha_2}, X_{\pm \alpha_3} \ra$ of type $B_2$, where $\bar{L}$ is a Levi factor of $\bar{P}$. Then $C_{\bar{G}}(\bar{L}')^0 = \langle X_{\pm \beta_1},X_{\pm \beta_2},X_{\pm \beta_3},X_{\pm \beta_4} \ra  = B_2$. 

Consider the maximal subgroup $M = {}^2B_2(q) \wr {\rm Sym}_2$ of $H_0$ (see \cite{Mal}) and set 
\[
K = {}^2B_2(q) \times Z_{q-1} < M.
\]
By choosing $M$ appropriately, we may assume that $K$ is a Levi subgroup of the maximal parabolic $P = \bar{P}_\psi$ of $H_0$. Let $\mathcal{M}$ be the set of maximal overgroups of $K$ in $H_0$. We claim that $\mathcal{M} = \{M,P,P^{\rm op}\}$, where $P^{\rm op} = (\bar{P}^{\rm op})_{\psi}$ is the opposite parabolic subgroup to $P$. 

Firstly, by inspecting \cite{Mal} we deduce that each subgroup in $\mathcal{M}$ is conjugate to $M$ or $P$. Let $X$ be the ${}^2B_2(q)$ factor in $K$ and set $N = {}^2B_2(q) \times {}^2B_2(q) <M$. Now $N \leqs N_{H_0}(X)$ since $X$ is normal in $N$, but $M$ is clearly the unique maximal overgroup of $N$ in $H_0$, and $X$ is not normal in $M$, so $N_{H_0}(X) = N$ and we deduce that $N_{H_0}(K) \leqs N$. By \cite[Proposition 3]{Suz}, the normalizer of the torus $Z_{q-1} < {}^2B_2(q)$ is a dihedral group $D_{2(q-1)}$ and so  
\[
N_{H_0}(K) = N_{N}(K) = K.2 = {}^2B_2(q) \times D_{2(q-1)}.
\]
Since $N_{H_0}(K) = K.2 < M$, it follows that $M$ is the only $H_0$-conjugate of $M$ containing $K$. Now let us turn to the conjugates of $P$ in $\mathcal{M}$. Here we proceed as in the proof of Lemma \ref{l:maxrank_e7_4} (see the argument for the case $q=2$), noting that if $K \leqs P \cap P^h$ for some $h \in H_0$, then $K$ and $K^{h^{-1}}$ are $P$-conjugate (this follows from \cite[Proposition 26.1(b)]{MT}).

By the claim, each subgroup in $\mathcal{M}$ contains $x_1 = x_{\beta_3}(1)x_{\beta_4}(1)$ or $x_2 = x_{-\beta_3}(1)x_{-\beta_4}(1)$.

The subgroup $K$ is contained in a Levi subgroup ${\rm Sp}_{4}(q) \times Z_{q-1}^2$ of the parabolic subgroup $\bar{P}_{\psi^2}$ of $G_0$, where $Z_{q-1}^2 = \la h_{\beta_1}(c), h_{\beta_2}(c') \,:\, c,c' \in \mathbb{F}_q^{\times}\ra$ in terms of the standard Lie notation. Since $X < {\rm Sp}_{4}(q)$ it follows that $Z_{q-1}^2 \leqs N_{G_0}(K)$. Fix $1 \ne c \in \mathbb{F}_q^\times$ and set 
$g = h_{\beta_3}(c) \in N_{G_0}(K)$ so $x_1^g = x_{\beta_3}(c^2)x_{\beta_4}(c^2)$ (here we are using the fact that $q \geqs 8$). Now
\[
\psi(x_{\beta_3}(c^2)x_{\beta_4}(c^2)) = x_{\beta_4}((c^2)^{2^{m+1}})x_{\beta_3}((c^2)^{2^{m}}) = x_{\beta_3}(c^{2^{m+1}}) x_{\beta_4}(c^{2^{m+2}})
\]
and thus $x_1^g$ is not fixed by $\psi$. In particular, $x_1^g \not\in H_0$. An entirely similar calculation shows that $\psi(x_2^g) \neq x_2^g$. Therefore, for each $J \in \mathcal{M}$ we have $J^g \not\leqs H_0$. Finally, since $K$ and the three subgroups in $\mathcal{M}$ are stable under all automorphisms of $G_0$, the desired conclusion follows from Lemma \ref{l:simple}.     
\end{proof}

\section{Almost simple subgroups}\label{s:part3}

In this section, we complete the proof of Theorem \ref{t:main} for the groups with socle  $G_0$ as in \eqref{e:T}. To do this, it remains to handle the subgroups of Type V in Theorem \ref{t:types}. Our main result is the following (we have been unable to determine the exact base size in (ii)).

\begin{thm}\label{t:typeV}
If $H$ is a Type V subgroup of $G$, then $G$ is not extremely primitive. Moreover, $b(G,H) = 2$ unless one of the following holds:
\begin{itemize}\addtolength{\itemsep}{0.2\baselineskip}
\item[{\rm (i)}] $G_0 = {}^2E_6(2)$, $S = {\rm Fi}_{22}$ and $b(G,H) = 3$;
\item[{\rm (ii)}] $G_0 = F_4(2)$, $S = {\rm L}_{4}(3)$ and $b(G,H) \leqs 3$.
\end{itemize}
\end{thm}

The following theorem, which is part of \cite[Theorem 8]{LS03}, describes the possibilities for the socle of a Type V subgroup (note that the value of $u(E_8(q))$ in part (ii)(c) is taken from \cite{Laww}). In the statement, if $X$ is a simple group of Lie type then ${\rm rk}(X)$ is the (untwisted) Lie rank of $X$. In addition, ${\rm Lie}(p)$ denotes the set of finite simple groups of Lie type defined over fields of characteristic $p$. 

\begin{thm}\label{t:simples}
Let $G$ be an almost simple group with socle $G_0$, a simple exceptional group of Lie type over $\mathbb{F}_q$, where $q=p^f$ with $p$ a prime. Assume $G_0$ is one of the groups in \eqref{e:T} and let $H$ be a maximal almost simple subgroup of $G$ as in part (V) of Theorem \ref{t:types}, with socle $S$.  Then one of the following holds:
\begin{itemize}\addtolength{\itemsep}{0.2\baselineskip}
\item[{\rm (i)}] $S \not\in {\rm Lie}(p)$ and the possibilities for $S$ are described in \cite{LS99};
\item[{\rm (ii)}] $S  = H(q_0) \in {\rm Lie}(p)$, ${\rm rk}(S) \leqs \frac{1}{2}{\rm rk}(G_0)$ and one of the following holds: 
\begin{itemize}\addtolength{\itemsep}{0.2\baselineskip}
\item[{\rm (a)}] $q_0 \leqs 9$;
\item[{\rm (b)}] $S = {\rm L}_{3}^{\e}(16)$;
\item[{\rm (c)}] $S = {\rm L}_{2}(q_0)$, ${}^2B_2(q_0)$ or ${}^2G_2(q_0)$, where $q_0 \leqs (2,q-1)u(G_0)$ and $u(G_0)$ is defined in the following table.  
{\small
\renewcommand{\arraystretch}{1.2}
\[
\begin{array}{llllll} \hline
G_0 & G_2(q) & F_4(q) & E_6^{\e}(q) & E_7(q) & E_8(q) \\ \hline
u(G_0) & 12 & 68 & 124 & 388 & 1312 \\ \hline
\end{array}
\]
\renewcommand{\arraystretch}{1}}
\end{itemize}
\end{itemize}
\end{thm}

More recently, the list of possibilities for $S$ in parts (i) and (ii) of Theorem \ref{t:simples} has been significantly refined. For the so-called \emph{non-generic} subgroups arising in (i), we refer the reader to Litterick \cite{Litt} and Craven \cite{Craven17}. For instance, the main theorem of \cite{Craven17} states that $S = {\rm Alt}_n$ only if $n=6$ or $7$. Craven has also made substantial progress in eliminating many \emph{generic} subgroups in (ii). Indeed, by combining the main results of \cite{Cr1,Cr2}, we get the following theorem.

\begin{thm}[Craven]\label{t:craven}
Let $G$ be an almost simple group with socle $G_0$, a simple exceptional group of Lie type over $\mathbb{F}_q$, where $q=p^f$ with $p$ a prime. Assume $G_0$ is one of the groups in \eqref{e:T} and let $H$ be a maximal almost simple subgroup of $G$ as in part (V) of Theorem \ref{t:types}, with socle $S \in {\rm Lie}(p)$.  Then one of the following holds:
\begin{itemize}\addtolength{\itemsep}{0.2\baselineskip}
\item[{\rm (i)}] $G_0 = E_8(q)$ and either $S = {\rm L}_2(q_0)$ with $q_0 \leqs (2,q-1)u(G_0)$, or
\[
S \in \{ {\rm L}_3^{\e}(3), {\rm L}_3^{\e}(4), {\rm U}_3(8), {\rm PSp}_4(2)', {\rm U}_4(2), 
{}^2B_2(8)\};
\]
\item[{\rm (ii)}] $G_0 = E_7(q)$ and $S = {\rm L}_{2}(q_0)$ with $q_0 \in \{7,8,25\}$.
\end{itemize}
\end{thm}

We begin the proof of Theorem \ref{t:typeV} by handling the groups with socle $G_0 = G_2(q)$.

\begin{prop}\label{p:g2simple}
The conclusion to Theorem \ref{t:typeV} holds when $G_0 = G_2(q)$.
\end{prop}

\begin{proof}
In view of Theorem \ref{t:small} (and Remark \ref{r:small}), we may assume $q \geqs 7$. By inspecting \cite{Coop, K88} (also see \cite[Tables 8.30, 8.41 and 8.42]{BHR}), we observe that there are four cases to consider (in each case $S \not\in {\rm Lie}(p)$):
\begin{itemize}\addtolength{\itemsep}{0.2\baselineskip}
\item[{\rm (a)}] $S = {\rm L}_{2}(13)$ and either $q = p \equiv 1,3,4,9,10,12 \imod{13}$, or $q=p^2$ and $p \equiv 2,5,6,7,8,11 \imod{13}$;
\item[{\rm (b)}] $S = {\rm L}_{2}(8)$ and either $q = p \equiv 1,8 \imod{9}$, or $q=p^3$ and $p \equiv 2,4,5,7 \imod{9}$;
\item[{\rm (c)}] $G = G_2(q)$, $H = {\rm U}_{3}(3){:}2 = G_2(2)$ and $q=p \geqs 7$;
\item[{\rm (d)}] $G = G_2(11)$ and $H = {\rm J}_{1}$.
\end{itemize} 
In all four cases, we claim that $b(G,H)=2$.

In (a) and (b), we have $|H| \leqs |{\rm PGL}_{2}(13)| = 2184=a_1$ and $|x^G| \geqs q^3(q^3-1) = b_1$ for all nontrivial $x \in G$. Moreover, the given conditions imply that $q \geqs 17$ and thus $\mathcal{Q}(G,H) < a_1^2/b_1<1$ (in addition, this upper bound is less than $q^{-1}$ if $q \geqs 23$).

Next consider (c), so $|H| = 12096$. Let $x \in H$ be an element of prime order $r$ and note that $r \in \{2,3,7\}$. Let us also observe that 
\[
i_2(H) = 315 = a_1,\; i_3(H) = 728 = a_2, \; i_7(H) = 1728=a_3.
\]
If $r=2$ then $G$ has a unique conjugacy class of elements of order $r$ and we have $|x^G| = q^4(q^4+q^2+1)=b_1$. Similarly, if $r=3$ then $|x^G| \geqs q^3(q^3-1)=b_2$ and for $r=7$ and $q \geqs 11$ we get $|x^G| \geqs q^5(q^3-1)(q^2-q+1) = b_3$. Putting these estimates together, we deduce that $\mathcal{Q}(G,H) < \sum_{i=1}^{3}a_i^2/b_i$, which is less than $1$ if $q \geqs 11$ (and it is less than $q^{-1}$ if $q \geqs 17$). 

To complete the proof in case (c), we may assume $q=7$. Here we need to be more careful when estimating the contributions to $\mathcal{Q}(G,H)$ from elements of order $3$ and $7$. To do this, let $V$ be the minimal $7$-dimensional module for $G$ over $\mathbb{F}_7$ and observe that $H$ acts irreducibly on $V$ (see \cite[Theorem A]{K88}, for example). With the aid of {\sc Magma} \cite{Magma}, we can compute the action of each $x \in H$ on $V$ (we refer the reader to \cite[Lemma 2.11]{BTh_comp} for the details of these computations). If $x$ has order $7$, then we find that $x$ has Jordan form $(J_7)$ on $V$ and thus $x$ is a regular unipotent element in $G$ (see \cite[Table 1]{Lawunip}). Therefore, the contribution to $\mathcal{Q}(G,H)$ from elements of order $7$ is precisely $a_2^2/b_2$, where $a_2 = 1728$ and $b_2 = 7^4(7^2-1)(7^6-1)$. 

Finally, suppose $x \in H$ has order $3$ and note that both $H$ and $G$ have two conjugacy classes of elements of order $3$. We will write \texttt{3A} and \texttt{3B} to denote the two $H$-classes (they have sizes $56=a_3$ and $672=a_4$, respectively). We find that \texttt{3A}-elements and \texttt{3B}-elements have Jordan form $(I_1, \omega I_3, \omega^2I_{3})$ and $(I_3, \omega I_2, \omega^2I_{2})$ on $V$, respectively, where $\omega \in \mathbb{F}_7$ is a primitive cube root of unity. In particular, we see that the two classes are not fused in $G$. Moreover, we deduce that if $x \in H$ is a \texttt{3A}-element then $C_G(x) = {\rm SL}_{3}(7)$, so $|x^G| = 7^3(7^3+1) = b_3$, whereas $|x^G| = 7^5(7^6-1)/6 = b_4$ for the elements in \texttt{3B}. Setting $a_1 = 315$ and $b_1 = 7^4(7^4+7^2+1)$ as before, we conclude that 
\[
\mathcal{Q}(G,H) = \sum_{i=1}^{4}a_i^2/b_i = \frac{4649}{103243} < 1.
\]

Finally, let us turn to case (d). Suppose $x \in H$ has prime order $r$, so $r \in \{2,3,5,7,11,19\}$ and we note that
\[
i_2(H) = 1463 = a_1, \; i_3(H) = 5852 = a_2,\; i_5(H) = 9704=a_3,
\]
\[
i_7(H) = 25080=a_4,\; i_{11}(H) = 27720 = a_5,\; i_{19}(H) = 27720 = a_6.
\]
If $r=2$ then $|x^G| = 11^4(11^4+11^2+1) = b_1$. Similarly, if $r \in \{5,7,19\}$ then $C_{\bar{G}}(x) = A_1T_1$ or $T_2$ and thus $|x^G| \geqs 11^5(11^3-1)(11^2-11+1)=b_3=b_4=b_6$. 

Now assume $r \in \{3,11\}$. Let $V$ be the minimal module for $G$ over $\mathbb{F}_{11}$ and note that $H$ acts irreducibly on $V$. Using {\sc Magma}, we can compute the action of $x$ on $V$ (see \cite[Lemma 2.11]{BTh_comp}). If $r=3$ then $x$ has Jordan form $(I_3, \omega I_2, \omega^2I_2)$ on $V \otimes \bar{\mathbb{F}}_{11}$, so $C_{\bar{G}}(x) \ne A_2$ and thus $|x^G| = 11^5(11^3-1)(11^2-11+1) = b_2$.  Finally, if $r=11$ then $x$ has Jordan form $(J_{7})$ on $V$, so $x$ is a regular unipotent element in $G$ and $|x^G| = 11^4(11^2-1)(11^6-1) = b_5$. 

We conclude that
\[
\mathcal{Q}(G,H) < \sum_{i=1}^{6}a_i^2/b_i < 1
\]
and the proof of the proposition is complete.
\end{proof}

Next we consider the two special cases highlighted in the statement of Theorem \ref{t:typeV}.

\begin{lem}\label{l:ndef2}
The conclusion to Theorem \ref{t:typeV} holds if $G_0= {}^2E_6(2)$ and $S = {\rm Fi}_{22}$.
\end{lem}

\begin{proof}
Here $(G,H) = ({}^2E_6(2),{\rm Fi}_{22})$ or $({}^2E_6(2).2,{\rm Fi}_{22}.2)$. In both cases, we observe that 
\[
\frac{\log |G|}{\log |G:H|}>2
\]
which implies that $b(G,H) \geqs 3$ (in fact, we have $b(G,H)=3$, as noted in the proof of \cite[Proposition 4.21]{BLS}). 

First assume $G = {}^2E_6(2)$ and $H = {\rm Fi}_{22}$. Set $\O = G/H$. The character tables of $G$ and $H$ are available in the \textsf{GAP} Character Table Library \cite{GAPCTL}, together with the corresponding fusion map from $H$-classes to $G$-classes. This allows us to compute $|C_{\O}(x)|$ for each $x \in H$, where $C_{\O}(x)$ is the set of fixed points of $x$ on $\O$. In this way, via the Orbit Counting Lemma, we deduce that $H$ has $8$ orbits on $\O$ (see \cite[Lemma 2.12]{BTh_comp} for the details of this computation). Let $d_1, \ldots, d_7$ be the lengths of the nontrivial $H$-orbits, ordered so that $d_{i} \leqs d_{i+1}$ for each $i$.

Seeking a contradiction, suppose $G$ is extremely primitive. Then each $d_i$ must be the index of a maximal subgroup of $H = {\rm Fi}_{22}$ and by inspecting the Web-Atlas \cite{WebAt} we deduce that each $d_i$ is one of the following:
\[
3510,14080,61776,142155,694980,1216215,1647360,
\]
\[
3592512,3648645,12812800, 17791488,679311360.
\]
Since  
\[
1+6 \! \cdot \! 17791488+679311360<|\O|
\]
it follows that $d_6 = d_7 = 679311360$. But $1+2 \! \cdot \! 679311360>|\O|$ and we have reached a contradiction.

The case $G = {}^2E_6(2).2$ with $H = {\rm Fi}_{22}.2$ is entirely similar. Once again, by computing fixed points, we find that $H$ has $8$ orbits on $\O = G/H$ and by inspecting \cite{WebAt} we deduce that if $G$ is extremely primitive then the length of each nontrivial $H$-orbit is one of
\[
3510,61776,142155,694980,1216215,1647360,3612614,
\]
\[
3648645,5125120,
12812800,15206400,17791488.
\]
But $1+7  \! \cdot \! 17791488 < |\O|$, so $G$ is not extremely primitive.
\end{proof}

\begin{lem}\label{l:ndef20}
The conclusion to Theorem \ref{t:typeV} holds if $G_0 = F_4(2)$ and $S = {\rm L}_4(3)$.
\end{lem}

\begin{proof}
First assume $G = F_4(2)$, so $H = {\rm L}_{4}(3).2$ (see \cite{NW}). The character tables of $G$ and $H$ are available in the \textsf{GAP} Character Table Library \cite{GAPCTL}. Although the precise fusion of $H$-classes in $G$ is not available in \cite{GAPCTL}, we can use \texttt{PossibleClassFusions} to compute 
\[
{\rm fpr}(x,G/H) = \frac{|x^G \cap H|}{|x^G|}
\]
for all $x \in G$ of prime order (there are two possible fusion maps and they both give the same fixed point ratios). This allows us to compute $\mathcal{Q}(G,H)$ precisely and we find that $\mathcal{Q}(G,H)>1$ (indeed, just the contribution from involutions is greater than $1$). If $x_1, \ldots, x_k$ are representatives of the $G$-classes of elements of prime order, then 
\[
\sum_{i=1}^{k}|x_i^G|\cdot {\rm fpr}(x_i,G/H)^3 < 1
\]
and thus $b(G,H) \leqs 3$ by \cite[Corollary 2.4]{BOW}. Since $\log |G|< 2\log |G/H|$, we cannot rule out $b(G,H) = 2$ and we have been unable to determine the exact base size in this case. 

To show that $G$ is not extremely primitive, we can argue as in the proof of the previous lemma. By applying the Orbit Counting Lemma, we deduce that $H$ has $94$ orbits on $G/H$. However, if $M$ is a core-free maximal subgroup of $H$ then $|H:M| \leqs 10530$ and we have $1+93  \! \cdot \! 10530<|G:H|$. We conclude that $G$ is not extremely primitive. (See \cite[Lemma 2.13]{BTh_comp} for further details on the computation.)

The case $G = F_4(2).2$, $H = {\rm L}_{4}(3).2^2$ is entirely similar and we omit the details (here $H$ has $66$ orbits on $G/H$).
\end{proof}

\begin{lem}\label{l:special}
The conclusion to Theorem \ref{t:typeV} holds when $G_0 = E_7(2)$, $E_6^{\e}(2)$ or $F_4(2)$.
\end{lem}

\begin{proof}
By inspecting \cite{BBR, KW, NW, Wil2}, we see that $(G_0,S) = ({}^2E_6(2),\O_7(3))$, $({}^2E_6(2),{\rm Fi}_{22})$ or $(F_4(2),{\rm L}_{4}(3))$. The latter two cases were handled in Lemma \ref{l:ndef2} and \ref{l:ndef20}, so we may assume that $G_0 = {}^2E_6(2)$ and $S = \O_7(3)$, so by \cite{Wil2} we have $(G,H) = ({}^2E_6(2), \O_7(3))$ or $({}^2E_6(2).2, {\rm SO}_{7}(3))$. In both cases, by proceeding as in the proof of Lemma \ref{l:ndef20} we can compute $\mathcal{Q}(G,H)$ precisely and we deduce that $\mathcal{Q}(G,H)<1$ (see \cite[Lemma 2.14]{BTh_comp}). Therefore, $b(G,H) = 2$.
\end{proof}

For the remainder of this section, we will assume
\begin{equation}\label{e:list}
G_0 \in \{E_8(q), E_7(q), E_6^{\e}(q), F_4(q)\} \setminus \{E_7(2),E_6^{\e}(2),F_4(2)\}.
\end{equation}
We partition the remainder of the proof of Theorem \ref{t:typeV} into two parts, according to the cases $S \not\in {\rm Lie}(p)$ and $S \in {\rm Lie}(p)$. Before launching into the details, let us record the following result on long root elements, which will be useful in the subsequent analysis.

\begin{prop}\label{p:long}
Suppose $p>2$ and let $H$ be a Type V subgroup of $G$. Then $H$ does not  
contain a long root element of $G$.
\end{prop}

\begin{proof}
This follows immediately from \cite[Corollary 6.2]{LS94}. Indeed, if $x \in H$ is a long root element of $G$, then $x^2 \in H$ is also a long root element and thus 
\cite[Corollary 6.2]{LS94} implies that $H \leqs N_G(\bar{H}_{\s}) < G$, where $\bar{H}$ is a $\s$-stable positive dimensional maximal closed subgroup of $\bar{G}$. But this is incompatible with the definition of a Type V subgroup.  
\end{proof}

\begin{rem}
For $p=2$, it is worth noting that the conclusion to Proposition \ref{p:long} is false in general. For example, $G = {}^2E_6(2)$ has a maximal subgroup $H = {\rm Fi}_{22}$ (this case was handled in Lemma \ref{l:ndef2}) and we find that the \texttt{2A}-involutions in $H$ embed in $G$ as long root elements. 
\end{rem}

\subsection{Non-generic subgroups}\label{ss:ndef}

In this section we handle the non-generic subgroups arising in part (i) of Theorem \ref{t:simples}, where $S \not\in {\rm Lie}(p)$. 

\begin{lem}\label{l:ndef1}
The conclusion to Theorem \ref{t:typeV} holds if $S = {\rm Alt}_n$.
\end{lem}

\begin{proof}
By the main theorem of \cite{Craven17}, we may assume $n \in \{6,7\}$, so $|H|$ is at most $7! = a_1$ and we note that $|x^G|>q^{16}=b_1$ for all nontrivial $x \in G$. This gives $\mathcal{Q}(G,H)<a_1^2/b_1$, which is less than $1$ for all $q \geqs 3$ (and it is less than $q^{-1}$ for $q \geqs 4$). Finally, if $q=2$ then $G = E_8(2)$ is the only option (see \cite[Theorem 1]{Craven17}) and $|x^G|>2^{58}$ for all $1 \ne x \in G$. The result follows as before. 
\end{proof}

\begin{lem}\label{l:ndef3}
The conclusion to Theorem \ref{t:typeV} holds if $S$ is a sporadic simple group.
\end{lem}

\begin{proof}
The possibilities for $G_0$ and $S$ are recorded in \cite[Table 10.2]{LS99}, which is further refined in \cite[Theorem 8]{Litt} to give the list of cases recorded in Table \ref{tab:spor}. Recall that we may assume $G_0$ is one of the groups in \eqref{e:list}. In each case, we claim that $b(G,H)=2$.

\begin{table}
\begin{center}
\[
\begin{array}{ll}\hline
S & \bar{G} \\ \hline 
{\rm M}_{11} & E_6 \, (p=3,5),\; E_8 \, (p=3,11) \\
{\rm M}_{12} & E_6\, (p=5) \\
{\rm M}_{22} & E_7 \, (p=5) \\
{\rm J}_{1} & E_6 \, (p=11) \\
{\rm J}_{2} & E_6\, (p=2),\; E_7 \, (p=2) \\
{\rm J}_{3} & E_6 (p=2),\; E_8 \, (p=2) \\
{\rm Ru} & E_7 \, (p=5) \\
{\rm Fi}_{22} & E_6 \, (p=2) \\
{\rm HS} & E_7 \, (p=5) \\
{\rm Th} & E_8 \, (p=3) \\ \hline
\end{array}
\]
\caption{The possibilities for $\bar{G}$ and $S$, where $S$ is sporadic}
\label{tab:spor}
\end{center}
\end{table}

The cases with $S \in \{{\rm M}_{11}, {\rm M}_{12}, {\rm M}_{22}, {\rm J}_{1}, {\rm Ru}, {\rm HS}\}$ are very straightforward; we have $|H| \leqs |{\rm Aut}(S)|=a_1$ and by applying Proposition \ref{p:bounds} we deduce that $|x^G|>f(q)=b_1$ for all $x \in G$ of prime order, where $f(q) =q^{34}$ if $S = {\rm M}_{22}$, ${\rm Ru}$ or ${\rm HS}$, otherwise $f(q)  = (q-1)q^{21}$. One checks that this gives $\mathcal{Q}(G,H)<a_1^2/b_1<q^{-1}$ for all $q$ satisfying the restrictions on $p$ in Table \ref{tab:spor}. The case $S = {\rm Th}$ is also straightforward. Here $\bar{G}=E_8$ and $p=3$, so Proposition \ref{p:long} implies that there are no long root elements in $H$. In particular, if $x \in H$ has prime order, then $|x^G|>q^{92}=b_1$ and we deduce that $\mathcal{Q}(G,H)< a_1^2/b_1< q^{-1}$, where $a_1 = |S|$. 

If $S = {\rm J}_{2}$ then the same argument reduces the problem to the cases $G_0 = E_6^{\e}(2), E_7(2)$, but by inspecting \cite{BBR,KW,Wil2} we see that none of these groups have a maximal subgroup with socle ${\rm J}_{2}$. Similarly, if $S = {\rm J}_{3}$ then we may assume $G_0 = E_6^{\e}(4)$. Let $x \in G$ be an element of prime order $r$. If $r=2$ then $|x^G|>(4-1)4^{21} = b_1$ and we note that $i_2(H) \leqs i_2(S.2) = 46683=a_1$. On the other hand, if $r>2$ then $|x^G|>(4-1)4^{31}=b_2$. Setting $a_2 = |S|$, it follows that $\mathcal{Q}(G,H)<a_1^2/b_1+a_2^2/b_2<1$. 

Finally, let us assume $S = {\rm Fi}_{22}$. Here $\bar{G}=E_6$, $p=2$ and we may assume $q \geqs 4$.  Let $x \in G$ be an element of prime order $r$. If $r=2$ then $|x^G|>(q-1)q^{21} = b_1$ and we note that $i_2(H) \leqs i_2({\rm Fi}_{22}.2) = 79466751 = a_1$. For $r>2$ we have $|x^G|>(q-1)q^{31}=b_2$ and it follows that $\mathcal{Q}(G,H)<a_1^2/b_1+a_2^2/b_2$, where $a_2 = |S|$. This yields $\mathcal{Q}(G,H)<q^{-1}$ if $q \geqs 8$.

The case $q=4$ requires special attention. First observe that $|^2E_6(4)|$ is indivisible by $11$, so $G_0 = E_6(4)$ is the only option.  Let $V$ be the adjoint module for $G_0$ and note that $S$ acts irreducibly on $V$ (see \cite[p.27]{Litt}). In fact, $V$ is the unique $78$-dimensional irreducible module for $S.2$ (over $\mathbb{F}_4$) and we can use {\sc Magma} to compute the action on $V$ of a set of conjugacy class representatives in $H$ (see \cite[Lemma 2.11]{BTh_comp} for the details).

Let $x \in H$ be an element of prime order $r$. First assume $r \in \{2,3\}$. If $r=2$ and $x \in S$, then we compute the Jordan form of $x$ on $V$ and we identify $x^G$ by inspecting \cite[Table 6]{Lawunip}. There are $3$ classes of involutions in $S.2\setminus S$ and we note that $G_0.2$ has two classes of involutory graph automorphisms, represented by $\tau$ and $\tau'$, where 
$C_{G_0}(\tau) = F_4(4)$. As in previous cases, we can identify the corresponding $G$-class of each involution in $S.2\setminus S$ by computing the Jordan form on $V$. Indeed, if $x$ has Jordan form $(J_2^{26},J_1^{26})$ then $x$ is conjugate to $\tau$, whereas the graph automorphisms in the other class have Jordan form $(J_2^{36},J_1^6)$ on $V$. Finally, if $r=3$ then $\dim C_V(x) = \dim C_{\bar{G}}(x)$ and this uniquely determines $C_{\bar{G}}(x)$. The results are summarised in Table \ref{tab:fi22}. (Here we use the notation from \cite{ATLAS} for the classes in ${\rm Fi}_{22}.2$, while the unipotent classes in $G_0$ are labelled as in \cite{LieS}. For elements of order $3$, we give the structure of $C_{\bar{G}}(x)^0$.) 

\begin{table}
\begin{center}
\[
\begin{array}{lllllll} \hline
\texttt{2A} & 3510 & A_1 & & \texttt{3A} & 3294720 & A_5T_1 \\ 
\texttt{2B} & 1216215 & A_1^2  & & \texttt{3B} & 25625600 & A_2^3 \\
\texttt{2C} & 36486450 & A_1^3 & & \texttt{3C} & 461260800  & D_4T_2 \\
\texttt{2D} & 61776 & \tau & & \texttt{3D} & 3690086400 & A_2^3 \\
\texttt{2E} & 19459440 & \tau' & & &  & \\
\texttt{2F} & 22239360 & \tau' & & & & \\ \hline
\end{array}
\] 
\caption{Elements of order $2$ and $3$ in ${\rm Fi}_{22}.2 < E_6(4).2$}
\label{tab:fi22}
\end{center}
\end{table}

In each case, it is easy to determine a lower bound on $|x^G|$ and we deduce that the contribution to $\mathcal{Q}(G,H)$ from elements of order $2$ or $3$ is less than $\sum_{i=1}^{8}a_i^2/b_i$, where
\[
a_1 = 3510,\; a_2 = 1216215, \; a_3 = 36486450,\; a_4 = 61776,
\]
\[
a_5 = 41698800,\; a_6 = 3294720,\; a_7 = 3715712000,\; a_8 = 461260800
\]
and
\[
b_1 = 3.4^{21},\; b_2 = 3.4^{31},\; b_3 = \frac{1}{2}4^{40},\; b_4 = \frac{1}{6}4^{26}, \; 
b_5 = b_6 = \frac{1}{6}4^{42},\; b_7 = \frac{1}{6}4^{54},\; b_8 = \frac{1}{6}4^{48}.
\]
Finally, if $r \geqs 5$ then we find that $\dim C_V(x) = \dim C_{\bar{G}}(x) \leqs 18$ and thus $|x^G|>\frac{1}{6}4^{60}=b_9$. By setting $a_9 = |S|$, we conclude that 
\[
\mathcal{Q}(G,H) < \sum_{i=1}^{9}a_i^2/b_i <1
\]
and thus $b(G,H) = 2$.
\end{proof}

\begin{lem}\label{l:ndef4}
The conclusion to Theorem \ref{t:typeV} holds if $S \not\in {\rm Lie}(p)$ is a simple group of Lie type. 
\end{lem}

\begin{proof}
The possibilities for $S$ are recorded in \cite[Tables 10.3 and 10.4]{LS99} and we may assume $G_0$ is one of the groups in \eqref{e:list}. In each case, set
\[
a_1 = i_2({\rm Aut}(S)), \;\; a_2 = |{\rm Aut}(S)|,\;\; b_1 = \left\{\begin{array}{ll} \ell_1 & \mbox{if $p=2$} \\ \min\{\ell_3,\ell_5\} & \mbox{if $p>2$} \end{array}\right., \; b_2 = \left\{\begin{array}{ll} \ell_4 & \mbox{if $p=2$} \\ \ell_2 & \mbox{if $p>2$} \end{array}\right.
\]
where the $\ell_i$ are defined in Table \ref{tab:cbds}. Then by applying Propositions \ref{p:bounds} and \ref{p:long}, we deduce that 
\[
\mathcal{Q}(G,H) < a_1^2/b_1 + a_2^2/b_2.
\]
It is routine to check that this upper bound is less than $1$ unless $(G_0,S)$ is one of the following:
\[
(F_4(3), {}^3D_4(2)),\; (F_4(5), {}^3D_4(2)), \;
(E_6^{\e}(3), {}^3D_4(2)),\; (E_6^{\e}(3), {}^2F_4(2)'),\; (E_6^{\e}(4), \O_7(3)).
\]

First assume $(G_0,S) = (F_4(5), {}^3D_4(2))$. Here $b_1 = 5^{16}$ and since $|H|$ is indivisible by $5$ we can take $b_2 = \ell_4 = 4.5^{29}$ as a lower bound on $|x^G|$ for all $x \in H$ of odd prime order. One checks that $a_1^2/b_1+a_2^2/b_2<1$.

Next assume $G_0 = E_6^{\e}(3)$ and $S = {}^3D_4(2)$ or ${}^2F_4(2)'$. Let $x \in H$ be an element of prime order $r$. If $r=2$ then $|x^G|>2.3^{25} = b_1$. Similarly, if $r=3$ then $|x^G|>2.3^{31} = b_2$ since $H$ does not contain long root elements by Proposition \ref{p:long}. Now assume $r \geqs 5$. Since $r$ divides $|Z(C_{G_0}(x))|$ it follows that $C_{\bar{G}}(x)^0 \ne D_5T_1$ or $A_5T_1$, whence  $|x^G|>\frac{1}{6}3^{48} = b_3$ and we conclude that
\[
\mathcal{Q}(G,H) < \sum_{i=1}^{3}a_i^2/b_i < 1,
\]
where $a_1 = i_2({\rm Aut}(S))$, $a_2 = i_3({\rm Aut}(S))$ and $a_3 = |{\rm Aut}(S)|$.

Now assume $(G_0,S) = (F_4(3), {}^3D_4(2))$, so $G = F_4(3)$ and $H = {}^3D_4(2)$ or ${}^3D_4(2).3$. Here we proceed as we did in the proof of Lemma \ref{l:ndef3} for the case $S = {\rm Fi}_{22}$ with $G_0=E_6(4)$. First, with the aid of {\sc Magma}, we observe that ${}^3D_4(2).3$ has a unique $52$-dimensional irreducible module $V$ over $\mathbb{F}_3$, which we may identify with the adjoint module for $G_0$ (as noted in \cite[Table 6.36]{Litt}, $S$ acts irreducibly on $V$). Suppose $x \in {}^3D_4(2).3$ has prime order $r$. If $r=3$, then we can compute the Jordan form of $x$ on $V$ and use \cite[Table 4]{Lawunip} to determine the $\bar{G}$-class of $x$ up to one of two possibilities. Similarly, if $r \in \{2,7,13\}$ then we can compute $\dim C_V(x) = \dim C_{\bar{G}}(x)$, which yields a lower bound on $|x^G|$. See \cite[Lemma 2.11]{BTh_comp} for further details on the computation.

For example, suppose $r=3$ and $x$ is in the $H$-class labelled \texttt{3A} in \cite{ATLAS}, so $|x^H| = 139776$. Then we calculate that $x$ has Jordan form $(J_3^{15},J_1^7)$ on $V$, which implies that $x$ is either in the $\bar{G}$-class labelled $A_2$ or $\tilde{A}_2$. In particular, $|x^G|>\frac{1}{4}3^{30}$. Similarly, if $r=7$ then $\dim C_V(x) = 10$ and we deduce that $|x^G|>\frac{1}{2}3^{42}$. 

In this way, by considering each $H$-class of prime order elements in turn, we obtain an upper bound on $\mathcal{Q}(G,H)$ which allows us to conclude that $\mathcal{Q}(G,H)<1$. We leave the reader to check this details. 

Finally, let us assume $(G_0,S) = (E_6^{\e}(4), \O_7(3))$. By Lagrange's Theorem, we see that $\e=+$ is the only possibility. Now ${\rm SO}_{7}(3)$ has a unique $78$-dimensional irreducible module $V$ over $\mathbb{F}_4$, which we identify with the adjoint module for $G_0$. We can now proceed as in the previous case, using {\sc Magma} to compute $\dim C_V(x)$ for each $x \in H$ of prime order (see \cite[Lemma 2.11]{BTh_comp}). As before, this information translates into a lower bound on $|x^G|$ and this allows us to determine an upper bound on $\mathcal{Q}(G,H)$. In this way, one checks that $\mathcal{Q}(G,H)<1$ and the result follows. 
\end{proof}

\subsection{Generic subgroups}\label{ss:def}

To complete the proof of Theorem \ref{t:typeV}, we may assume $G_0$ is one of the groups in \eqref{e:list} and $S$ is in ${\rm Lie}(p)$, as in part (ii) of Theorem \ref{t:simples}. In view of Craven's theorem (see Theorem \ref{t:craven}), there are very few possibilities for $G_0$ and $S$ and it is a straightforward exercise to verify Theorem \ref{t:typeV} in these cases.

\begin{lem}\label{l:liep}
The conclusion to Theorem \ref{t:typeV} holds if $S \in {\rm Lie}(p)$.
\end{lem}

\begin{proof}
By Theorem \ref{t:craven}, we have $G_0 = E_8(q)$ or $E_7(q)$. First assume $G_0= E_8(q)$. By inspecting the possibilities for $S$ in Theorem \ref{t:craven} and by applying Proposition \ref{p:i23}, we deduce that
\[
i_2({\rm Aut}(S)) \leqs 2s(s+1)=a_1,\;\; |{\rm Aut}(S)| \leqs |{\rm Aut}({\rm L}_{2}(3^7))| = a_2,
\]
where $s=2621$. Setting $b_1 = q^{58}$ and $b_2 = q^{92}$, we deduce that $\mathcal{Q}(G,H) < a_1^2/b_1+a_2^2/b_2<1$ and thus $b(G,H) = 2$. Similarly, if $G_0 = E_7(q)$ then $|{\rm Aut}(S)| \leqs 31200=a_1$ and we have $|x^G|>q^{34}=b_1$ for all nontrivial $x \in G$, whence $\mathcal{Q}(G,H) < a_1^2/b_1<1$ and the result follows.
\end{proof}

This completes the proof of Theorem \ref{t:typeV}.

\section{Twisted groups}\label{s:part4}

In this final section of the paper, we complete the proof of Theorem \ref{t:main} by handling the remaining almost simple primitive groups with socle 
\begin{equation}\label{e:list2}
G_0 \in \{  {}^3D_4(q), {}^2F_4(q)', {}^2G_2(q)' \, (q \geqs 27), {}^2B_2(q) \}.
\end{equation}
Our main result is the following.

\begin{thm}\label{t:twist}
If $G_0$ is one of the groups in \eqref{e:list2}, then $G$ is not extremely primitive.
\end{thm}

Let $G$ be an almost simple primitive group with point stabilizer $H$ and socle $G_0$ as in \eqref{e:list2}. Recall that we handled the special cases 
\[
G_0 \in \{{}^2B_2(8), {}^2B_2(32), {}^2F_4(2)', {}^3D_4(2) \}
\] 
in Theorem \ref{t:small}, so for the remainder of this section we will assume $G_0$ is not one of these groups. Furthermore, in view of Theorems \ref{t:parab} and \ref{t:maxrank}, we may assume that $H$ is neither a parabolic nor a maximal rank subgroup of $G$. Then by inspecting \cite{Mal} and \cite[Tables 8.16, 8.43 and 8.51]{BHR}, it follows that either
\begin{itemize}\addtolength{\itemsep}{0.2\baselineskip}
\item[{\rm (a)}] $H$ is a subfield subgroup; or
\item[{\rm (b)}] $G_0 = {}^3D_4(q)$ and $H_0 = H \cap G_0$ is either $G_2(q)$ or ${\rm PGL}_{3}^{\e}(q)$ with $q \equiv \e \imod{3}$. 
\end{itemize}

First we handle the subfield subgroups in (a).

\begin{lem}\label{l:3d4sub}
If $G_0 = {}^3D_4(q)$ and $H$ is a subfield subgroup, then $G$ is not extremely primitive.
\end{lem}

\begin{proof}
We proceed as in the proof of Proposition \ref{p:sub}. Write $q=q_0^k$, where $k \ne 3$ is a prime and set $\bar{G} = D_4$. Fix a Steinberg endomorphism $\psi = \s\tau$ of $\bar{G}$, where $\s$ is a standard Frobenius morphism of $\bar{G}$ corresponding to the map $\l \mapsto \l^{q_0}$ on $\mathbb{F}_q$ and $\tau$ is the standard triality graph automorphism of $\bar{G}$. Then $H_0 = \bar{G}_{\psi} = {}^3D_4(q_0)$ and $G_0 = \bar{G}_{\psi^k} = {}^3D_4(q)$.

Let $\a_0$ be the highest root in the root system of $\bar{G}$ and let $X_{\a}$ be the root subgroup of $\bar{G}$ corresponding to the root $\a$. Consider the $\psi$-stable maximal rank subgroup $\bar{M}  = \bar{X}\bar{Y}$ of $\bar{G}$, where $\bar{X} = \la X_{\pm \a_0} \ra$ and $\bar{Y} = C_{\bar{G}}(\bar{X})^0$. Then $\bar{M}$ is of type $A_1^4$ and we set 
\[
M = \bar{M}_{\psi} = ({\rm SL}_{2}(q_0) \circ Y).d = d.({\rm L}_{2}(q_0) \times {\rm L}_{2}(q_0^3)).d,
\]
where $Y = (\bar{Y}_{\psi})'$ and $d = (2,q-1)$. By \cite{LSS}, $M$ is a maximal subgroup of $H_0$ and we focus our attention on the subgroup $K = T \circ Y \leqs M$, where $T = Z_{q_0+1}$ is a maximal torus of ${\rm SL}_{2}(q_0)$. 

Then by arguing as in the proof of Lemma \ref{l:claim1}, using \cite{Kl3} for information on the maximal subgroups of $H_0$, we deduce that $M$ is the unique maximal overgroup of $K$ in $H_0$. Writing $G = G_0.A$ and $H = H_0.A$, we can now repeat the argument in the proof of Proposition \ref{p:sub} to show that $G$ is not extremely primitive, noting that $K$ and $M$ are $A$-stable.   
\end{proof}

\begin{lem}\label{l:twisted_sub}
Suppose $G_0 \in \{ {}^2F_4(q), {}^2G_2(q), {}^2B_2(q)\}$  and $H$ is a subfield subgroup of $G$. Then $b(G,H) = 2$.
\end{lem}

\begin{proof}
In view of \cite[Propositions 4.38 and 4.40]{BLS}, we may assume $G_0 = {}^2F_4(q)$. Let $H$ be a subfield subgroup of $G$ with $H_0 = H \cap G_0 = {}^2F_4(q_0)$, where $q = q_0^k$ for some odd prime $k$.

First assume $k \geqs 5$, so $q \geqs 32$. If $x \in G_0$ is an involution of type $u_1$ in the notation of \cite[Table II]{Shin}, then
\[
|x^G\cap H| = (q_0^3+1)(q_0^2-1)(q_0^6+1) < q^{11/5} = a_1,\;\; |x^G| > (q-1)q^{10} = b_1.
\]
For all other nontrivial elements in $G$, we have $|x^G|>(q-1)q^{13}=b_2$ and we note that $|H|<\log_2q.q^{26/5}=a_2$. It follows that $\mathcal{Q}(G,H)<a_1^2/b_1+a_2^2/b_2$ and one checks that this upper bound is less than $q^{-1}$ for all $q \geqs 32$.

Now suppose $k=3$. Let $x \in G$ be an element of prime order $r$. First assume $r=2$, so $x$ is conjugate to $u_1$ or $u_2$ in the notation of \cite{Shin}. As above, if $x = u_1$ then $|x^G \cap H|<q^{11/3}=a_1$ and $|x^G|>(q-1)q^{10}=b_1$. Similarly, if $x=u_2$ then $|x^G \cap H|<q^{14/3}=a_2$ and $|x^G|>(q-1)q^{13}=b_2$. Next assume $x$ is semisimple. Both $H_0$ and $G_0$ have a unique conjugacy class of elements of order $3$ (represented by the element $t_4$ in \cite[Table IV]{Shin}) and we get $|x^G\cap H|< q^6=a_3$ and $|x^G|>(q-1)q^{17} = b_3$. For $r \geqs 5$, we have $|x^G|>\frac{1}{3}q^{20}=b_4$ and we note that $|H_0|< q^{26/3}=a_4$. Finally, suppose $x \in G$ is a field automorphism. If $r =3$ then $|x^G|>\frac{1}{2}q^{52/3}=b_5$ and we observe that $H$ contains precisely $2(i_3({}^2F_4(q_0))+1) < 2q^6 = a_5$ field automorphisms of order $3$. And for $r \geqs 5$, we get $|x^G|>\frac{1}{2}q^{104/5} = b_6$ and we note that $|H|<\log_2q.q^{26/3}=a_6$. We conclude that 
\[
\mathcal{Q}(G,H)< \sum_{i=1}^{6}a_i^2/b_i < q^{-1}
\]
and the result follows. 
\end{proof}

Finally, let us turn to the two remaining cases with $G_0 = {}^3D_4(q)$.

\begin{lem}\label{l:3d4g2}
If $G_0 = {}^3D_4(q)$ and $H_0 = G_2(q)$, then $G$ is not extremely primitive.
\end{lem}

\begin{proof}
Here $H_0$ is the centralizer in $G_0$ of a triality graph automorphism. Therefore, if $G$ contains a graph automorphism then ${\rm soc}(H)$ will be a direct product of non-isomorphic simple groups and thus $G$ is not extremely primitive by Lemma \ref{l:structure}(v). This allows us to assume that $G = G_0.A$ and $H = H_0.A$, where $A$ is a group of field automorphisms.

Through the work of Cooperstein \cite{Coop} and Kleidman \cite{K88}, the maximal subgroups of $H_0$ are known for all $q$. In particular, we note that $H_0$ has a maximal subgroup $M = K.2$, where $K=\text{SL}_3(q)$ and the outer involution acts as a graph automorphism on $K$. Clearly, $M$ is the unique maximal overgroup of $K$ in $H_0$. By inspecting \cite{Kl3}, we see that 
\[
N_{G_0}(K) = N_{G_0}(L) = (K \circ L).(3,q^2+q+1).2
\]
is a maximal subgroup of $G_0$, where $L = Z_{q^2+q+1}$. Here the outer involution induces a graph automorphism on $K$ and inverts $L$. Write $L = \la g \ra$ and note that $L = \la g^{-2} \ra$ and $g \in C_{G_0}(K)$. Seeking a contradiction, suppose $g$ normalizes $M$ and choose an element $x \in M \setminus K$. Then $g^x = g^{-1}$, so $[g,x] = g^{-2} \in M$ and thus $L < K = M'$ since $L$ has odd order. But $K \cap L = Z(K) = Z_{(3,q-1)}$ and so we have reached a contradiction. Therefore, $g$ does not normalize $M$ and the result now follows from Lemma \ref{l:simple}, noting that $K$ and $M$ are both $A$-stable.
\end{proof}

\begin{lem}\label{l:3d4pgl3}
If $G_0 = {}^3D_4(q)$ and $H_0 = {\rm PGL}_{3}^{\e}(q)$, where $q \equiv \e \imod{3}$, then $b(G,H) = 2$.
\end{lem}

\begin{proof}
Here $H_0 = C_{G_0}(\tau)$, where $\tau$ is a triality graph automorphism of $G_0$. Since $q \ne 3$ and we handled the case $q=2$ in Theorem \ref{t:small}, we may assume $q \geqs 4$. We claim that $\mathcal{Q}(G,H)<1$. 

Let $x \in H$ be an element of prime order $r$ and write $q=p^f$ with $p$ prime. Let $U$ and $V$ be the natural modules for $H_0$ and $\bar{G}=D_4$, respectively, where $\bar{G}$ is the ambient simple algebraic group. Note that the embedding of $H_0$ in $G_0$ arises from the embedding of $H_0$ in ${\rm P\O}_{8}^{+}(q^3)$ through the action of $H_0$ on its adjoint module. In particular, we can work with the adjoint representation to determine the Jordan form on $V$ for each $x \in H_0$.

First assume $x \in H_0$ and $r=p$. If $p=2$ then $H_0$ has a unique class of involutions and we calculate that $x$ has Jordan form $(J_2^4)$ on $V$, so $x$ is contained in the $G_0$-class labelled $3A_1$ in \cite[Section 0.5]{Spal}. In particular,
\[
|x^G \cap H| = i_2(H_0) = (q+\e)(q^3-\e)=a_1,\;\; |x^G| = q^2(q^6-1)(q^8+q^4+1) = b_1.
\]
Now assume $p \ne 2$ (so $p \geqs 5$). If $x$ has Jordan form $(J_2,J_1)$ on $U$, then $x$ acts on $V$ as $(J_3,J_2^2,J_1)$, which implies that $x$ is in the $G_0$-class $3A_1$. Similarly, if $x$ has Jordan form $(J_3)$ on $U$, then it acts as $(J_5,J_3)$ on $V$, which places $x$ in the $G_0$-class labelled $D_4(a_1)$. In both cases, this allows us to compute $|x^G \cap H|$ and $|x^G|$ precisely and we conclude that the total contribution to $\mathcal{Q}(G,H)$ from unipotent elements is at most $a_1^2/b_1+a_2^2/b_2$, where
\[
a_2 = q(q^2-1)(q^3-\e),\;\; b_2 = q^6(q^2-1)(q^6-1)(q^8+q^4+1).
\]

Next assume $x \in H_0$ is semisimple. If $r=2$ then 
\[
|x^G \cap H| = i_2(H_0) = q^2(q^2+\e q+1) = a_3,\;\; |x^G| = i_2(G_0) = q^8(q^8+q^4+1)=b_3.
\]
If $r>2$, then $|x^G|>(q-1)q^{17} = b_4$ and we note that $|H_0|<q^8 = a_4$.

Next suppose $x \in G$ is a graph automorphism. Here $|x^G|>q^{14}=b_5$ and the total number of graph automorphisms in $H$ is equal to 
\[
2(i_3({\rm PGL}_{3}^{\e}(q))+1) = 4q^2(q^4+2q^2+3\e q+2) +2= a_5.
\]
Finally, suppose $x \in G$ is a field automorphism of order $r$. If $r=2$, then 
\[
|x^G \cap H| \leqs \frac{|{\rm PGU}_{3}(q)|}{|{\rm SL}_{2}(q)|} = q^2(q^3+1) = a_6,\;\; |x^G|>\frac{1}{2}q^{14} = b_6.
\]
Similarly, if $r \geqs 3$ then $|x^G|>\frac{1}{2}q^{56/3} = b_7$ and we note that $H$ contains fewer than $\log_2q.q^8 = a_7$ field automorphisms of $G_0$.

Putting these estimates together, we conclude that
\[
\mathcal{Q}(G,H) < \sum_{i=1}^{7}a_i^2/b_i,
\]
which is less than $1$ for all $q \geqs 4$ (and it is less than $q^{-1}$ for $q \geqs 11$).
\end{proof}

This completes the proof of Theorem \ref{t:twist}. By combining this result with Theorems \ref{t:small}, \ref{t:parab}, \ref{t:maxrank}, \ref{t:34}, \ref{t:typeI}, \ref{t:sub} and \ref{t:typeV}, we conclude that the proof of Theorem \ref{t:main} is complete.

\end{document}